\documentclass[reqno, 12pt]{amsart}

\usepackage{array}
\usepackage{amsmath}
\usepackage{amsfonts}
\usepackage{amssymb}
\usepackage{enumerate}
\usepackage{amsthm}
\usepackage{amsmath, amscd}
\usepackage{xy}
\xyoption{all}
\usepackage{supertabular}
\usepackage{marvosym}
\usepackage{wasysym}
\usepackage{hyperref}
\usepackage{tikz}
\usetikzlibrary{matrix,arrows,backgrounds}
	
\newtheorem{theorem}{Theorem}[section]

\newtheorem{lemma}[theorem]{Lemma}
\newtheorem{proposition}[theorem]{Proposition}
\newtheorem{corollary}[theorem]{Corollary}

\newtheorem{conjecture}[theorem]{Conjecture}

\theoremstyle{definition}
\newtheorem{definition}[theorem]{Definition}
\newtheorem{definition-lemma}[theorem]{Definition/Lemma}
\newtheorem{remark}[theorem]{Remark}

\newtheorem{example}[theorem]{Example}

\newcommand{\op}[1]{\operatorname{#1}}

\newcommand{\leftexp}[2]{{\vphantom{#2}}^{#1}{#2}}

\newcommand{\dbcoh}[1]{\operatorname{D}^{\operatorname{b}}(\operatorname{coh }#1)}

\newcommand{\inj}[1]{\mathsf{Inj}_{\operatorname{coh}}(#1)}
\newcommand{\dbmod}[1]{\operatorname{D}^{\operatorname{b}}(\operatorname{mod }#1)}

\newcommand{\proj}[3]{\mathsf{proj}(#1,#2,#3)}
\newcommand{\barproj}[3]{\overline{\mathsf{proj}}(#1,#2,#3)}

\newcommand{\dabs}{\op{D}^{\op{abs}}}

\newcommand{\newterm}{\textsf}

\def\Z{\mathbb{Z}}
\def\C{\mathbb{C}}

\def\O{\mathcal{O}}

\def\P{\mathbb{P}}

\def\smallfermatgroup{\bar{G}(B_\mathbf{d})}
\def\intermediatefermatgroup{H}

\def\hodgcat{\op{Ho}(\op{dg-cat}_k)}

\def\tritime{\text{\Clocklogo}}

\def\tor{\text{ or } }
\def\tand{\text{ and } }

\def\ok{\otimes_k}

\title[Kernels for equivariant factorizations, II]{A category of kernels for equivariant factorizations, II: further implications}

\author[Ballard]{Matthew Ballard}
\address{
  \begin{tabular}{l}
   Matthew Ballard  \\
   \hspace{.1in} University of South Carolina, Department of Mathematics, Columbia, SC, USA \\
   \hspace{.1in} Email: {\bf ballard@math.sc.edu} \\
  \end{tabular}
}

\author[Favero]{David Favero}
\address{
  \begin{tabular}{l}
   David Favero \\
   \hspace{.1in} University of Alberta, Department of Mathematics,  Edmonton, AB, Canada \\
   \hspace{.1in} Email: {\bf favero@gmail.com} \\
  \end{tabular}
}

\author[Katzarkov]{Ludmil Katzarkov}
\address{
  \begin{tabular}{l}
   Ludmil Katzarkov \\
   \hspace{.1in} Universit\"at von Wien, Fakult\"at f\"ur Mathematik,  Wien, \"Osterreich \\
   \hspace{.1in} Email: {\bf lkatzark@math.uci.edu} \\
  \end{tabular}
}

\numberwithin{equation}{section}
\begin{document}
\renewcommand{\labelenumi}{\emph{\alph{enumi})}}

\begin{abstract}
 We leverage the results of the prequel, \cite{BFK11}, in combination with a theorem of D.\ Orlov to create a categorical covering picture for factorizations.   As applications, we provide a conjectural geometric framework to further understand M. Kontsevich's Homological Mirror Symmetry conjecture and obtain new cases of a conjecture of Orlov concerning the Rouquier dimension of the bounded derived category of coherent sheaves on a smooth variety.
\end{abstract}

\maketitle

\section{Introduction}

The homotopy category of dg-categories comes equipped with a natural commutative and associative product, which we term the Morita product. Given two dg-categories, $\mathsf C$ and $\mathsf D$, the Morita product $\mathsf C \circledast \mathsf D$, is any choice of dg-enhancement, contained in $(\mathsf C \otimes_k \mathsf D)^{\op{op}}\op{-Mod}$, of the derived category of perfect $\mathsf{C} \otimes_k \mathsf{D}$-modules, $\op{perf}( \mathsf{C} \otimes_k \mathsf{D} )$.

The product is a natural one in many settings. For a variety, $Z$, let $\mathsf{Inj}_{\op{coh}}(Z)$ denote the dg-category of bounded below complexes of injectives with bounded and coherent cohomology. B. T\"oen proved \cite[Section 8]{Toe}, see also \cite{BFN}, that there is a quasi-equivalence
\begin{equation} \label{eq: Toen product}
 \inj{X} \circledast \inj{Y} \simeq \inj{X \times Y}
\end{equation}
whenever $X$ and $Y$ are smooth and proper varieties. Thus, the monoidal product in the Grothendieck group of smooth varieties is naturally lifted by the Morita product. As correspondences arise as cohomology classes on $X \times Y$, the Morita product of $X$ and $Y$ is an enriched setting to study correspondences and other Hodge theoretic questions.

However, there are many more dg-categories than those coming from varieties. One strictly larger class is that of equivariant factorizations. Let $X$ be a smooth variety equipped with an action of an affine algebraic group $G$. Let $w \in \Gamma(X,\mathcal O_X(\chi))^G$ be a semi-invariant regular function on $X$. Then, one can form a dg-category of factorizations mimicking $\inj{Z}$. Let $\mathsf{Inj}_{\op{coh}}(X,G,w)$ be the dg-category of $G$-equivariant factorizations of $w$ with injective components and such that each is quasi-isomorphic to a factorization with coherent components in L. Positselski's (absolute) derived category of factorizations, $\dabs[\mathsf{Fact}(X,G,w)]$ \cite{Pos1}. This dg-category is an enhancement of the derived category of coherent factorizations,
\begin{displaymath}
 [\mathsf{Inj}_{\op{coh}}(X,G,w)] \cong \dabs[\mathsf{fact}(X,G,w)].
\end{displaymath}

In the prequel to this article \cite{BFK11}, the authors determined Morita products and internal Hom dg-categories for such dg-categories of equivariant factorizations. Let $(X,G,w, \mathcal L, \chi)$ be as above and let $(Y,H,v, \mathcal L^{\prime}, \chi^{\prime})$ be another such five-tuple. Define
\begin{displaymath}
 w \boxplus v := w \otimes_k 1 + 1 \otimes_k v \in \Gamma(X, \mathcal O_X) \otimes_k \Gamma(Y, \mathcal O_Y) \cong \Gamma(X \times Y, \mathcal O_{X \times Y}).
\end{displaymath}
The potential $w \boxplus v$ is not fully $G \times H$ semi-invariant. The largest, natural, group for which it is semi-invariant is
\begin{displaymath}
 G \times_{\mathbb{G}_m} H : = \op{ker}(\chi - \chi^{\prime}: G \times H \to \mathbb{G}_m).
\end{displaymath}

Let $Z_w, Z_v,$ and $Z_{w \boxplus v}$ denote the zero loci of $w, v$, and $w \boxplus v$ respectively. The following is one of the main results of \cite{BFK11}.

\begin{theorem} \label{theorem: morita product of equivariant factorizations}
 Assume that $\chi - \chi^{\prime}: G \times H \to \mathbb{G}_m$ is not torsion. Further, assume that $\op{Sing} Z_{w \boxplus v} \subset Z_w \times Z_v$. Then, there is an isomorphism in the homotopy category of dg-categories
 \begin{displaymath}
  \inj{X,G,w} \circledast \inj{Y,H,v} \simeq \overline{\mathsf{Inj}}_{\op{coh}}(X \times Y, G \times_{\mathbb{G}_m} H, w \boxplus v)
 \end{displaymath}
 where $\overline{\mathsf{Inj}}_{\op{coh}}(X \times Y, G \times_{\mathbb{G}_m} H, w \boxplus v)$ is a dg-idempotent completion of $\mathsf{Inj}_{\op{coh}}(X \times Y, G \times_{\mathbb{G}_m} H, w \boxplus v)$.
\end{theorem}

This built on work of T. Dyckerhoff for dg-categories of non-equivariant factorizations \cite{Dyc}. Additionally, a special case of Theorem~\ref{theorem: morita product of equivariant factorizations} where $X, Y$ are affine spaces $G, H$ are finite extensions of $\mathbb{G}_m$, and $w, v$ have isolated singular loci was proven independently by A. Polishchuk and A. Vaintrob \cite{PVnew}.

In this paper, we leverage our knowledge of these Morita products to create a categorical covering picture for factorizations. We apply our results to Homological Mirror Symmetry and Orlov spectra. After setting up the necessary background in Section~\ref{sec: background}, the primary observation, appearing in Section~\ref{sec: Orlov}, is that Orlov's relationship between graded categories of singularities and derived categories of coherent sheaves \cite{Orl09} has some very interesting geometric consequences when combined with Theorem~\ref{theorem: morita product of equivariant factorizations}.  

Namely, for  $1 \leq i \leq t$ let $X_i$ be affine spaces, $G_i$ be abelian groups acting linearly on $X_i$ with positive weights, and $w_i$ be sections of the $G_i$-equivariant line bundles $\O(\chi_i)$ corresponding to characters $\chi_i : G_i \to \mathbb G_m$.  The sum 
$w_1 + ... + w_t$ is naturally a $G$-equivariant section of $\O_{X_1 \times ... \times X_t}(\chi) := \O_{X_1 \times ... \times X_t}(\chi_i)$ with
\[
G  := G_1 \times_{\mathbb G_m} ... \times_{\mathbb G_m} G_t.
\]
Suppose $G \cong \mathbb G_m \times \Gamma$ where $\Gamma$ is a finite group and fix a splitting $s: \mathbb G_m \to G$.  From this we construct the stack $Z$ which is the quotient of the zero locus of the sum of the $w_i$ in $[X_1 \times ... \times X_n /G]$.

\begin{corollary}
Suppose that the weights of the $\mathbb G_m$ action on $X_1 \times ... \times X_t$ are all positive.  There is a parameter $\mu$ and a factorization denoted by $k^{\otimes t}$ such that

 \renewcommand{\labelenumi}{\emph{\roman{enumi})}}
 \begin{enumerate}
  \item If $\mu > 0$, there is a semi-orthogonal decomposition,
  \begin{displaymath}
  \mathsf{Inj}_{\op{coh}} (Z) \simeq \left \langle \bigoplus_{ \chi \circ s =  t^{-\mu}} \O_Z(\chi ), \ldots , \bigoplus_{{ \chi \circ s =  t^{-1}}} \O_Z(\chi), \mathsf{Inj}_{\op{coh}}(X_1,G_1,w_1) \circledast \cdots \circledast \mathsf{Inj}_{\op{coh}}(X_t, G_t,w_t) \right \rangle.
  \end{displaymath}
  \item If $\mu = 0$, there is a quasi equivalence of dg-categories,
  \begin{displaymath}
   \mathsf{Inj}_{\op{coh}} (Z) \simeq  \mathsf{Inj}_{\op{coh}}(X_1,G_1,w_1) \circledast \cdots \circledast \mathsf{Inj}_{\op{coh}}(X_t, G_t,w_t).
  \end{displaymath}
  \item If $\mu < 0$, there is a semi-orthogonal decomposition,
  \begin{displaymath}
   \mathsf{Inj}_{\op{coh}}(X_1,G_1,w_1) \circledast \cdots \circledast \mathsf{Inj}_{\op{coh}}(X_t, G_t,w_t) \simeq \left\langle \bigoplus_{ \pi \circ s =  t^{-\mu-1}}k^{\otimes t}(\chi),\ldots,\bigoplus_{ \pi \circ s =  \op{Id} }k^{\otimes t}(\chi),\mathsf{Inj}_{\op{coh}}(Z) \right\rangle.
  \end{displaymath}
 \end{enumerate}
\end{corollary}

This produces an enormous web of interrelated categories.  For example, as an application, we have the following result which is derived from
a form of graded Kn\"orrer periodicity.  The idea is an extension of the conic bundle construction, first observed by the third named author as a means of understanding mirror symmetry for varieties of general type.
\begin{proposition}
Let $Z_1, \ldots, Z_t$ be hypersurfaces of general type.  There is a smooth Fano Deligne-Mumford stack $Z$ such that
$\mathsf{Inj}_{\op{coh}}(Z_1 \times \cdots  \times Z_t)$ is an admissible subcategory of $\mathsf{Inj}_{\op{coh}}(Z)$.
\end{proposition}
Using the above proposition, one can try to reduce mirror symmetry for general type varieties to the better understood case of mirror symmetry for Fano varieties \cite{Aur}.  Let us state the precise conjecture, which also extends nicely to products as discussed in Section~\ref{sec: graded knorrer}.

\begin{conjecture} \label{introconj: general}
 Suppose $Z \subseteq \P^{n-1}$ is a smooth hypersurface of general type defined by a polynomial $w$ of degree $d$.  Let the group,
\begin{displaymath}
 G(A) := \{(\alpha, \beta, \gamma) \mid \alpha^d=\beta^2 = \gamma^2\} \subseteq \mathbb{G}_m^3,
\end{displaymath}
 act on $k[x_1, \ldots, x_n, u,v]$ by multiplication by $\alpha$ on the $x_i$, multiplication by $\beta$ on $u$, and multiplication by $\gamma$ on $v$.  Denote by $Y$ the affine variety defined by the vanishing of $w+u^2+v^2$.

 Then, the mirror to $Z$ can be realized as the Landau-Ginzburg mirror to the global quotient stack $[(Y \setminus 0) / G(A)]$ with all singular fibers away from the origin removed.
\end{conjecture}

Aside from mirror symmetry, our results conjecturally have nontrivial implications towards decompositions of motives, see \cite{Orl05}. We also believe that in the case of a curve, graded Kn\"orrer periodicity should be related to the classical Prym variety construction of D. Mumford \cite{Mum} for Jacobians of conic bundles by work of A. Kuznetsov in \cite{Kuz05, Kuz09a}.  Building on the seminal work of C. Voisin, see \cite{Voison}, this picture is also readily applicable to the study of Griffiths groups, see \cite{FIK}.

In the context of homological algebra and triangulated categories, the precise relationship between the quotients arising in the HMS picture above can be expressed in the language of orbit categories. Inspired further by the ideas of B. Keller, D. Murfet, and M. Van den Bergh in \cite{KMV}, we provide this description in Section~\ref{sec:orbit}.

This categorical/geometric picture is readily applicable to the notions of Rouquier dimension, Orlov spectra, and generation time - an impetus of this work. Roughly, given a strong generator, $G$, of a triangulated category, the generation time is the number of exact triangles necessary to reconstruct the category from $G$ after closing under sums, shifts, and summands.  The Orlov spectrum of a triangulated category is the set of all such generation times, and the Rouquier dimension is the minimum of the Orlov spectrum.  Orlov has conjectured that this categorical notion of dimension coincides with the geometric one in \cite{O4}, where he proves the following for smooth algebraic curves.

\begin{conjecture}[Orlov] \label{conj: dimension}
 Let $\mathcal X$ be a smooth and tame Deligne-Mumford stack. Then, the Krull dimension of $\mathcal X$ equals the Rouquier dimension of $\mathcal X$:
 \begin{displaymath}
  \op{dim} \mathcal X = \op{rdim} \dbcoh{\mathcal X}.
 \end{displaymath}
\end{conjecture}

\noindent Conjecture~\ref{conj: dimension} seems like a difficult question in general, see \cite{BF} for a discussion of known examples. Section~\ref{sec: gt} provides an overview of these notions and related questions and works out the case of weighted Fermat hypersurfaces in full detail. As a consequence, we obtain various special cases of the above conjecture by considering the relationship between Orlov spectra of triangulated categories and their orbit categories. In particular, we prove the first case of Orlov's conjecture for a non-rational projective variety of dimension $> 1$. Let us include one example from Section \ref{subsec: Fermat},  see Theorem~\ref{thm: Fermat bound} and Example~\ref{eg: ExExK3}.
\begin{theorem}\label{thm: rdim ExE}
 The Rouquier dimension of the product of two Fermat elliptic curves with the Fermat $K3$ surface is $4$.
\end{theorem}

\vspace{2.5mm}
\noindent \textbf{Acknowledgments:}
 The authors are greatly appreciative of the valuable insight gained from conversations and correspondence with Mohammed Abouzaid, Denis Auroux, Andrei C\u{a}ld\u{a}raru, Dragos Deliu, Colin Diemer, Tobias Dyckerhoff, Manfred Herbst, M. Umut Isik, Gabriel Kerr, Maxim Kontsevich, Jacob Lewis, Dmitri Orlov, Pranav Pandit, Tony Pantev, Anatoly Preygel, Victor Przyjalkowski, Ed Segal, Paul Seidel, and Paolo Stellari and would like to thank them for their time and patience.  Furthermore, the authors are deeply grateful to Alexander Polishchuk and Arkady Vaintrob for providing us with a preliminary version of their work \cite{PVnew} and for allowing us time to prepare the original version of this paper in order to synchronize arXiv posting of the articles due to the overlap. The second named author would also like to thank the Korean Institute for Advanced Study for its hospitality during final preparation of this document.  The first named author was funded by NSF DMS 0636606 RTG, NSF DMS 0838210 RTG, and NSF DMS 0854977 FRG. The second and third named authors were funded by NSF DMS 0854977 FRG, NSF DMS 0600800, NSF DMS 0652633 FRG, NSF DMS 0854977, NSF DMS 0901330, FWF P 24572 N25, by FWF P20778 and by an ERC Grant.
\vspace{2.5mm}

\section{Background} \label{sec: background}

Let us recall some background and results from \cite{BFK11}. Throughout this article, we will only encounter $G$-equivariant factorizations on affine varieties with $G$ an Abelian affine algebraic group. This setting allows for a pleasant simplification and more algebraic reformulation of much of \cite{BFK11}. We take the more algebraic perspective of graded modules instead of equivariant sheaves. For the convenience of the reader, we first provide some details on the translation between the two languages. For the whole of the paper, $k$ will denote an algebraically-closed field of characteristic zero.

\subsection{Group actions and graded rings}

Let $X$ be an affine scheme and $G$ be an Abelian affine algebraic group acting on $X$. We denote the category of quasi-coherent $G$-equivariant sheaves on $X$ by $\op{Qcoh}_G X$. Let $[X/G]$ denote the associated quotient stack. It is well-known, see e.g. \cite{Vis}, that there is an equivalence of Abelian categories
\begin{displaymath}
 \op{Qcoh}_G X \cong \op{Qcoh} [X/G].
\end{displaymath}
We shall often use this equivalence and its descendants implicitly.

\begin{definition}
 Given an Abelian affine algebraic group, $G$, we let
 \begin{displaymath}
  \widehat{G} := \op{Hom}_{\op{alg \ grp}}(G,\mathbb{G}_m).
 \end{displaymath}
 The finitely-generated Abelian group, $\widehat{G}$, is called the \newterm{group of characters} of $G$. As $\widehat{G}$ is Abelian, we shall use additive notation for the group structure on $\widehat{G}$.

 For a character, $\chi \in \widehat{G}$, we also get an autoequivalence
 \begin{align*}
  (\chi) : \op{Qcoh}_G X & \to \op{Qcoh}_G X \\
  \mathcal E & \mapsto \mathcal E \otimes_{\mathcal O_X} p^*\mathcal L_{\chi}
 \end{align*}
 where $p: X \to \op{Spec} k$ is the structure map and $\mathcal L_{\chi}$ is the object of $\op{Qcoh}_G (\op{Spec} k)$ corresponding to $\chi$.
\end{definition}

Let us write $X = \op{Spec} R$. The action of $G$ on $X$, $\sigma: G \times X \to X$, corresponds to a ring homomorphism,
\begin{displaymath}
 \Delta: R \to \Gamma(G,\mathcal O_G) \otimes_k R,
\end{displaymath}
after taking global sections. Since we have assumed that $G$ is Abelian, we have an isomorphism
\begin{displaymath}
 \Gamma(G,\mathcal O_G) \cong k[\widehat{G}]
\end{displaymath}
where $k[\widehat{G}]$ is the monoid algebra of $\widehat{G}$, i.e. elements are finite formal sums
\begin{displaymath}
 \sum \alpha_{\chi} u^{\chi}
\end{displaymath}
with componentwise addition and multiplication given by $u^{\chi} \cdot u^{\chi^{\prime}} = u^{\chi + \chi^{\prime}}$. We can decompose $R$ as
\begin{displaymath}
 R = \bigoplus_{\chi} R_{\chi}
\end{displaymath}
where
\begin{displaymath}
 R_{\chi} := \{ r \in R \mid \Delta(r) = u^{\chi} \otimes r \}.
\end{displaymath}
Since $\Delta$ is a ring homomorphism, we have
\begin{displaymath}
 R_{\chi} \cdot R_{\chi^{\prime}} \subseteq R_{\chi + \chi^{\prime}}
\end{displaymath}
and $R$ becomes a $\widehat{G}$-graded ring.

The following is very well-known. We recall it and provide a proof for the ease of the reader.

\begin{proposition} \label{proposition: Qcoh-G = Mod-hat-G}
 Let $X = \op{Spec} R$ be an affine variety acted on by an Abelian affine algebraic group $G$. There exists an equivalence of Abelian categories
 \begin{displaymath}
  \Gamma: \op{Qcoh}_G X \to \op{Mod}_{\widehat{G}} R
 \end{displaymath}
 induced by taking global sections. Moreover, this restricts to equivalences
 \begin{align*}
  \Gamma & : \op{coh}_G X \to \op{mod}_{\widehat{G}} R \\
  \Gamma & : \op{Vect}_G X \to \op{Proj}_{\widehat{G}} R
 \end{align*}
 between the Abelian category of coherent $G$-equivariant sheaves on $X$ and the Abelian category of finitely-generated $\widehat{G}$-graded modules over $R$, and between the category of locally-free $G$-equivariant sheaves on $X$ and the category of $\widehat{G}$-graded $R$-modules which are projective as $R$-modules.
\end{proposition}

\begin{proof}
 Forgetting the $G$-action, it is well-known that $\Gamma$ has the sheafification functor,
 \begin{displaymath}
  \widetilde{(\bullet)} : \op{Mod} R \to \op{Qcoh} X,
 \end{displaymath}
 as its inverse. In addition, $\widetilde{(\bullet)}$ provides an inverse to
 \begin{align*}
  \Gamma & : \op{coh} X \to \op{mod} R \\
  \Gamma & : \op{Vect} X \to \op{Proj} R.
 \end{align*}
 Let $\pi: G \times X \to X$ be the projection. We must check that equipping a quasi-coherent sheaf $\mathcal E$ with an equivariant structure,
 \begin{displaymath}
  \theta: \sigma^* \mathcal E \overset{\sim}{\to} \pi^* \mathcal E,
 \end{displaymath}
 is equivalent to equipping $\Gamma(X, \mathcal E)$ with a $\widehat{G}$-grading compatible with the action of $R$.

 Set $M := \Gamma(X, \mathcal E)$. The equivariant structure, $\theta$, is equivalent to an isomorphism of $k[\widehat{G}] \otimes_k R$-modules
 \begin{displaymath}
  \vartheta: (k[\widehat{G}] \otimes_k R) \otimes_{R} M \overset{\sim}{\to} (k[\widehat{G}] \otimes_k R) \otimes_R M
 \end{displaymath}
 where on the left-hand side we use the copy of $R$ under the homomorphism $\Delta: R \to k[\widehat{G}] \otimes_k R$ corresponding to $\Gamma(\sigma^*)$ while on the right-hand side we use the copy of $R$ under the homomorphism
 \begin{displaymath}
  r \mapsto 1 \otimes r
 \end{displaymath}
 corresponding to $\Gamma(\pi^*)$. The isomorphism $\vartheta$ is determined by where it sends $M$ and, on $M$, it must satisfy
 \begin{equation} \label{equation: vartheta}
  \vartheta(rm) = \Delta(r) \vartheta(m).
 \end{equation}
 Furthermore, any morphism $\vartheta: M \to k[\widehat{G}] \otimes M$ satisfying Equation \eqref{equation: vartheta} extends to give a desired isomorphism of $k[\widehat{G}] \otimes_k R$-modules.

 Assume we have an equivariant structure, $\theta$, and $\vartheta = \Gamma(\theta)$. We set
 \begin{displaymath}
  M_{\chi} := \{ a \in A \mid \vartheta(m) = u^{\chi} \otimes m\}.
 \end{displaymath}
 Since we satisfy Equation \eqref{equation: vartheta}, we have
 \begin{displaymath}
  R_{\chi} \cdot M_{\chi^{\prime}} \subseteq M_{\chi+\chi^{\prime}}
 \end{displaymath}
 and $M$ now has the structure of a $\widehat{G}$-graded module.

 Now assume $M$ has the structure of a $\widehat{G}$-graded $R$-module. We set
 \begin{displaymath}
  \vartheta(m) := u^{\chi} \otimes m
 \end{displaymath}
 if $m \in M_{\chi}$. It is straightforward to see that $\vartheta$ satisfies Equation \eqref{equation: vartheta}. Thus, the data of an equivariant structure on $\mathcal E$ is equivalent to a $\widehat{G}$-grading on $M$ making $M$ into a $\widehat{G}$-graded $R$-module. As such, $\Gamma$ and $\widetilde{(\bullet)}$ induce mutually inverse equivalences
 \begin{align*}
  \Gamma: \op{Qcoh}_G X & \leftrightarrow \op{Mod}_{\widehat{G}} R : \widetilde{(\bullet)} \\
  \Gamma: \op{coh}_G X & \leftrightarrow \op{mod}_{\widehat{G}} R : \widetilde{(\bullet)} \\
  \Gamma: \op{vect}_G X & \leftrightarrow \op{Proj}_{\widehat{G}} R : \widetilde{(\bullet)}
 \end{align*}
 as desired.
\end{proof}

\begin{remark}
 Given a finitely-generated Abelian group $A$, we can form an Abelian algebraic group $G(A) := \op{Spec} k[A]$. If $R$ is an $A$-graded ring, then $G(A)$ acts on $\op{Spec} R$. Thus, one can approach the Abelian categories above, and the factorizations categories below, either from the geometric perspective of group actions or from the algebraic perspective of graded rings.
\end{remark}

\begin{lemma} \label{lemma: proj equiv = equiv + proj}
 Let $A$ be a finitely-generated Abelian group and let $R$ be an $A$-graded commutative $k$-algebra finitely-generated over $k$. Then, projective $A$-graded modules are exactly those $A$-graded modules that are projective as ungraded $R$-modules. Furthermore, an injective $R$-module that possesses an $A$-grading is an injective $A$-graded module.
\end{lemma}

\begin{proof}
 Set $X := \op{Spec} R$. By Proposition~\ref{proposition: Qcoh-G = Mod-hat-G}, we have an equivalence
 \begin{displaymath}
  \op{Mod}_A R \cong \op{Qcoh}_{G(A)} X
 \end{displaymath}
 which takes $A$-graded projective $R$-modules to locally-free $G(A)$-equivariant sheaves.

 We have by definition of morphisms
 \[
 \op{Hom}_{\op{Qcoh}_{G(A)} X} (\bullet, \bullet) := \op{Hom}_{\op{Qcoh} X} (\bullet, \bullet)^{G(A)}.
 \]
 Since $G(A)$ is an Abelian affine algebraic group, it is the product of a torus and a finite group, and hence is reductive.  As $k$ has characteristic zero, $G(A)$ is linearly reductive, i.e.\ the functor of taking $G(A)$ invariants is exact.

 Therefore, if $\mathcal E$ is locally-free, so $\op{Hom}_{\op{Qcoh} X} (\mathcal E, \bullet)$ is exact, then $\op{Hom}_{\op{Qcoh}_{G(A)} X} (\mathcal E, \bullet)$ is exact and $\mathcal E$ is projective as a quasi-coherent $G(A)$-equivariant sheaf. If $\mathcal E$ is a projective quasi-coherent $G(A)$-equivariant sheaf, then it is an equivariant summand of $\mathcal O_X^{\oplus I}$ making it projective as an non-equivariant quasi-coherent sheaf.

 Similarly, if $\mathcal E$ is an injective quasi-coherent sheaf with an equivariant structure, then $\op{Hom}_{\op{Qcoh} X}(\bullet, \mathcal E)$ is exact and consequently so is $\op{Hom}_{\op{Qcoh}_{G(A)} X}(\bullet, \mathcal E)$. Hence, $\mathcal E$ is injective as a quasi-coherent $G(A)$-equivariant sheaf.
\end{proof}

As before, let $X = \op{Spec} R$ be an affine variety and let $G$ be an Abelian affine algebraic group acting on $X$. Let $H$ be a closed subgroup of $G$. Operations of fundamental import to the applications in this paper are restriction and induction of equivariant sheaves along the inclusion of $H$ in $G$.

\begin{definition}
 We have a \newterm{restriction functor}
 \begin{align*}
  \op{Res}^G_H : \op{Qcoh}_G X & \to \op{Qcoh}_H X \\
  (\mathcal E, \theta) & \mapsto (\mathcal E, \theta|_{H \times X}).
 \end{align*}
 Denote the functor
 \begin{displaymath}
  \Gamma(\op{Res}^G_H) : \op{Mod}_{\widehat{G}} R \to \op{Mod}_{\widehat{H}} R
 \end{displaymath}
 by $\op{R}^{\widehat{G}}_{\widehat{H}}$. Furthermore, if $R$ is a finitely-generated $k$-algebra equipped with a grading by a finitely-generated Abelian group $A$ and we have an epimorphism $A \to B$, then we set
 \begin{displaymath}
  \op{R}^{A}_{B} := \Gamma(\op{Res}^{G(A)}_{G(B)}).
 \end{displaymath}

 Recall, from \cite[Lemma 2.13 and Lemma 2.16]{BFK11}, that the inclusion $i: X \to G/H \times X$, along the identity coset, induces an equivalence
 \begin{displaymath}
  i^* : \op{Qcoh}_H X \to \op{Qcoh}_G G/H \times X.
 \end{displaymath}
 Let $p: G/H \times X \to X$ denote the projection onto $X$. The \newterm{induction functor} is
 \begin{align*}
  \op{Ind}^G_H: \op{Qcoh}_H X & \to \op{Qcoh}_G X \\
  (\mathcal E, \theta) & \mapsto p_* (i^*)^{-1}( \mathcal E, \theta).
 \end{align*}
 Denote the functor
 \begin{displaymath}
  \Gamma(\op{Ind}^G_H) : \op{Mod}_{\widehat{H}} R \to \op{Mod}_{\widehat{G}} R
 \end{displaymath}
 by $\op{I}^{\widehat{G}}_{\widehat{H}}$. Furthermore, if $R$ is a finitely-generated $k$-algebra equipped with a grading by a finitely-generated Abelian group $A$ and we have an epimorphism $A \to B$, then we set
 \begin{displaymath}
  \op{I}^{A}_{B} := \Gamma(\op{Ind}^{G(A)}_{G(B)}).
 \end{displaymath}
\end{definition}

\begin{proposition} \label{proposition: restriction and induction properties}
 Assume $R$ is a finitely-generated $k$-algebra equipped with a grading by a finitely-generated Abelian group $A$ and we have an epimorphism $\psi: A \to B$. Also, equip $R$ with the induced $B$-grading
 \begin{displaymath}
  R_b := \bigoplus_{\psi(a) = b} R_a.
 \end{displaymath}
 Then,
 \begin{itemize}
  \item We have
  \begin{displaymath}
   (\op{R}^{A}_{B}M)_b = \bigoplus_{\psi(a) = b} M_a.
  \end{displaymath}
  \item We have
  \begin{displaymath}
   (\op{I}^A_B N)_a = N_{\psi(a)}.
  \end{displaymath}
  The $R$-module structure on $\op{I}^A_B N$ is as follows. For $r \in R_a$ and $n \in (\op{I}^A_B N)_{a^{\prime}}$,
  \begin{displaymath}
   r \cdot n  = rn \in  N_{\psi(a+a^{\prime})} = (\op{I}^A_B N)_{a+a^{\prime}}.
  \end{displaymath}
  \item Both $\op{R}^{A}_{B}$ and $\op{I}^{A}_{B}$ are exact functors.
  \item The functor, $\op{R}^{A}_{B}$, is left adjoint to $\op{I}^{A}_{B}$ so we write
  \begin{displaymath}
   \op{R}^{A}_{B} \dashv \op{I}^{A}_{B}.
  \end{displaymath}
  \item For any $M \in \op{Mod}_A R$, we have an isomorphism of $A$-graded modules
  \begin{displaymath}
   ( \op{I}^{A}_{B} \circ \op{R}^{A}_{B})(M) \cong \bigoplus_{a \in \op{ker} \psi} M(a).
  \end{displaymath}
 \end{itemize}
\end{proposition}

\begin{proof}
 Let $\phi: G(B) \to G(A)$ be the inclusion corresponding to $\psi: A \to B$. The final three statements are from \cite[Lemma 2.16, 2.18]{BFK11} using the equivalence of Proposition~\ref{proposition: Qcoh-G = Mod-hat-G}.

 For the first statement, note that $\op{Res}^{G(A)}_{G(B)}$ does not change the underlying quasi-coherent sheaf, just the equivariant structure. Similarly, $\op{R}^A_B$ only introduces a new $B$-grading on an $A$-graded module, $M$. An element $m \in \op{R}^A_B(M)$ has degree $b \in B$ if and only if $\Gamma(\theta|_{H \times X})(m) = u^{b} \otimes m$.

 Write $m = \sum_a m_a$ so that $\Gamma(\theta)(m) = \sum u^a \otimes m_a$. As $\Gamma(\theta|_{H \times X}) = (\phi \otimes 1)\Gamma(\theta)$, we have
 \[
\Gamma(\theta)(m) =  (\phi \otimes 1)\left(\sum_a u^a \otimes m_a\right) = \sum u^{\psi(a)} \otimes m_a = u^{b} \otimes \left( \sum_a m_a \right).
 \]
Therefore $m$ has degree $b$ if and only if $m_a = 0$ for $\psi(a) \not = b$.

For the second statement, we define a functor
 \begin{displaymath}
  \op{J}^A_B: \op{Mod}_B R \to \op{Mod}_A R
 \end{displaymath}
 by
 \begin{displaymath}
  (\op{J}^A_B N)_a := N_{\psi(a)}
 \end{displaymath}
 and $r \cdot n = rn \in N_{\psi(a+a^{\prime})} = (\op{J}^A_B N)_{a + a^{\prime}}$ if $a \in R_a$ and $n \in (\op{J}^A_B N)_{a^{\prime}} = N_{\psi(a^{\prime})}$. We check that $\op{J}^A_B$ is right adjoint to $\op{R}^A_B$. Since we already know that $\op{I}^A_B$ is right adjoint to $\op{R}^A_B$, this will give an isomorphism of functors, $\op{I}^A_B \cong \op{J}^A_B$.

 A morphism $f: M \to \op{J}^A_B N$ of $A$-graded modules is a morphism of $R$-modules satisfying
 \begin{displaymath}
  f(M_a) \subseteq (\op{J}^A_B N)_a = N_{\psi(a)}.
 \end{displaymath}
 Define
 \begin{align*}
  \tilde{f}: \op{R}^A_B M & \to N \\
  m & \mapsto f(n).
 \end{align*}
 This is a $B$-graded morphism of $R$-modules. Moreover, the assignment is natural in $M$ and $N$. Given a morphism $g: \op{R}^A_B M \to N$ define
 \begin{align*}
  \bar{g}: M & \to \op{J}^A_B N \\
  m & \mapsto g(m) \in N_{\psi(a)} = (\op{J}^A_B N)_a \text{ if } m \in M_a.
 \end{align*}
 This is inverse to the previous assignment.
\end{proof}

\subsection{Factorizations} \label{subsection: background on factorizations}

In this subsection, we recall some definitions and results from \cite[Section 3]{BFK11} and \cite{BDFIK} restricting attention to Abelian group actions on affine varieties.

Let $\mathcal A$ be a $k$-linear Abelian category,
\begin{displaymath}
 \Phi: \mathcal A \to \mathcal A
\end{displaymath}
be an autoequivalence of $\mathcal A$, and
\begin{displaymath}
 w: \op{Id}_{\mathcal A} \to \Phi
\end{displaymath}
be a natural transformation from the identity functor to $\Phi$. We assume that
\begin{displaymath}
 w_{\Phi(A)} = \Phi(w_A)
\end{displaymath}
for all $A \in \mathcal A$.

The objects in the following definition are a natural generalization of D. Eisenbud's matrix factorizations, \cite{EisMF}.

\begin{definition}
 The \newterm{dg-category of factorizations}, $\mathsf{Fact}(\mathcal A,w)$, of the triple, $(\mathcal A, \Phi, w)$, has objects a pair of objects, $E_{-1}, E_0 \in \mathcal A$, and a pair of morphisms,
 \begin{align*}
  \phi_{-1}^{E} &: E_{0} \to \Phi(E_{-1}) \\
  \phi_0^{E} &: E_{-1} \to E_0
 \end{align*}
 such that
 \begin{align*}
  \phi_{-1}^{E} \circ \phi_{0}^{E} & = w_{E_{-1}} : E_{-1} \to \Phi(E_{-1}), \\
  \Phi(\phi_{0}^{E}) \circ \phi^{E}_{-1} & = w_{E_{0}} : E_{0} \to \Phi(E_{0}).
 \end{align*}
 We shall often simply denote the factorization, $(E_{-1}, E_0, \phi^{E}_{-1}, \phi^{E}_0)$, by $E$.

 The morphism complex between two objects, $E$ and $F$, as a graded vector space, can be described as follows. For $n=2l$, we have
 \begin{displaymath}
  \op{Hom}^n_{\mathsf{Fact}(\mathcal A,w)}( E, F) = \op{Hom}_{\mathcal A}(E_{-1}, \Phi^l(F_{-1})) \oplus \op{Hom}_{\mathcal A}(E_0, \Phi^l(F_0))
 \end{displaymath}
 and for $n=2l+1$, we have
 \begin{displaymath}
  \op{Hom}^n_{\mathsf{Fact}(\mathcal A,w)}( E, F) = \op{Hom}_{\mathcal A}(E_{-1}, \Phi^l( F_0 )) \oplus \op{Hom}_{\mathcal A}( E_{0}, \Phi^{l+1}( F_{-1})).
 \end{displaymath}
 The differential applied to $(f_{-1},f_0) \in \op{Hom}^n_{\mathsf{Fact}(\mathcal A,w)}( E, F)$
 \begin{displaymath}
  = \begin{cases}
   \left((f_0 \circ \phi_0^{ E} - \Phi^l(\phi^{ F}_0) \circ f_{-1}, \Phi(f_{-1}) \circ \phi^{ E}_{-1} - \Phi^l( \phi^{ F}_{-1}) \circ f_0\right) & \text{if }n=2l \\
   \left((f_0 \circ \phi_0^{ E} + \Phi^l(\phi^{F}_{-1}) \circ f_{-1}, \Phi(f_{-1}) \circ \phi^{E}_{-1} + \Phi^{l+1}(\phi^{F}_0) \circ f_0\right) & \text{if }n=2l+1. \\
  \end{cases}
 \end{displaymath}
\end{definition}

One can pass to an associated Abelian category. It has the same objects as $\mathsf{Fact}(\mathcal A,w)$, but morphisms between $E$ and $F$ are closed degree-zero morphisms in $\op{Hom}_{\mathsf{Fact}(\mathcal A,w)}(E, F)$. Denote this category by $Z^0\mathsf{Fact}(\mathcal A,w)$. The category, $Z^0\mathsf{Fact}(\mathcal A,w)$, with component-wise kernels and cokernels is an Abelian category.

\begin{definition}
 Given a complex of objects from $Z^0\mathsf{Fact}(\mathcal A,w)$,
\begin{displaymath}
 \cdots \to E^{b} \overset{f^b}{\to} E^{b+1} \overset{f^{b+1}}{\to} \cdots \overset{f^{t-1}}{\to} E^{t} \to \cdots ,
\end{displaymath}
 the \newterm{totalization}, $T$, is the factorization

\begin{align*}
 T_{-1} & := \bigoplus_{i=2l} \Phi^{-l}(E_{-1}^i) \oplus \bigoplus_{i=2l-1} \Phi^{-l}(E_0^i) \\
 T_0 & := \bigoplus_{i=2l} \Phi^{-l}(E_{0}^i) \oplus \bigoplus_{i=2l+1} \Phi^{-l}(E_{-1}^i) \\
 \phi^{T}_0 & := \begin{pmatrix} \ddots & 0 & 0 & 0 & 0 \\ \ddots & -\phi_{-1}^{E^{-1}} & 0 & 0 & 0 \\ 0 & f_0^{-1} & \phi_0^{E^{0}} & 0 & 0 \\ 0 & 0 & f_{-1}^{0} & -\Phi^{-1}(\phi_{-1}^{E^1}) & 0 \\ 0 & 0 & 0 & \ddots & \ddots \end{pmatrix} \\
 \phi^{T}_{-1} & := \begin{pmatrix} \ddots & 0 & 0 & 0 & 0 \\ \ddots & -\Phi(\phi_{0}^{E^{-1}}) & 0 & 0 & 0 \\ 0 & \Phi(f_{-1}^{-1}) & \phi_{-1}^{E^{0}} & 0 & 0\\ 0 & 0 & f_0^{0} & -\phi_{0}^{E^1} & 0 \\ 0 & 0 & 0 & \ddots & \ddots \end{pmatrix}
\end{align*}
 For any closed morphism of cohomological degree zero, $f: E \to F$, in $\mathsf{Fact}(\mathcal A,w)$, we can form the cone factorization, $C(f)$, as the totalization of the complex
 \begin{displaymath}
  E \overset{f}{\to} F
 \end{displaymath}
 where $F$ is in degree zero.

 Let $\op{Tot}: \mathsf{Ch}^{\op{b}}(Z^0\mathsf{Fact}(\mathcal A,w)) \to \mathsf{Fact}(\mathcal A,w)$ denote the totalization dg-functor from the dg-category of bounded chain complexes over $\mathsf{Fact}(\mathcal A,w)$ to $\mathsf{Fact}(\mathcal A,w)$.
\end{definition}

\begin{proposition} \label{prop: facts are triangulated}
 The homotopy category, $[\mathsf{Fact}(\mathcal A,w)]$, is a triangulated category.
\end{proposition}

\begin{proof}
 The translation is $[1]$ is the dg-functor given by
 \begin{displaymath}
  E[1] := (E_0, \Phi(E_{-1}), -\phi^E_{-1}, \Phi(\phi^E_0)).
 \end{displaymath}
 The class of triangles are the sequences of morphisms
 \begin{displaymath}
  E \overset{f}{\to} F \to C(f) \to E[1].
 \end{displaymath}
 The proof now runs completely analogously to proving that the homotopy category of chain complexes of an Abelian category is triangulated. It is therefore suppressed.
\end{proof}

The following definitions are due to Positselski, \cite{Pos1}.

\begin{definition}{(Positselski)}
 Let $\mathsf{Acyc}(\mathcal A,w)$ denote the full subcategory of objects of $\mathsf{Fact}(\mathcal A,w)$ consisting of totalizations of bounded exact complexes from $Z^0\mathsf{Fact}(X,G,w)$. Objects of $\mathsf{Acyc}(\mathcal A,w)$ are called \newterm{acyclic}.

 The \newterm{absolute derived category}, or simply the \newterm{derived category}, of $[\mathsf{Fact}(\mathcal A,w)]$ is the Verdier quotient of $[\mathsf{Fact}(\mathcal A,w)]$ by $[\mathsf{Acyc}(\mathcal A,w)]$,
 \begin{displaymath}
  \op{D}^{\op{abs}}[\mathsf{Fact}(\mathcal A,w)] := [\mathsf{Fact}(\mathcal A,w)]/[\mathsf{Acyc}(\mathcal A,w)].
 \end{displaymath}

 We say that two factorizations are \newterm{quasi-isomorphic} if they are isomorphic in the derived category.
\end{definition}

Let us recall a useful fact, due essentially to Positselski, about $\op{D}^{\op{abs}}[\mathsf{Fact}(\mathcal A,w)]$. Let $\mathsf{Proj}(\mathcal A,w)$ be the dg-subcategory consisting of factorizations with projective components.

\begin{proposition} \label{prop: projective enhancement}
 Assume that any object of $\mathcal A$ has finite projective dimension. The composition
 \begin{align*}
  [\mathsf{Proj}(\mathcal A,w)] & \to [\mathsf{Fact}(\mathcal A,w)] \to \op{D}^{\op{abs}}[\mathsf{Fact}(\mathcal A,w)]
 \end{align*}
 is an equivalence.
\end{proposition}

\begin{proof}
 This is \cite[Corollary 2.23]{BDFIK}.
\end{proof}

We shall study dg-categories of factorizations in two closely-related situations: one,
\begin{itemize}
 \item $\mathcal A = \op{Qcoh}_G X$ or $\op{coh}_G X$ with $X$ a smooth variety acted on by an affine algebraic group, $G$,
 \item $\Phi = \bullet \otimes_{\mathcal O_X} \mathcal L$ for an invertible equivariant sheaf $\mathcal L$, and
 \item $w \in \Gamma(X, \mathcal L)^G$,
\end{itemize}
and, two,
\begin{itemize}
 \item $\mathcal A = \op{Mod}_A R$ or $\op{mod}_A R$ with $A$ a finitely-generated Abelian group and $R$ a smooth, commutative finitely-generated $A$-graded $k$-algebra,
 \item $\Phi = (a)$ for $a \in A$, and
 \item $w \in R_a$.
\end{itemize}

\begin{definition} \label{defn: multi facts}
 In situation one, we will denote the dg-category of quasi-coherent equivariant factorizations by $\mathsf{Fact}(X,G,w)$ and denote the dg-category of coherent equivariant factorizations by $\mathsf{fact}(X,G,w)$. Furthermore, $\mathsf{Vect}(X,G,w)$, respectively $\mathsf{vect}(X,G,w)$, will denote the full sub-dg-category of equivariant factorizations with locally-free components, respectively locally-free coherent components.

 In situation two, we will denote the dg-category of graded factorizations by $\mathsf{Fact}(R,A,w)$ and denote the dg-category of finitely-generated graded factorizations by $\mathsf{fact}(R,A,w)$. Furthermore, $\mathsf{Proj}(R,A,w)$, respectively $\mathsf{proj}(R,A,w)$, will denote the full sub-dg-category of graded factorizations with projective components, respectively finitely-generated projective components.
\end{definition}

\begin{remark}
 By Lemma~\ref{lemma: proj equiv = equiv + proj}, projective graded modules are graded projective modules and the slight ambiguity in Definition~\ref{defn: multi facts} is acceptable.
\end{remark}

\begin{lemma} \label{lemma: graded = equivariant facts}
 Let $X = \op{Spec} R$ and let $w \in \Gamma(X,\mathcal O_X(\chi))^G = R_{\chi}$. Taking global sections induces equivalences of dg-categories
 \begin{align*}
  \mathsf{Fact}(X,G,w) & \cong \mathsf{Fact}(R,\widehat{G},w) \\
  \mathsf{fact}(X,G,w) & \cong \mathsf{fact}(R,\widehat{G},w) \\
  \mathsf{Vect}(X,G,w) & \cong \mathsf{Proj}(R,\widehat{G},w) \\
  \mathsf{Proj}(X,G,w) & \cong \mathsf{proj}(R,\widehat{G},w).
 \end{align*}
\end{lemma}

\begin{proof}
 This is an immediate consequence of Proposition~\ref{proposition: Qcoh-G = Mod-hat-G}.
\end{proof}

\begin{definition}
 We will let $\overline{\mathsf{vect}}(X,G,w)$ be the dg-subcategory of $\mathsf{Vect}(X,G,w)$ consisting of summands, in $[\mathsf{Vect}(X,G,w)]$, of objects from $[\mathsf{vect}(X,G,w)]$. Similarly, let $\overline{\mathsf{proj}}(R,A,w)$ be the dg-subcategory of $\mathsf{Proj}(R,A,w)$ consisting of summands, in $[\mathsf{Proj}(R,A,w)]$, of objects from $\mathsf{proj}(R,A,w)$.
\end{definition}

We will use the following functor later. Assume that $R$ is smooth. Let $\mathsf{mod}^{\op{b}}_A R$ be the dg-category of bounded complexes of finitely-generated $A$-graded $R$-modules. We have a dg-functor
\begin{align*}
 \otimes_R : \mathsf{mod}^{\op{b}}_A R \otimes_k \mathsf{fact}(R,A,w) & \to \mathsf{fact}(R,A,w) \\
 C \otimes_k E & \mapsto \op{Tot}(C^i \otimes_R E).
\end{align*}

If $C$ is a complex of flat $R$-modules, then $C \otimes_R \bullet$ takes acyclic factorizations to acyclic factorizations. Thus, we get an exact functor
\begin{displaymath}
 C \otimes_R \bullet: \dabs[\mathsf{fact}(R,A,w)] \to \dabs[\mathsf{fact}(R,A,w)].
\end{displaymath}
If $C$ is an arbitrary bounded complex of finitely-generated $A$-graded $R$-modules, choose $D$ a bounded complex of finitely-generated flat modules quasi-isomorphic to $C$ and define
\begin{displaymath}
 C \overset{\mathbf{L}}{\otimes}_R \bullet := D \otimes_R \bullet.
\end{displaymath}
Note this does not depend on the choice of $D$, at least up to isomorphism.

Similarly, if $F$ is a factorization with flat components, then we have an exact functor
\begin{displaymath}
 \bullet \otimes_R F : \op{D}^{\op{b}}(\op{mod}_A R) \to \dabs[\mathsf{fact}(R,A,w)].
\end{displaymath}
If $E$ is a factorization with finitely-generated components, by Proposition~\ref{prop: projective enhancement}, there is a factorization, $F$, with finitely-generated projective components that is quasi-isomorphic to $E$. Define
\begin{displaymath}
 \bullet \overset{\mathbf{L}}{\otimes}_R E := \bullet \otimes_R F.
\end{displaymath}
Note, again, that this does not depend upon the choice of $F$.

\begin{remark}
 The functor $\otimes_R$ is a special case of \cite[Definition 3.22]{BFK11} under the equivalences of Lemma~\ref{lemma: graded = equivariant facts}.
\end{remark}

\subsection{Morita products and factorizations}

Let $\op{dg-cat}_k$ denote the category of dg-categories over $k$. If $\mathsf D$ is a dg-category, the homotopy category, $[\mathsf D]$, has the same objects as $\mathsf D$ and
\begin{displaymath}
 \op{Hom}_{[\mathsf D]}(d,d^{\prime}) := \op{H}^0(\op{Hom}_{\mathsf D}(d,d^{\prime})).
\end{displaymath}

\begin{definition}
 Let $\mathsf D$ be a small dg-category. The category of left $\mathsf D$-modules, denoted $\mathsf D\op{-Mod}$, is the dg-category of dg-functors, $\mathsf D \to \mathsf{C}(k)$ where $\mathsf{C}(k)$ is the dg-category of chain complexes of vector spaces over $k$. The category of right $\mathsf D$-modules is the category of left $\mathsf D^{\op{op}}$-modules.

 Each object $d \in \mathsf D$ provides a representable right module
 \begin{align*}
  h_d : \mathsf{D}^{\op{op}} & \to \mathsf{C}(k) \\
  d^{\prime} & \mapsto \op{Hom}_{\mathsf{D}}(d^{\prime},d).
 \end{align*}
 We denote the dg-Yoneda embedding by $h: \mathsf{D} \to \mathsf{D}^{\op{op}}\op{-Mod}$. We also denote the contravariant version by
 \begin{align*}
  h^d : \mathsf{D} & \to \mathsf{C}(k) \\
  d^{\prime} & \mapsto \op{Hom}_{\mathsf{D}}(d,d^{\prime}).
 \end{align*}

 The Verdier quotient of $[\mathsf{D}\op{-Mod}]$ by the subcategory of acyclic modules is called the \newterm{derived category of $\mathsf{D}$-modules} and is denoted by $\op{D}[\mathsf{D}\op{-Mod}]$. The smallest thick subcategory of $\op{D}[\mathsf{D}^{\op{op}}\op{-Mod}]$ containing the image of $[h]$ is called the \newterm{category of perfect $\mathsf{D}$-modules} and is denoted by $\op{perf}(\mathsf{D})$.
\end{definition}

\begin{remark}
 We will also consider dg-categories that are \newterm{quasi-small}. A dg-category $\mathsf{D}$ is quasi-small if $[\mathsf{D}]$ is essentially small. In this case, we can choose a small full subcategory of $\mathsf{D}$ quasi-equivalent to $\mathsf{D}$ and work with that subcategory to define categories of modules and bimodules. However, doing this in each example is tedious and not edifying. So we will suppress these arguments throughout the paper.
\end{remark}

\begin{definition}
 A dg-category, $\mathsf D$, is \newterm{pre-triangulated} if $[\mathsf D]$ is a triangulated subcategory, via $h$, of $\op{D}[\mathsf{D}^{\op{op}}\op{-Mod}]$. A dg-category, $\mathsf{D}$, is \newterm{pre-thick} if $[\mathsf D]$ is a thick, which is by definition triangulated, subcategory of $\op{D}[\mathsf{D}^{\op{op}}\op{-Mod}]$.
\end{definition}

\begin{definition}
 A dg-functor, $f: \mathsf{C} \to \mathsf{D}$, is a \newterm{quasi-equivalence} if
 \begin{displaymath}
  \op{H}^{\bullet}(f): \op{H}^{\bullet}(\op{Hom}_{\mathsf{C}}(c,c^{\prime})) \to \op{H}^{\bullet}(\op{Hom}_{\mathsf{D}}(f(c),f(c^{\prime})))
 \end{displaymath}
 is a an isomorphism for all $c,c^{\prime} \in \mathsf{C}$ and $[f]: [\mathsf{C}] \to [\mathsf{D}]$ is essentially surjective.

 Let $\op{Ho(dg-cat)}_k$ denote the localization of $\op{dg-cat}_k$ at the class of quasi-equivalences. This category is called  \newterm{the homotopy category of dg-categories}. If $\mathsf{C}$ and $\mathsf{D}$ are quasi-equivalent, we shall write $\mathsf{C} \simeq \mathsf{D}$.
\end{definition}

\begin{definition}
 Let $\mathsf{C}$ and $\mathsf{D}$ be two dg-categories. A \newterm{quasi-functor} $a: \mathsf{C} \to \mathsf{D}$ is a dg-functor
 \begin{displaymath}
  a: \mathsf{C} \to \mathsf{D}^{\op{op}}\op{-Mod}
 \end{displaymath}
 such that for each $c \in \mathsf{C}$, $a(c)$ is \newterm{quasi-representable}, i.e. quasi-isomorphic to $h_d$ for some $d \in \mathsf{D}$. Note that a quasi-functor corresponds to a bimodule $a \in \mathsf{C} \otimes_k \mathsf{D}^{\op{op}}\op{-Mod}$. Also note, that any quasi-functor induces a functor on homotopy categories which we denote by $[a]: [\mathsf{C}] \to [\mathsf{D}]$.
\end{definition}

\begin{lemma} \label{lemma: morphisms in hodgcat}
 The isomorphism classes of morphisms from $\mathsf{C}$ to $\mathsf{D}$ in $\op{Ho(dg-cat)}_k$ are in bijection with isomorphism classes of quasi-functors from $\mathsf{C}$ to $\mathsf{D}$ viewed as objects of $\op{D}[\mathsf{C} \otimes_k \mathsf{D}^{\op{op}}\op{-Mod}]$.
\end{lemma}

\begin{proof}
 This is an immediate consequence of the internal Hom constructed by T\"oen for $\op{Ho(dg-cat)}_k$, \cite[Theorem 6.1]{Toe}.
\end{proof}

\begin{remark}
 As a consequence of Lemma~\ref{lemma: morphisms in hodgcat}, to demonstrate that two dg-categories are isomorphic in $\op{Ho(dg-cat)}_k$ is equivalent to proving the existence of a quasi-functor between them whose associated functor on homotopy categories is an equivalence.
\end{remark}

\begin{definition}
 Let $\mathsf{C}$ be a small dg-category. The category $\mathsf{C}\op{-Mod}$ possesses the structure of a model category with $f: F \to G$ being a fibration, respectively a weak equivalence, if $f(c): F(c) \to G(c)$ is an epimorphism in each degree, respectively a quasi-isomorphism, for each $c \in \mathsf{C}$. This determines the cofibrations: they are those morphisms satisfying the left lifting property with respect to all acyclic fibrations, i.e. those maps that are fibrations and weak equivalences.

 Any object of $\mathsf{C}\op{-Mod}$ is fibrant. We let $\widehat{\mathsf{C}}$ be the subcategory of cofibrant objects in $\mathsf{C}^{\op{op}}\op{-Mod}$. The dg-category $\widehat{\mathsf{C}}$ is an enhancement of $\op{D}[\mathsf{C}^{\op{op}}\op{-Mod}]$. We let $\widehat{\mathsf{C}}_{\op{pe}}$ be the full sub-dg-category of $\widehat{\mathsf{C}}$ consisting of all objects that are compact in $\op{D}[\mathsf{C}^{\op{op}}\op{-Mod}]$. As any representable dg-module is cofibrant, we have a dg-functor \begin{displaymath}
  h: \mathsf{C} \to \widehat{\mathsf{C}}_{\op{pe}}.
 \end{displaymath}

 Following the lead of T\"oen, we introduce the following product. Let $\mathsf{D}$ be another small dg-category over $k$. The \newterm{Morita product} of $\mathsf{C}$ and $\mathsf{D}$ is
 \begin{displaymath}
  \mathsf{C} \circledast \mathsf{D} := \widehat{( \mathsf{C} \otimes_k \mathsf{D} )}_{\op{pe}}
 \end{displaymath}
 viewed as an object of $\hodgcat$. Because we view it as an object of $\hodgcat$, it is unique up to quasi-equivalence.
\end{definition}

\begin{remark}
 The cofibrant objects of $\mathsf{C}^{\op{op}}\op{-Mod}$ are exactly the summands of semi-free dg-modules \cite{FHT}. One can check that summands of semi-free dg-modules have the appropriate lifting property. Furthermore, for any dg-module, $M$, there exists a semi-free dg-module, $F$, and an acyclic fibration, $F \to M$. If we assume that $M$ is cofibrant, this must split.
\end{remark}

\begin{lemma} \label{lemma: alternative characterization of Morita product}
 Let $\mathsf{C}$ be a small dg-category. Let $\mathsf{D}$ be a pre-thick dg-subcategory of $\mathsf{C}^{\op{op}}\op{-Mod}$ such that the functor
 \begin{displaymath}
  [\mathsf{D}] \to \op{D}[\mathsf{C}^{\op{op}}\op{-Mod}]
 \end{displaymath}
 is fully-faithful and its essential image is $\op{perf}(\mathsf{C})$. Then, there is a quasi-equivalence
 \begin{displaymath}
  \widehat{\mathsf C}_{\op{pe}} \simeq \mathsf{D}.
 \end{displaymath}
\end{lemma}

\begin{proof}
 Both $\widehat{\mathsf C}_{\op{pe}}$ and $\mathsf{D}$ are dg-subcategories of $\mathsf{C}^{\op{op}}\op{-Mod}$. Consider the dg-functor,
 \begin{align*}
  a: \mathsf{D} & \to (\widehat{\mathsf C}_{\op{pe}})^{\op{op}}\op{-Mod} \\
  d & \mapsto \op{Hom}_{\mathsf{C}^{\op{op}}\op{-Mod}}(\bullet, d).
 \end{align*}
 We claim that $a$ is a quasi-functor and $[a]: [\mathsf{D}] \to [\widehat{\mathsf{C}}_{\op{pe}}]$ is an equivalence.

 Any object of $\mathsf{C}^{\op{op}}\op{-Mod}$ is fibrant so
 \begin{displaymath}
  \op{H}^0(\op{Hom}_{\mathsf{C}^{\op{op}}\op{-Mod}}(N, M)) \cong \op{Hom}_{ \op{D}[\mathsf{C}^{\op{op}}\op{-Mod}]}(N,M)
 \end{displaymath}
 for $N \in \widehat{\mathsf{C}}_{\op{pe}}$. In particular, if $d \in \mathsf D$ is isomorphic to $N \in \widehat{\mathsf{C}}_{\op{pe}}$ in $\op{D}[\mathsf{C}^{\op{op}}\op{-Mod}]$, then there is a morphism of modules, $f: N \to d$, inducing the quasi-isomorphism. The induced natural transformation
 \begin{displaymath}
  \op{Hom}(\bullet,f): \op{Hom}_{\mathsf{C}^{\op{op}}\op{-Mod}}(\bullet, N) \to \op{Hom}_{\mathsf{C}^{\op{op}}\op{-Mod}}(\bullet, d)
 \end{displaymath}
 is a quasi-isomorphism when applied to any object of $\widehat{\mathsf C}_{\op{pe}}$. Thus, $a(d)$ is quasi-representable when $d$ is isomorphic to an object of $\widehat{\mathsf C}_{\op{pe}}$ in $\op{D}[\mathsf{C}^{\op{op}}\op{-Mod})]$. Assume $d$ is quasi-isomorphic to $N$ and $d$ is quasi-isomorphic to $N^{\prime}$, we have natural isomorphisms
 \begin{align*}
  \op{Hom}_{\op{D}[\widehat{\mathsf{C}}_{\op{pe}}^{\op{op}}\op{-Mod}]} (a(d),a(d^{\prime})))
  & \cong \op{Hom}_{\op{D}[\widehat{\mathsf{C}}_{\op{pe}}^{\op{op}}\op{-Mod}]}(h_{N},h_{N^{\prime}}) \\
  & \cong \op{Hom}_{[\widehat{\mathsf{C}}_{\op{pe}}^{\op{op}}]}(N,N^{\prime}) \\
  & \cong \op{Hom}_{\op{D}[\mathsf{C}^{\op{op}}\op{-Mod}]}(N,N^{\prime}) \\
  & \cong \op{Hom}_{\op{D}[\mathsf{C}^{\op{op}}\op{-Mod}]}(d,d^{\prime}) \\
  & \cong \op{H}^0(\op{Hom}_{[\mathsf{D}]}(d,d^{\prime}))
 \end{align*}
  where the first isomorphism uses that $a(d)$ is quasi-isomorphic to $h_N$ and $a(d^{\prime})$ is quasi-isomorphic to $h_{N^{\prime}}$, the second isomorphism uses the Yoneda embedding, the third isomorphism that $N$ is cofibrant, the fourth uses the quasi-isomorphisms we have assumed, and the final isomorphism is because we assumed that $[\mathsf{D}] \to \op{D}[\mathsf{C}^{\op{op}}\op{-Mod}]$ is fully-faithful. Thus, $a$ is fully-faithful whenever the two objects of $\mathsf{D}$ are mapped to quasi-representable modules.

  Consequently, $a$ is fully-faithful on the full subcategory of $\mathsf{D}$ consisting of modules that are isomorphic to an object of $\op{perf}(\mathsf{C})$ in $\op{D}[\mathsf{C}^{\op{op}}\op{-Mod}]$. However, by assumption, this is every object of $\mathsf{D}$. Thus, $[a]$ is fully-faithful and, by assumption, its essential image is $\op{perf} \mathsf C$. Thus, $[a] : [\mathsf D] \to [\widehat{\mathsf C}_{\op{pe}}]$ is an equivalence and $\mathsf D$ and $\widehat{\mathsf C}_{\op{pe}}$ are quasi-equivalent.
\end{proof}

\begin{remark}
 By Lemma~\ref{lemma: alternative characterization of Morita product}, the Morita product of $\mathsf{C}$ and $\mathsf{D}$ can be defined to be \textit{any} enhancement of $\op{perf}(\mathsf C \otimes \mathsf D)$ contained in $(\mathsf{C} \otimes_k \mathsf{D})^{\op{op}}\op{-Mod}$.
\end{remark}

\begin{lemma} \label{lemma: comm and ass of Morita product}
 The Morita product, $\circledast$, is commutative up to equivalence and associative up to quasi-equivalence.
\end{lemma}

\begin{proof}
 For two small dg-categories, $\mathsf{C}$ and $\mathsf{D}$, there is an obvious equivalence between $\mathsf{C} \otimes_k \mathsf{D}$ and $\mathsf{D} \otimes_k \mathsf{C}$ and their corresponding categories of modules. This clearly preserves fibrations and cofibrations and therefore induces an equivalence between $\mathsf{C} \circledast \mathsf{D}$ and $\mathsf{D} \circledast \mathsf{C}$.

 Let $\mathsf{C}, \mathsf{D},$ and $\mathsf{E}$ be small dg-categories. Let
 \begin{displaymath}
  \mathsf{F} := \widehat{( \mathsf{C} \otimes_k \mathsf{D} \otimes_k \mathsf{E} )}_{\op{pe}}.
 \end{displaymath}
 We will show that $\mathsf{F}$ is quasi-equivalent to $(\mathsf{C} \circledast \mathsf{D}) \circledast \mathsf{E}$ by applying Lemma~\ref{lemma: alternative characterization of Morita product}. Clearly any object of $(\mathsf{C} \circledast \mathsf{D}) \circledast \mathsf{E}$ induces a $\mathsf{C} \otimes_k \mathsf{D} \otimes_k \mathsf{E}$-module. Consider the largest thick subcategory of $[ (\mathsf{C} \circledast \mathsf{D}) \circledast \mathsf{E} ]$ for which
 \begin{displaymath}
  [ (\mathsf{C} \circledast \mathsf{D}) \circledast \mathsf{E} ] \to \op{D}[(\mathsf{C} \otimes_k \mathsf{D} \otimes_k \mathsf{E})^{\op{op}}\op{-Mod}]
 \end{displaymath}
 is fully-faithful. By the Yoneda embedding, this contains $\mathsf{C} \otimes_k \mathsf{D} \otimes_k \mathsf{E}$ which generates $[ (\mathsf{C} \circledast \mathsf{D}) \circledast \mathsf{E} ]$.  It follows that this category agrees with $[ (\mathsf{C} \circledast \mathsf{D}) \circledast \mathsf{E} ]$ i.e.\ the functor is fully-faithful.

 The essential image of $[ (\mathsf{C} \circledast \mathsf{D}) \circledast \mathsf{E} ]$ is a thick subcategory generated by the objects from $\mathsf{C} \otimes_k \mathsf{D} \otimes_k \mathsf{E}$ and, therefore, must be $\op{perf}(\mathsf{C} \otimes_k \mathsf{D} \otimes_k \mathsf{E})$. We have verified the hypotheses of Lemma~\ref{lemma: alternative characterization of Morita product} and get a quasi-equivalence
 \begin{displaymath}
  (\mathsf{C} \circledast \mathsf{D}) \circledast \mathsf{E} \simeq \mathsf{F}.
 \end{displaymath}
 A similar argument gives a quasi-equivalence
 \begin{displaymath}
  \mathsf{C} \circledast (\mathsf{D} \circledast \mathsf{E}) \simeq \mathsf{F}.
 \end{displaymath}
\end{proof}

\begin{lemma} \label{lemma: Morita product and idempotent completion}
 Assume that $\mathsf{D}$ is a dg-subcategory of another dg-category $\overline{\mathsf{D}}$ which is pre-thick and assume that the image of $[\mathsf{D}]$ generates $[\overline{\mathsf{D}}]$. Then, there are quasi-equivalences
 \begin{displaymath}
  \widehat{\mathsf{D}}_{\op{pe}} \simeq \mathsf{D} \circledast \mathsf{C}(k)_{\op{pe}} \simeq \overline{\mathsf{D}} \circledast \mathsf{C}(k)_{\op{pe}} \simeq \overline{\mathsf{D}}.
 \end{displaymath}
\end{lemma}

\begin{proof}
 Note that $\mathsf{D} \otimes_k k \cong \mathsf{D}$ is a full dg-subcategory of $\mathsf{D} \circledast \mathsf{C}(k)_{\op{pe}},$ of $\overline{\mathsf{D}}$, and of $\overline{\mathsf{D}} \circledast \mathsf{C}(k)_{\op{pe}}$. The argument providing a quasi-equivalence with $\widehat{\mathsf{D}}_{\op{pe}}$ runs the same for each category so we only present it for $\overline{\mathsf{D}} \circledast \mathsf{C}(k)_{\op{pe}}$.

 By restriction, each object of $\overline{\mathsf{D}} \circledast \mathsf{C}(k)_{\op{pe}}$ provides a $\mathsf{D}^{\op{op}}$-module,
 \begin{displaymath}
  a: M \mapsto \op{Hom}_{\mathsf{D} \circledast \mathsf{C}(k)_{\op{pe}}}(\bullet, M)|_{\mathsf{D}}.
 \end{displaymath}
 We will apply Lemma~\ref{lemma: alternative characterization of Morita product} to prove that this induces a quasi-equivalence
 \begin{displaymath}
  \widehat{\mathsf{D}}_{\op{pe}} \simeq \overline{\mathsf{D}} \circledast \mathsf{C}(k)_{\op{pe}}.
 \end{displaymath}
 N.B. the slight notational conflict, namely in the notation of  Lemma~\ref{lemma: alternative characterization of Morita product} we are letting $\mathsf{C} = \mathsf{D}$ and $\mathsf{D} =  \overline{\mathsf{D}} \circledast \mathsf{C}(k)_{\op{pe}}$.

 For any object, $d$, of $\mathsf{D}$, viewed as $d \otimes k \in \overline{\mathsf{D}} \circledast \mathsf{C}(k)_{\op{pe}}$, the induced module is representable, and, by Yoneda, the functor is fully-faithful on objects from $\mathsf{D}$. The full subcategory of $\overline{\mathsf{D}} \circledast \mathsf{C}(k)_{\op{pe}}$ on which
 \begin{equation*} \label{equation: something something}
  \bar{a}: [\overline{\mathsf{D}} \circledast \mathsf{C}(k)_{\op{pe}}] \to \op{D}[\mathsf{D}^{\op{op}}\op{-Mod}]
 \end{equation*}
 is fully-faithful is thick and contains objects from $\mathsf{D}$. Since $[\mathsf{D}]$ generates $[\overline{\mathsf{D}}]$ and $[\overline{\mathsf{D}}] \otimes_k k$ generates $[\overline{\mathsf{D}} \circledast \mathsf{C}(k)_{\op{pe}}]$, $\bar{a}$ must be fully-faithful on all of $[\overline{\mathsf{D}} \circledast \mathsf{C}(k)_{\op{pe}}]$. Hence, its essential image is a thick subcategory containing and generated by $[\mathsf{D}]$.  Thus, the essential image is $\op{perf}(\mathsf{D})$. Consequently, $\bar{a}$ is fully-faithful onto $\op{perf}(\mathsf{D})$ and we may apply Lemma~\ref{lemma: alternative characterization of Morita product}.
\end{proof}

Let us give a few examples of Morita products.

\begin{definition}
 Let $R$ be a $k$-algebra. We do not assume that $R$ is commutative. Recall that $R$ is \newterm{(left) coherent} if the kernel of any map between finitely-generated left $R$-modules is finitely-generated.

 If $R$ is coherent, denote by $\mathsf{proj}(R)$ the dg-category of bounded above complexes of finitely-generated projective left $R$-modules with bounded cohomology.
\end{definition}

\begin{remark}
 If $R$ is Noetherian, then it is coherent. In particular, if $R$ is finite-dimensional over $k$, it is coherent.
\end{remark}

\begin{lemma} \label{lemma: Morita product of algebras}
 Assume $R, S$, and $R \otimes_k S$ are coherent $k$-algebras, not necessarily commutative. Assume that either $R \otimes_k R^{\op{op}}$ or $S \otimes_k S^{\op{op}}$ is coherent and has finite global dimension. Then,
 \begin{displaymath}
  \mathsf{proj}(R) \circledast \mathsf{proj}(S) \simeq \mathsf{proj}(R \otimes_k S).
 \end{displaymath}
\end{lemma}

\begin{proof}
 We first use the assumption that either $R \otimes_k R^{\op{op}}$ or $S \otimes_k S^{\op{op}}$ is coherent and has finite global dimension. For the sake of specificity, let us assume $R \otimes_k R^{\op{op}}$ is coherent and has finite global dimension. For $t  >> 0$, form an exact sequence of $R$-$R$ bimodules
 \begin{displaymath}
  0 \to K \to P_t \otimes_k P_t^{\prime} \to \cdots \to P_0 \otimes_k P_0^{\prime} \to R \to 0
 \end{displaymath}
 with $P_i,P_i^{\prime}$ projective $R$-modules. Since we have assumed that $R \otimes_k R^{\op{op}}$ is coherent, we can take $P_i$ and $P_i^{\prime}$ to be finitely-generated. Since we have assumed that $R \otimes_k R^{\op{op}}$ has finite global dimension, we have
 \begin{displaymath}
  \op{Ext}^t_{R \otimes_k R^{\op{op}}}(R, K) = 0
 \end{displaymath}
 for sufficiently large $t$ and hence $R$ is a summand of the complex
 \begin{equation} \label{equation: diagonal}
  C := P_t \otimes_k P_t^{\prime} \to \cdots \to P_0 \otimes_k P_0^{\prime}
 \end{equation}
 in $\dbmod{R \otimes_k R^{\op{op}}}$.
 Take an object $M \in \mathsf{proj}(R \otimes_k S)$ and tensor over $R$ with $C$. As $R$ is a summand of $C$  in $\dbmod{R \otimes_k R^{\op{op}}}$, we know that $M$ is a summand of $C \otimes_R M$ in $\dbmod{R \otimes_k S}$. However, each of the terms in $C \otimes_R M$ lie in $\mathsf{proj}(R) \otimes_k \mathsf{proj}(S)$. Consequently, $[\mathsf{proj}(R \otimes_k S)]$ is generated by $[\mathsf{proj}(R)] \otimes_k [\mathsf{proj}(S)]$. 
 
 Next, we seek to apply Lemma~\ref{lemma: alternative characterization of Morita product}. Let $\mathsf{mod}(R \otimes_k S)$ be the dg-category of complexes with bounded and finitely-generated cohomology. Both $\mathsf{proj}(R) \otimes_k \mathsf{proj}(S)$ and $\mathsf{proj}(R \otimes_k S)$ are naturally subcategories of $\mathsf{mod}(R \otimes_k S)$. As such we can consider the dg-functor
 \begin{align*}
  a: \mathsf{proj}(R \otimes_k S) & \to (\mathsf{proj}(R) \otimes_k \mathsf{proj}(S))^{\op{op}}\op{-Mod} \\
  M & \mapsto \op{Hom}_{\mathsf{mod}(R \otimes_k S)}(\bullet, M).
 \end{align*}
 
 We need to check that $[a]$ is fully-faithful and has essential image $\op{perf}(\mathsf{proj}(R) \otimes_k \mathsf{proj}(S))$. By Yoneda, $[a]$ is fully-faithful on $[\mathsf{proj}(R)] \otimes_k [\mathsf{proj}(S)]$. We have seen that $[\mathsf{proj}(R)] \otimes_k [\mathsf{proj}(S)]$ generates $[\mathsf{proj}(R \otimes_k S)]$. Thus, $[a]$ is fully-faithful on all of $[\mathsf{proj}(R \otimes_k S)]$. 
 
 Similarly, the essential image of $[\mathsf{proj}(R)] \otimes_k [\mathsf{proj}(S)]$ under $[a]$ is $[\mathsf{proj}(R)] \otimes_k [\mathsf{proj}(S)] \subset \op{perf}(\mathsf{proj}(R) \otimes_k \mathsf{proj}(S))$. Since $[\mathsf{proj}(R)] \otimes_k [\mathsf{proj}(S)]$ generates both $[\mathsf{proj}(R \otimes_k S)]$ and $\op{perf}(\mathsf{proj}(R) \otimes_k \mathsf{proj}(S))$, the essential image of $[a]$ must be $\op{perf}(\mathsf{proj}(R) \otimes_k \mathsf{proj}(S))$.
 q
\end{proof}

\begin{remark}
 The essential fact in the proof of Lemma~\ref{lemma: Morita product of algebras} is that the derived category $\dbmod{R \otimes_k S}$ is generated by the image of $\dbmod{R} \otimes_k \dbmod{S}$. The restriction on the global dimension was made to guarantee this. If we work with commutative rings, essentially of finite type, then this fact follows from the methods of \cite{Ro2} without assumptions on global dimension.
\end{remark}

\begin{definition}
 Let $Z$ be a variety equipped with the action of an algebraic group $G$. Let $\mathsf{Inj}_{\op{coh}_G}(Z)$ be the dg-category of bounded below complexes of injective quasi-coherent $G$-equivariant sheaves with bounded and coherent cohomology.
\end{definition}

\begin{remark}
 Here we encounter our main quasi-small, but not small, dg-category.
\end{remark}

\begin{theorem} \label{theorem: Morita product of varieties}
 Let $X$ and $Y$ be smooth and projective varieties over $k$. Then,
 \begin{displaymath}
  \mathsf{Inj}_{\op{coh}}(X) \circledast \mathsf{Inj}_{\op{coh}}(Y) \simeq \mathsf{Inj}_{\op{coh}}(X \times Y).
 \end{displaymath}
\end{theorem}

\begin{proof}
 This is in \cite[Section 8]{Toe}.
\end{proof}

Next, we turn to factorizations.

\begin{definition}
 Let $X$ and $Y$ be smooth affine varieties and let $w \in \Gamma(X,\mathcal O_X)$ and $v \in \Gamma(Y,\mathcal O_Y)$. We set
 \begin{displaymath}
  w \boxplus v := w \otimes 1 + 1 \otimes v \in \Gamma(X,\mathcal O_X) \otimes_k \Gamma(Y,\mathcal O_Y) \cong \Gamma(X \times Y, \mathcal O_{X \times Y}).
 \end{displaymath}
\end{definition}

We will have to deal with two potentials, $w, v \in \Gamma(X,\mathcal O_X)$, that are semi-invariant with respect to different characters of different groups. The largest group for which $w \boxplus v$ is semi-invariant is as follows.

\begin{definition}
 Let $G$ and $H$ be Abelian affine algebraic groups and let $\chi: G \to \mathbb{G}_m$ and $\chi^{\prime}: H \to \mathbb{G}_m$ be characters. Define a character of $G \times H$ by
 \begin{align*}
  \chi^{\prime}-\chi : G \times H & \to \mathbb{G}_m \\
  (g,h) & \mapsto \chi(g)^{-1}\chi^{\prime}(h).
 \end{align*}
 Let $G\times_{\mathbb G_m} H$ be the kernel of $\chi^{\prime}-\chi$, or, equivalently, the fiber product of $G$ and $H$ over $\mathbb G_m$.

 Dual to this construction we have the following. Let $A$ and $B$ be finitely-generated Abelian groups. Let $a \in A$ and $b \in B$ define
 \begin{displaymath}
  A \boxminus B := A \oplus B / (a,-b).
 \end{displaymath}
\end{definition}

\begin{proposition} \label{proposition: morita product of factorizations}
 Let $X$ be a smooth affine variety, $G$ an Abelian affine algebraic group acting on $X$, $\chi: G \to \mathbb{G}_m$ a character, and $w \in \Gamma(X,\mathcal O_X(\chi))^G$. Let $Y$ be a smooth affine variety, $H$ an Abelian affine algebraic group acting on $X$, $\chi^{\prime}: H \to \mathbb{G}_m$ is a character, and $v \in \Gamma(Y,\mathcal O_Y(\chi^{\prime}))^H$.

 If we assume that $\chi - \chi^{\prime}$ is not torsion and the singular locus of $w \boxplus v$ is contained in the product of the zero loci of $w$ and $v$ in $X \times Y$, then there is a quasi-equivalence
 \begin{displaymath}
  \mathsf{vect}(X,G,w) \circledast \mathsf{vect}(Y,H,v) \simeq \overline{\mathsf{vect}}(X \times Y, G \times_{ \mathbb{G}_m} H, w \boxplus v).
 \end{displaymath}
\end{proposition}

\begin{proof}
 This is \cite[Corollary 5.14]{BFK11}.
\end{proof}

\begin{corollary} \label{cor: Morita product of MF}
 Let $A$ and $B$ be finitely-generated Abelian groups. Let $R$ and $S$ be smooth finitely-generated commutative $k$-algebras. Assume that $R$ is $A$-graded and $S$ is $B$-graded. Let $w \in R_a$ and $v \in S_b$. Assume that $(a,-b)$ is not torsion in $A \oplus B$ and that the singular locus of $w \boxplus v$ is contained in the product of the zero loci of $w$ and $v$, then there is an isomorphism
 \begin{displaymath}
  \mathsf{proj}(R,A,w) \circledast \mathsf{proj}(S,B,v) \simeq \overline{\mathsf{proj}}(R \otimes_k S, A \boxminus B, w \boxplus v).
 \end{displaymath}
\end{corollary}

\begin{proof}
 This is an immediate corollary of Proposition~\ref{proposition: morita product of factorizations} and Lemma~\ref{lemma: graded = equivariant facts}.
\end{proof}

\begin{remark}
 We will use the following observation frequently. Assume that
 \begin{align*}
  R & = k[x_1,\ldots,x_{n_1}] \\
  S & = k[x_1,\ldots,x_{n_2}].
 \end{align*}
 Further assume that there are homomorphisms
 \begin{align*}
  \phi & : A \to \Z  \\
  \phi^{\prime} & : B \to \Z
 \end{align*}
 such that the $\Z$-degrees of all $x_i$ are positive and the $\Z$-degrees of $w$ and $v$ are nonzero, hence positive.

 Then, Euler's formula applies and both $w$ and $v$ are elements of their respective Jacobian ideals: $w \in (\partial w)$ and $v \in (\partial v)$. Thus, the singular locus of $w \boxplus v$, is contained in the product of the zero loci of $w$ and $v$ i.e.
 \begin{displaymath}
  w \otimes 1, 1 \otimes v \in (\partial(w \boxplus v), w \boxplus v) = (\partial w \otimes 1, 1 \otimes \partial v, w \boxplus v).
 \end{displaymath}
\end{remark}

Next, we wish to describe the effect of a Morita product on a semi-orthogonal decomposition. First, we recall these notions in the setting of triangulated categories.

To define a semi-orthogonal decomposition we will need the following notion.

\begin{definition}
 Let $i: \mathcal S \to \mathcal T$ be the inclusion of a full triangulated subcategory of $\mathcal T$. The subcategory, $\mathcal S$, is called \newterm{right admissible} if $i$ has a right adjoint and \newterm{left admissible} if it has a left adjoint. A full triangulated subcategory is called \newterm{admissible} if it is both right and left admissible.
\end{definition}

\begin{definition}
 An object, $E$, in a $k$-linear triangulated category, $\mathcal T$, is called \newterm{exceptional} if,
 \[
  \op{Hom}_{\mathcal T}(E,E[i])  =
  \begin{cases}
   k & \text{ if } i=0  \\
   0 & \text{ if } i \neq 0. \\
  \end{cases}
 \]
\end{definition}

Following \cite{CT}, one can give an extension of the notion of admissible subcategories for pre-triangulated dg-categories by passing to their homotopy categories.

\begin{definition}
 Let $\mathsf C$ be a pre-triangulated dg-category and $\mathsf D$ another pre-triangulated dg-category.  Then $\mathsf D$ is \newterm{(right/left) admissible} in $\mathsf C$ if $\mathsf{D}$ is quasi-equivalent to a dg-subcategory $\mathsf{D}^{\prime}$ of $\mathsf{C}$ and $[\mathsf D^{\prime}]$ is (right/left) admissible in $[\mathsf C]$. An object $E \in \mathsf C$ is \newterm{exceptional} if $E \in [\mathsf C]$ is exceptional.
\end{definition}

Let $\mathcal T$ be a triangulated category and $\mathcal I$ a full subcategory. Recall that the left orthogonal, $\leftexp{\perp}{\mathcal I}$, is the full subcategory of $\mathcal T$ consisting of all objects, $T \in \mathcal T$, with $\op{Hom}_{\mathcal T}(T,I) = 0$ for any $I \in \mathcal I$. The right orthogonal, $\mathcal I^{\perp}$, is defined similarly.

\begin{definition}\label{def:SO}
 A \newterm{semi-orthogonal decomposition} of a triangulated category, $\mathcal T$, is a sequence of full triangulated subcategories, $\mathcal S_1, \dots ,\mathcal S_m$, in $\mathcal T$ such that $\mathcal S_i \subseteq \mathcal S_j^{\perp}$ for $i<j$ and, for every object $T \in \mathcal T$, there exists a diagram:
\begin{center}
\begin{tikzpicture}[description/.style={fill=white,inner sep=2pt}]
\matrix (m) [matrix of math nodes, row sep=1em, column sep=1.5em, text height=1.5ex, text depth=0.25ex]
{  0 & & T_{m-1} & \cdots & T_2 & & T_1 & & T   \\
   & & & & & & & &  \\
   & S_m & & & & S_2 & & S_1 & \\ };
\path[->,font=\scriptsize]
 (m-1-1) edge (m-1-3)
 (m-1-3) edge (m-1-4)
 (m-1-4) edge (m-1-5)
 (m-1-5) edge (m-1-7)
 (m-1-7) edge (m-1-9)

 (m-1-9) edge (m-3-8)
 (m-1-7) edge (m-3-6)
 (m-1-3) edge (m-3-2)

 (m-3-8) edge node[sloped] {$ | $} (m-1-7)
 (m-3-6) edge node[sloped] {$ | $} (m-1-5)
 (m-3-2) edge node[sloped] {$ | $} (m-1-1)
;
\end{tikzpicture}
\end{center}
 where all triangles are distinguished and $S_k \in \mathcal S_k$. We shall denote a semi-orthogonal decomposition by $\langle \mathcal S_1, \ldots, \mathcal S_m \rangle$.
\end{definition}

\begin{definition}
 Let $\mathsf C$ be a pre-triangulated dg-category and $\mathsf D_1, \ldots, \mathsf D_m$ be quasi-equivalent to dg-subcategories of $\mathsf C$, such that
 \begin{displaymath}
  [\mathsf C] = \langle [\mathsf D_1], \ldots, [\mathsf D_m] \rangle
 \end{displaymath}
 is a semi-orthogonal decomposition. Then, we say that $\mathsf D_1, \ldots, \mathsf D_m$ form a \newterm{semi-orthogonal decomposition} of $\mathsf C$ and similarly write
 \begin{displaymath}
  \mathsf C \simeq \langle \mathsf D_1, \ldots, \mathsf D_m \rangle.
 \end{displaymath}
\end{definition}

\begin{remark}
 When $E$ is either an exceptional object of a hom-finite dg-category or hom-finite triangulated category, then the inclusion of the smallest triangulated category generated by $E$ has right adjoint $\op{Hom}(E, -) \otimes_k E$ and left adjoint $\op{Hom}(-, E)^{\vee} \otimes_k E$.  Hence, the smallest triangulated category generated by $E$ is admissible. Following conventions, when some $[\mathsf D_i]$ and/or $\mathcal S_i$ as in the definitions above, is equal to the smallest triangulated category generated by $E$,  we replace $\mathsf D_i$ and/or $\mathcal S_i$ by $E$ in the notation.
\end{remark}

\begin{definition}
 A sequence of exceptional objects, $E_1, \ldots, E_m$, in a triangulated category, $\mathcal T$, is called a \newterm{full exceptional collection} if
 \[
\mathcal T =  \langle E_1, \ldots, E_m \rangle
 \]
 is a semi-orthogonal decomposition. It is called a \newterm{full strong exceptional collection} if it is a full exceptional collection and
 \[
  \op{Hom}(E_i, E_j [s]) = 0
 \]
 for all $s \neq 0$ and all $i,j$.

 A sequence of exceptional objects, $E_1, \ldots, E_m$, in a pre-triangulated dg-category $\mathsf C$ is called a \newterm{full exceptional collection} if $E_1, \ldots, E_m$ is a full exceptional collection in $[\mathsf C]$. Similarly, $E_1,\ldots,E_m$ in $\mathsf C$ is called \newterm{strong} if it is strong in $[\mathsf C]$.
\end{definition}

Now, we record how Morita products interact with semi-orthogonal decompositions.

\begin{lemma} \label{lem:decomposition of products}
 Let $\mathsf C$ and $\mathsf D$ be small dg-categories. Assume that $\mathsf C$ is pre-triangulated and there is a semi-orthogonal decomposition,
 \begin{displaymath}
  \mathsf C \simeq \langle \mathsf A, \mathsf B \rangle.
 \end{displaymath}
 Then, there exists a semi-orthogonal decomposition,
 \begin{displaymath}
  \mathsf C \circledast \mathsf D \simeq \langle \mathsf A \circledast \mathsf D, \mathsf B \circledast \mathsf D \rangle.
 \end{displaymath}
\end{lemma}

\begin{proof}
 We can assume that $\mathsf{A}$ and $\mathsf{B}$ are subcategories of $\mathsf{C}$ as this only changes $\mathsf{A} \circledast \mathsf{D}$ and $\mathsf{B} \circledast \mathsf{D}$ by quasi-equivalences.

 Let $\langle [\mathsf A \otimes_k \mathsf D] \rangle$ be the smallest thick category containing $[\mathsf A \otimes_k \mathsf D]$ in $[\mathsf C \circledast \mathsf D]$. Let $\langle \mathsf A \otimes_k \mathsf D \rangle$ be the full dg-subcategory of $\mathsf C \circledast \mathsf D$ consisting of objects from $\langle [\mathsf A \otimes_k \mathsf D] \rangle$. Define $\langle [\mathsf B \otimes_k \mathsf D] \rangle$ and $\langle \mathsf B \otimes_k \mathsf D \rangle$ similarly. Since $[\mathsf A] \subset [\mathsf B]^{\perp}$, it is immediate from the definition of tensor product that
\begin{displaymath}
 [\mathsf A \otimes_k \mathsf D] \subset [\mathsf B \otimes_k \mathsf D]^{\perp}.
\end{displaymath}
 Since this property is closed under taking triangles and summands, we also have
\begin{displaymath}
\langle [\mathsf A \otimes_k \mathsf D] \rangle \subset  \langle [\mathsf B \otimes_k \mathsf D] \rangle^{\perp} .
\end{displaymath}
 Since $[\mathsf A]$ and $[\mathsf B]$ generate $[\mathsf C]$, $[\mathsf A \otimes_k \mathsf D]$ and $[\mathsf B \otimes_k \mathsf D]$ generate $[\mathsf C \otimes_k \mathsf D]$ and hence $[\mathsf C \circledast \mathsf D]$. Thus, we have a semi-orthogonal decomposition,
\begin{displaymath}
 [\mathsf C \circledast \mathsf D] \cong \langle \langle [\mathsf A \otimes_k \mathsf D] \rangle, \langle [\mathsf B \otimes_k \mathsf D] \rangle \rangle.
\end{displaymath}

 To finish, we need to check that $\langle \mathsf A \otimes_k \mathsf D \rangle$ is quasi-equivalent to $[\mathsf A \circledast \mathsf D]$ and $\langle \mathsf B \otimes_k \mathsf D \rangle$ is quasi-equivalent to $[\mathsf  B \circledast \mathsf D]$. From the inclusion, $\mathsf A \otimes_k \mathsf D \subset \langle \mathsf A \otimes_k \mathsf D \rangle$, we get a functor,
\begin{displaymath}
 a: \langle \mathsf A \otimes_k \mathsf D \rangle \to (\mathsf A \otimes_k \mathsf D)^{\op{op}}\op{-Mod}.
\end{displaymath}
 We first check that any object in the essential image of $a$ is quasi-isomorphic to an object of $\mathsf{A} \circledast \mathsf{D}$. This is true for any object of $\mathsf{A} \otimes_k \mathsf{D}$ and the subcategory of objects it is true for is thick. Thus, by definition, it is true for all objects of $\langle \mathsf A \otimes_k \mathsf D \rangle$. Note that, by a similar argument, $[a]$ is fully-faithful.
 
 Since $[\langle \mathsf A \otimes_k \mathsf B \rangle]$ is the smallest thick triangulated subcategory containing $[\mathsf A \otimes_k \mathsf B]$, the essential image of $[a]$ is also contained in, hence equivalent to, $\op{perf} (\mathsf A \otimes_k \mathsf B)$. Consequently, we may apply Lemma~\ref{lemma: alternative characterization of Morita product} to conclude that there is a quasi-equivalence
 \begin{displaymath}
  \langle \mathsf A \otimes_k \mathsf B \rangle \simeq \mathsf A \circledast \mathsf B.
 \end{displaymath}
 A similar statement holds for $\mathsf B$. This finishes the proof.
\end{proof}

\section{A theorem of Orlov and a geometric description of the functor category} \label{sec: Orlov}

\subsection{A theorem of Orlov} \label{sec:begin Orlov}

In \cite{Orl09}, Orlov demonstrates a beautiful correspondence between derived categories of coherent sheaves and triangulated  categories of singularities involving semi-orthogonal decompositions. By work of C\u{a}ld\u{a}raru and Tu, or alternatively work of Orlov and Lunts, Orlov's equivalence also generalizes to a statement on the level of dg-categories, see \cite{CT, LO}. Some details of the dg-version are contained in \cite[Theorem 6.13]{BFK11}.

In what follows we will need a very slight generalization, from $\Z$-graded rings to $A$-graded rings, where $A$ is a finitely generated Abelian group of rank one. We first recall the following notion.

\begin{definition} \label{defn: Gorenstein}
 Let $R$ be a finitely-generated $A$-graded $k$-algebra with a unique homogeneous maximal ideal, $\mathfrak{m}$. We say that $R$ is \newterm{Gorenstein} if $R$ has finite $A$-graded injective dimension and there is an element $\eta_{R,A} \in A$ such that
 \begin{displaymath}
  \op{Ext}^i_{\op{Mod}_A R}(R/\mathfrak{m},R(\eta_{R,A})) = \begin{cases} R/\mathfrak{m} & i= \op{dim }R \\ 0 & \text{otherwise.} \end{cases}
 \end{displaymath}
 In other words, the derived dual of $k =  R/\mathfrak{m}$, $k^{\vee}$, is quasi-isomorphic to $k(\eta_{R,A})[-\op{dim} R]$. The element, $\eta_{R,A}$, is called the \newterm{Gorenstein parameter} of $R$. We will simply denote $\eta_{R,A}$ by $\eta$ when $R$ and $A$ are clear from the context.
\end{definition}

Fix
\begin{displaymath}
 S := k[x_1, \ldots, x_n]
\end{displaymath}
and correspondingly $n := \op{dim} S$.

Let $\mathbf{e}_1, \ldots, \mathbf{e}_n$ be the standard basis for the Abelian group, $\Z^n$. The polynomial ring, $S$, has a natural $\Z^n$-grading where $x_i$ has degree $\mathbf{e}_i$. Let $L$ denote a subgroup of rank $n-1$ and consider the quotient map,
\[
 \pi: \Z^n \to \Z^n/L =: A.
\]
Then $A$ is a finitely generated Abelian group of rank one generated by the cosets of the $\mathbf{e}_i$, which we denote by,
\[
 a_i := \pi(\mathbf{e}_i).
\]

Fix an isomorphism, $\phi: A/A_{\op{tors}} \tilde{\longrightarrow} \Z$.  This induces a degree map,
\begin{displaymath}
 \op{deg}: A \to \Z.
\end{displaymath}

\begin{definition}
 Assume that $A$ is a rank one finitely-generated Abelian group as above. If, with our choice of isomorphism $\phi$, $S$ is positively $\Z$-graded, i.e.~that $\op{deg }a_i > 0$ for all $i$, we say $S$ is \newterm{positively $A$-graded} or that \newterm{$A$ grades $S$ positively}.
\end{definition}

This assumption assures that $S$, and any quotient of $S$ by a homogeneous ideal, has a unique $A$-graded maximal ideal corresponding to the origin.

\begin{lemma} \label{lem:parameter}
 Assume $A$ grades $S$ positively and let $w \in S_d$ for $\op{deg} d > 0$. The $A$-graded algebra, $R := S/(w)$, is Gorenstein with Gorenstein parameter,
 \[\eta = -d + \sum_{i=1}^{n}  a_i. \]
\end{lemma}

\begin{proof}
 By Lemma~\ref{lemma: proj equiv = equiv + proj} any object of $\op{Mod}_A R$ with finite injective dimension as an ungraded module has finite injective dimension as an $A$-graded module. It is well-known that $R$ has finite injective dimension as an ungraded module so $R$ has finite injective dimension as an $A$-graded module.

 Assume that $w \not = 0$. It is a simple exercise to verify that there are natural isomorphisms,
\begin{displaymath}
 \op{Ext}^{n+1}_{\op{Mod}_A R}(k,S(-d)) \cong \op{Ext}^{n}_{\op{Mod}_A R}(k,R),
\end{displaymath}
 see \cite[Lemma 18.2.a]{Mat89} for the ungraded case. Thus,
\begin{displaymath}
 \eta_{R,A} = \eta_{S,A} - d.
\end{displaymath}
 We reduce to computing the Gorenstein parameter for $S$, i.e. the case $w = 0$. Since the $x_i$ are $A$-homogeneous in $S$ by assumption, we may view the Koszul complex,
 \begin{displaymath}
  K(x_1, \ldots, x_n) = \bigotimes_{i=1}^n (S(-a_i) \overset{x_i}{\to} S),
 \end{displaymath}
 as an $A$-graded projective resolution of $k$ as an $S$-module.  Hence,
 \begin{align*}
  k^{\vee} & \cong  K(x_1, \ldots, x_n)^{\vee} \cong \bigotimes_{i=1}^n (S(-a_i) \overset{x_i}{\to} S)^{\vee}  \\
                 & \cong  \bigotimes_{i=1}^n (S \overset{x_i}{\to} S(a_i))[-1] \cong (\bigotimes_{i=1}^n (S(-a_i) \overset{x_i}{\to} S))(\sum_{i=1}^n a_i)[-n]  \\
                 & \cong  K(x_1, \ldots, x_n)(\sum_{i=1}^n a_i)[-n]) \cong k(\sum_{i=1}^n a_i)[-n]
 \end{align*}
 So, the Gorenstein parameter of $S$ is $\sum_{i=1}^n a_i$.
\end{proof}

\begin{definition} \label{defn: stack}
 Given an $A$-homogeneous element, $w \in S$, let $Y_w := \text{Spec }S/(w)$. The action of $G(A)$ restricts to an action on $Y_w$ and we define a global quotient stack by,
 \begin{equation} \label{eq: definition of Z}
  Z_{(w,A)} := [(Y_w \backslash 0 )/ G(A)].
 \end{equation}
 We shall often denote it simply by $Z$ when $w$ and $A$ are clear from the context.

 We let $\mathsf{Inj}_{\op{coh}}(Z)$ denote the dg-category of bounded below complexes of injective quasi-coherent $G(A)$-equivariant sheaves with bounded and coherent cohomology on $Y_w \backslash 0$.
\end{definition}

In what follows, we will need some distinguished objects in $\mathsf{Inj}_{\op{coh}}(Z)$ and $\mathsf{fact}(S,A,w)$ respectively. Abusing notation, let $\mathcal O_Z(a)$ be a choice of injective resolution of $\mathcal O_Z(a)$. Similarly, in $\mathsf{fact}(S,A,w)$ consider the factorization:
\begin{displaymath} \label{eq: k}
\begin{tikzpicture}[description/.style={fill=white,inner sep=2pt}]
\matrix (m) [matrix of math nodes, row sep=3em, column sep=3em, text height=1.5ex, text depth=0.25ex]
{  0 & k(a) \\ };
\path[->,font=\scriptsize]
(m-1-1) edge[out=30,in=160] node[above] {$0$} (m-1-2)
(m-1-2) edge[out=200, in=330] node[below] {$0$} (m-1-1);
\end{tikzpicture}
\end{displaymath}
where $k$ denotes the $A$-graded $S/(w)$-module, $S/(x_1, \ldots, x_n)$. Choose a factorization in $\mathsf{proj}(S,A,w)$ quasi-isomorphic to $k(a)$. Once again, abusing notation a bit, we will denote this factorization simply by $k(a)$.

\begin{theorem}[Orlov] \label{thm:Orlov generalized}
 Let $S$ be positively graded by $A$ as constructed above and let $w \in S_d$ be a nonzero homogeneous element with $\op{deg} d > 0$. Assume that $Y_w \setminus 0$ is smooth. With the notation above, let $\mu$ denote the degree of the Gorenstein parameter of $S/(w)$,
 \[
 \mu = \op{deg }\eta = \op{deg }(-d+\sum_{i=1}^n a_i).
 \]
 \renewcommand{\labelenumi}{\emph{\roman{enumi})}}
 \begin{enumerate}
  \item If $\mu > 0$, there is a semi-orthogonal decomposition,
 \begin{displaymath}
  \mathsf{Inj}_{\op{coh}}(Z) \simeq \left\langle \bigoplus_{\op{deg }a = -\mu} \O_Z(a),\ldots,\bigoplus_{\op{deg }a = -1} \O_Z(a), \proj{S}{A}{w} \right\rangle.
 \end{displaymath}
 \item If $\mu = 0$, there is a quasi-equivalence of dg-categories,
 \begin{displaymath}
  \mathsf{Inj}_{\op{coh}}(Z) \simeq \proj{S}{A}{w}.
 \end{displaymath}
 \item If $\mu < 0$, there is a semi-orthogonal decomposition,
 \begin{displaymath}
  \proj{A}{M}{w} \simeq \left\langle \bigoplus_{\op{deg }a = -\mu-1}k(a),\ldots,\bigoplus_{\op{deg }a = 0}k(a), \mathsf{Inj}_{\op{coh}}(Z) \right\rangle.
 \end{displaymath}
 \end{enumerate}
 In particular, if $\mu > 0$, $\O_Z(a) \in \mathsf{Inj}_{\op{coh}}(Z)$ is exceptional for all $a \in A$ and if $\mu < 0$, $k(a) \in \proj{S}{A}{w}$ is exceptional for all $a \in A$.
\end{theorem}

\begin{proof}
 Applying Proposition~\ref{proposition: Qcoh-G = Mod-hat-G} allows us to appeal to \cite[Theorem 6.13]{BFK11} giving all desired statements except using injective equivariant factorizations. However, \cite[Proposition 5.11 and Corollary 5.12]{BFK11} give an quasi-equivalence between dg-categories of projective and injective equivariant factorizations in the setting of a reductive group acting on an affine variety.
\end{proof}

\begin{corollary} \label{corollary: thickness for Orlov components}
 Let $S$ be positively graded by $A$ as constructed above and let $w \in S_d$ be a nonzero homogeneous element with $\op{deg} d > 0$. Assume that $Y_w \setminus 0$ is smooth. Then, $\mathsf{proj}(S,A,w)$ is pre-thick and the inclusion
 \begin{displaymath}
  \mathsf{proj}(S,A,w) \to \overline{\mathsf{proj}}(S,A,w)
 \end{displaymath}
 is a quasi-equivalence.
\end{corollary}

\begin{proof}
 By Theorem~\ref{thm:Orlov generalized}, $[\mathsf{proj}(S,A,w)]$ is either a semi-orthogonal component of an idempotent-complete triangulated category or it admits a semi-orthogonal decomposition whose components are idempotent-complete. By \cite[Lemma 2.2.6]{BDFIK13}, $[\mathsf{proj}(S,A,w)]$ is idempotent-complete and, consequently, $\mathsf{proj}(S,A,w)$ is pre-thick. Since $\mathsf{proj}(S,A,w)$ is pre-thick, the inclusion
 \begin{displaymath}
  \mathsf{proj}(S,A,w) \to \overline{\mathsf{proj}}(S,A,w)
 \end{displaymath}
 must be a quasi-equivalence.
\end{proof}

We now wish to combine this result with Corollary~\ref{cor: Morita product of MF}. With the notation as before, consider a collection of $A_i$-graded polynomial rings, $S_i$, together with $A_i$-homogeneous functions $w_i \in (S_i)_{d_i}$, and the corresponding stacks, $Z_i := Z_{(w_i, A_i)}$  for $1 \leq i \leq t$. Write $S_i = k[x_{i1},\ldots,k_{in_i}]$ where $x_{ij} \in (S_i)_{a_{ij}}$ for $a_{ij} \in A_i$. 

Recall that for two Abelian groups $A,A^{\prime}$ together with a choice of non-torsion elements $d \in A, d^{\prime} \in A^{\prime}$, we set $A \boxminus A^{\prime} = A \oplus A^{\prime}/(d,-d^{\prime})$.  Notice that this forms an associative product on the set of isomorphism classes of (rank one) Abelian groups together with a chosen non-torsion element.  Let $A := A_1 \boxminus \cdots \boxminus A_t$. For each $A_i$, we have a homomorphism given by the composition of the natural maps,
\begin{equation} \label{eq: phi}
 \phi_i : A_i \to \bigoplus_i A_i \to  A
\end{equation}
Now consider two rank one Abelian groups, $A,A^{\prime}$, together with degree maps,  $\op{deg}_{A}: A \to \Z$ and $\op{deg}_{A^{\prime}}: A^{\prime} \to \Z$ coming from isomorphisms $A/ A_{\op{tors}} \overset{\sim}{\to} \Z$ and $A^{\prime}/ A^{\prime}_{\op{tors}} \overset{\sim}{\to} \Z$ respectively.  In addition, suppose we have a choice of non-torsion elements, elements $d \in A, d^{\prime} \in A$. We define a degree map, $\op{deg}_{\boxminus}: A \boxminus A^{\prime} \to \Z$ as follows,
\begin{equation} \label{degree formula}
 \op{deg}_{\boxminus}((a,a^{\prime})) =  \frac{\op{deg}_{A^{\prime}}(d^{\prime})\op{deg}_A(a) + \op{deg}_A(d)\op{deg}_{A^{\prime}}(a^{\prime})}{\op{gcd}(\op{deg}_A(d),\op{deg}_{A^{\prime}}(d^{\prime}))}.
 \end{equation}
where $\op{gcd}(\op{deg}_A(d),\op{deg}_{A^{\prime}}(d^{\prime}))$ is the greatest common divisor of $\op{deg}_A(d)$ and $\op{deg}_{A^{\prime}}(d^{\prime})$.

\begin{remark}
 When the context is clear, we remove the subscripts, $A$, $A^{\prime}$, and $\boxminus$, from the notation above.
\end{remark}

\begin{lemma} \label{lem: positively graded}
 Assume that $S_i$ is an $A_i$-positively graded polynomial ring for $1 \leq i \leq t$. For each $i$, let $d_i \in A_i$ be a non-torsion element of positive degree. Set
 \[
  A := A_1 \boxminus \cdots \boxminus A_t.
 \]
 As an $A$-graded ring, $S_1 \otimes \cdots \otimes S_t$ is positively graded.
\end{lemma}

\begin{proof}
 It suffices to prove the case where $t=2$.  Let $x_1, \ldots, x_{n_1}$ be $A_1$-homogeneous generators of $S_1$ and $y_1, \ldots, y_{n_2}$ be $A_2$-homogeneous generators of $S_2$.  For each $i$, by Equation~\eqref{degree formula} we have,
 \[
  \op{deg}_{\boxminus}(x_i) = \frac{\op{deg}(d_1)\op{deg}(x_i)}{\op{gcd}(d_1,d_2)}.
 \]
 This is a product of positive numbers, hence $\op{deg}_{\boxminus}(x_i)$ is positive for all $i$.  Similarly, this holds for the $y_j$ for all $j$.
\end{proof}

Recall that for two small dg-categories, $\mathsf C$ and $\mathsf D$, over $k$, we have the Morita product, $\mathsf C \circledast \mathsf D$. This product fits perfectly with the usual notion of Cartesian product for schemes.  Indeed, iterating Theorem~\ref{theorem: Morita product of varieties}, we have
\begin{equation} \label {eq:geometric product}
 \mathsf{Inj}_{\op{coh}}(Z_1) \circledast \cdots \circledast \mathsf{Inj}_{\op{coh}}(Z_t) \simeq \mathsf{Inj}_{\op{coh}}(Z_1 \times \cdots \times Z_t).
\end{equation}
We have also analyzed how the operation $\circledast$ behaves for categories of graded factorizations. Assuming that each $A_i$ grades $S_i$ positively, iterating Corollary~\ref{cor: Morita product of MF} and using Lemma~\ref{lemma: Morita product and idempotent completion} and Corollary~\ref{corollary: thickness for Orlov components}, we obtain a quasi-equivalence,
\begin{equation} \label{eq:matrix product}
 \proj{S_1 \otimes_k \cdots \otimes_k S_t}{A}{w_1\boxplus \cdots  \boxplus w_t} \simeq \proj{S_1}{A_1}{w_1} \circledast \cdots  \circledast \proj{S_t}{A_t}{w_t},
\end{equation}
where $A = A_1 \boxminus \cdots \boxminus A_t$.
Applying Theorem~\ref{thm:Orlov generalized} on the left and right, yields new comparisons.

First let us gather some notation. Writing $c = \sum \phi_i(c_i)$, for some $c_i \in A_i$,  let,
 \begin{equation} \label{eq: products of k}
 k^{\otimes t}(c) := k(\phi_1(c_1)) \otimes \cdots \otimes k (\phi_s(c_s)) \in \proj{S_1}{A_1}{w_1} \circledast \cdots  \circledast \proj{S_t}{A_t}{w_t}.
 \end{equation}
To simplify notation in what follows let,
\[
 Z := Z_{(w_1 \boxplus \cdots\boxplus w_t, A)}.
\]

\begin{corollary} \label{cor:MF products}
 Assume that $A$ positively grades $S_1 \otimes \cdots \otimes S_t$. With the notation above, let
 \[
 \mu := \op{deg}(-d_1 + \sum_{i=1}^t\sum_{j=1}^{\op{dim} S_i} \phi_i(a_{ij})).
 \]

 \renewcommand{\labelenumi}{\emph{\roman{enumi})}}
 \begin{enumerate}
  \item If $\mu > 0$, there is a semi-orthogonal decomposition,
  \begin{displaymath}
  \mathsf{Inj}_{\op{coh}} (Z) \simeq \left \langle \bigoplus_{\op{deg }a = -\mu} \O_Z(a), \ldots , \bigoplus_{\op{deg }a = -1} \O_Z(a), \proj{S_1}{A_1}{w_1} \circledast \cdots \circledast \proj{S_t}{A_t}{w_t} \right \rangle.
  \end{displaymath}
  \item If $\mu = 0$, there is a quasi equivalence of dg-categories,
  \begin{displaymath}
   \mathsf{Inj}_{\op{coh}} (Z) \simeq \proj{S_1}{A_1}{w_1} \circledast \cdots \circledast \proj{S_t}{A_t}{w_t}.
  \end{displaymath}
  \item If $\mu < 0$, there is a semi-orthogonal decomposition,
  \begin{displaymath}
   \proj{S_1}{A_1}{w_1} \circledast \cdots \circledast \proj{S_t}{A_t}{w_t} \simeq \left\langle \bigoplus_{\op{deg }c = -\mu -1}k^{\otimes t}(c),\ldots,\bigoplus_{\op{deg }c = 0}k^{\otimes t}(c),\mathsf{Inj}_{\op{coh}}(Z) \right\rangle.
  \end{displaymath}
 \end{enumerate}
\end{corollary}

\begin{proof}
 We have an isomorphism
 \begin{displaymath}
  S_1 \otimes_k \cdots \otimes_k S_t \cong k[x_{11},x_{12},\ldots,x_{1 n_1},\ldots,x_{t 1},\ldots,x_{t n_t}]
 \end{displaymath}
 which is compatible with the $A$-grading if we equip $x_{ij}$ with $A$-degree $\phi_i(a_{ij})$. Applying Lemma~\ref{lem:parameter}, one sees that the Gorenstein parameter of $S_1 \otimes_k \cdots  \otimes_k S_t/ (w_1 \boxplus \cdots \boxplus w_t)$ in $A$ is $-d_1 + \sum_{i=1}^t \sum_{j=1}^{n_i} \phi_i(a_{ij})$.

 We may then appeal to Theorem~\ref{thm:Orlov generalized}. Note the inclusion
 \begin{displaymath}
  \mathsf{proj}(S_1 \otimes_k \cdots \otimes_k S_t,A,w_1 \boxplus \cdots \boxplus w_t) \to \overline{\mathsf{proj}}(S_1 \otimes_k \cdots \otimes_k S_t,A,w_1 \boxplus \cdots \boxplus w_t)
 \end{displaymath}
 is a quasi-equivalence by Corollary~\ref{corollary: thickness for Orlov components}.

 As in the discussion above, $A := A_1 \boxminus \cdots \boxminus A_t$ was constructed so that it is precisely the same group obtained by iterated application of Corollary~\ref{cor: Morita product of MF}, yielding an equivalence of dg-categories,
 \begin{displaymath}
 \barproj{S_1 \otimes_k \cdots  \otimes_k S_t}{A}{w_1 \boxplus \cdots \boxplus w_t} \cong \proj{S_1}{A_1}{w_1} \circledast \cdots  \circledast \proj{S_t}{A_t}{w_t}.
 \end{displaymath}
 One easily verifies that this equivalence takes $k(c)$ to $k^{\otimes t}(c)$.
\end{proof}

Many new semi-orthogonal decompositions can be prepared following the recipe below:
\begin{enumerate}

\item Take a semi-orthogonal decomposition of $\dbcoh{Z}$ with $\mathsf{proj}(S,A,w)$ as a component for $\mu \geq 0$.

\item If $(S,A,w)$ splits, realize $\mathsf{proj}(S,A,w)$ as a $\circledast$-product of several $\mathsf{proj}(S_i,A_i,w_i)$ with $\mu_i < 0$.

\item Take the semi-orthogonal decomposition of $\mathsf{proj}(S_i,A_i,w_i)$.

\item Induce a semi-orthogonal decomposition of the $\circledast$-product.
\end{enumerate}

Let us provide a first simple example with a promise of more to come in later sections.

\begin{example}
Let $d,e,m,n$ be integers such that:
\begin{enumerate}
\item $d  \geq n$,
\item $e  \geq m$,
\item $\op{gcd}(d,e)=1$,
\item $de \leq dm + en$.
\end{enumerate}
Let $w$ define a smooth (Calabi-Yau or general type) hypersurface, $Z_w \subseteq \P^{n-1}$, of degree $d$ and $v$ define a smooth  (Calabi-Yau or general type) hypersurface, $Z_v \subseteq \P^{m-1}$, of degree $e$.  Then $w+v$ defines a smooth hypersurface  (Fano or Calabi-Yau) of degree $de$,
\[
Z_{w+v} \subseteq \P(\underbrace{d: \cdots :d}_{m-\op{times}}:\underbrace{e: \cdots : e}_{n-\op{times}}).
\]
Let $S$ be the coordinate ring of $\P^{n-1}$ and $R$ be the coordinate ring of $\P^{m-1}$.
By Theorem~\ref{thm:Orlov generalized}, we have semi-orthogonal decompositions,
\begin{equation} \label{eq: Zf}
\proj{S}{\Z}{w} \simeq \langle k(-n+d-1), \ldots , k, \inj{Z_w}  \rangle ,
\end{equation}
and
\begin{equation} \label{eq: Zg}
\proj{R}{\Z}{v} \simeq \langle k(-m+e-1), \ldots, k, \inj{Z_v} \rangle,
\end{equation}
and
\begin{equation} \label{eq: Zfg}
\inj{Z_{w+v}} \simeq \langle \O_{Z_{w+v}}(de-dm - en), \ldots, \O_{Z_{w+v}}(-1), \proj{S \otimes_k R}{\Z}{w \boxplus v}   \rangle.
\end{equation}
On the other hand, applying Lemma~\ref{lem:decomposition of products} to Equations~\eqref{eq: Zf} and \eqref{eq: Zg} gives a semi-orthogonal decomposition:
\begin{gather*}
\proj{S}{\Z}{w} \circledast \proj{R}{\Z}{v} \simeq \langle k(-n+d-1) \otimes_k k(-m+e-1), \ldots , k \otimes_k k, \\
k(-n+d-1) \otimes_k \inj{Z_v}, \ldots, \inj{Z_w} \otimes_k k, \inj{Z_w} \circledast \inj{Z_v} \rangle.
\end{gather*}
Applying Equation \eqref{eq:matrix product} to the left hand side and Equation \eqref{eq:geometric product} to the right hand side we get,
\begin{gather*}
\proj{S\otimes_k R}{\Z}{w} \simeq \langle k(-n+d-1, -m+e-1), \ldots, k, \\
k(-n+d-1) \otimes_k \inj{Z_v}, \ldots, \inj{Z_w} \otimes_k k, \inj{Z_w \times Z_v} \rangle.
\end{gather*}
Substituting from Equation~\eqref{eq: Zfg} yields,
\begin{gather*}
\inj{Z_{w+v}} \simeq \langle \O_{Z_{w+v}}(de-dm - en), \ldots, \O_{Z_{w+v}}(-1), k(-n+d-1, -m+e-1), \ldots, k, \\
k(-n+d-1) \otimes_k \inj{Z_v}, \ldots, \inj{Z_w} \otimes_k k, \inj{Z_w \times Z_v} \rangle
\end{gather*}
where $k(a,b)$ is identified with an object of $\inj{Z_{w+v}}$ under the functor,
\[
\proj{S \otimes_k R}{\Z}{w \boxplus v} \to \inj{Z_{w+v}},
\]
as are $k(-n+d-1) \otimes_k \inj{Z_v}, \ldots, \inj{Z_w} \otimes_k k, \inj{Z_w \times Z_v}$. In conclusion, $\inj{Z_w \times Z_v}$ is an admissible subcategory of $\inj{Z_{w+v}}$ whose complement consists of exceptional objects, copies of $\inj{Z_w}$, and copies of $\inj{Z_v}$.
\end{example}

\subsection{Weighted Fermat hypersurfaces, weighted projective lines, and ADE quivers}

Perhaps the simplest class of factorizations to study are those associated to hypersurface weighted projective lines, in the sense of \cite{GL}. These examples are very similar to the ones found in \cite{KST}, especially Appendix A1 due to K. Ueda.  The main difference is that we do not reduce the  grading group to $\Z$ and look only at categories of factorizations as opposed to more general categories of singularities.

Rather than address the topic of weighted projective spaces alone, we also include in this section a discussion on weighted Fermat hypersurfaces which will prepare us for our computations of Rouquier dimension of the derived category of coherent sheaves on a weighted Fermat hypersurface in Section~\ref{subsec: Fermat}.

\begin{definition} \label{defn: weight sequence}
 A \newterm{weight sequence} is an $(n+1)$-tuple of $\mathbf{d} = (d_0, \ldots , d_n)$ of integers with $d_i \geq 1$.  We denote the concatenation of two weight sequences by
 \[
  \mathbf{d} \coprod \mathbf{d}^\prime := (d_0, \ldots, d_n, d_0^\prime, \ldots, d_{n^\prime}^\prime).
 \]
\end{definition}

To each weight sequence, $(d_0, \ldots ,d_n)$ we attach a finitely-generated Abelian group,
\begin{equation} \label{eq: weight group}
 B_{\mathbf{d}} := \Z \boxminus \cdots \boxminus \Z = \Z^{\oplus n+1} / ( d_i \mathbf{e}_i - d_j \mathbf{e}_j ).
\end{equation}
The Abelian affine algebraic group,
\[
 G(B_{\mathbf{d}}):=\{(\lambda_0, \ldots, \lambda_n ) \mid \lambda_i^{d_i} = \lambda_j^{d_j} \} \subseteq \mathbb{G}_{m}^{n+1},
\]
acts on $\mathbb A^{n+1} = \text{Spec }k[x_0, \ldots ,x_n]$ via multiplication by $\lambda_i$ on $x_i$. Consider the following subgroup of $G(B_{\mathbf{d}})$,
\begin{equation} \label{eq: minimal dual group}
 \smallfermatgroup := \{(\lambda_0, \ldots, \lambda_n ) \mid \lambda_i^{\frac{d_i}{\op{gcd}(d_i,d_j)}} = \lambda_j^{\frac{d_j}{\op{gcd}(d_i,d_j)}} \} \subseteq \mathbb{G}_{m}^{n+1},
\end{equation}

\begin{definition} \label{def: Fermat}
 Let $Y_{\sum x_i^{d_i}} \subset \mathbb{A}^{n+1}$ be the zero locus of $\sum x_i^{d_i}$. For any group,
 \[
  \smallfermatgroup \subseteq \intermediatefermatgroup \subseteq G(B_{\mathbf{d}}),
 \]
  the \newterm{weighted Fermat hypersurface} associated to a weight sequence, $\mathbf{d}$, and the group, $\intermediatefermatgroup$, is the global quotient stack,
 \[
  Z_{({\mathbf{d}}, \intermediatefermatgroup)} :=  [(Y_{\sum x_i^{d_i}} \backslash 0) / \intermediatefermatgroup].
 \]
  The \newterm{minimally weighted Fermat hypersurface} associated to a weight sequence, $\mathbf{d}$, is the global quotient stack,
 \[
  Z_{({\mathbf{d}}, \op{min})} := Z_{({\mathbf{d}}, \smallfermatgroup)}  :=  [(Y_{\sum x_i^{d_i}} \backslash 0) / \smallfermatgroup].
 \]
 The \newterm{maximally weighted Fermat hypersurface} associated to a weight sequence, $\mathbf{d}$, is the global quotient stack,
 \[
  Z_{\mathbf{d}} :=  [(Y_{\sum x_i^{d_i}} \backslash 0) / G(B_{\mathbf{d}})].
 \]
\end{definition}

\begin{remark}
 We will focus primarily on the maximally weighted Fermat hypersurfaces. This is why we simplify the notation for the maximally weighted Fermat hypersurface by writing, $Z_{\mathbf{d}}$, as opposed to $Z_{(\mathbf{d}, \op{max})}$ or $Z_{(\mathbf{d}, G(B_{\mathbf{d}}))}$.
\end{remark}

\renewcommand{\thefigure}{\thesection.\arabic{figure}}

\begin{figure}
 \begin{tikzpicture}
  \node[circle,fill=black,inner sep=0.1em,minimum size=5pt,label={[label distance=0.1em]90:$1$}] (a) at (-6,0) {};
  \node[circle,fill=black,inner sep=0.1em,minimum size=5pt,label={[label distance=0.1em]90:$2$}] (b) at (-4,0) {};
  \node[circle,fill=black,inner sep=0.1em,minimum size=5pt,label={[label distance=0.1em]90:$3$}] (c) at (-2,0) {};
  \node (d) at (0,0) {$\cdots$};
  \node[circle,fill=black,inner sep=0.1em,minimum size=5pt,label={[label distance=0.1em]90:$n-2$}] (e) at (2,0) {};
  \node[circle,fill=black,inner sep=0.1em,minimum size=5pt,label={[label distance=0.1em]90:$n-1$}] (f) at (4,0) {};
  \node[circle,fill=black,inner sep=0.1em,minimum size=5pt,label={[label distance=0.1em]90:$n$}] (g) at (6,0) {};
  \draw[->,thick] (a) -- (b);
  \draw[->,thick] (b) -- (c);
  \draw[->,thick] (c) -- (d);
  \draw[->,thick] (d) -- (e);
  \draw[->,thick] (e) -- (f);
  \draw[->,thick] (f) -- (g);
 \end{tikzpicture}

 \begin{tikzpicture}
  \node[circle,fill=black,inner sep=0.1em,minimum size=5pt,label={[label distance=0.1em]90:$1$}] (a) at (-6,0) {};
  \node[circle,fill=black,inner sep=0.1em,minimum size=5pt,label={[label distance=0.1em]90:$2$}] (b) at (-4,0) {};
  \node[circle,fill=black,inner sep=0.1em,minimum size=5pt,label={[label distance=0.1em]90:$3$}] (c) at (-2,0) {};
  \node (d) at (0,0) {$\cdots$};
  \node[circle,fill=black,inner sep=0.1em,minimum size=5pt,label={[label distance=0.1em]90:$n-3$}] (e) at (2,0) {};
  \node[circle,fill=black,inner sep=0.1em,minimum size=5pt,label={[label distance=0.1em]90:$n-2$}] (f) at (4,0) {};
  \node[circle,fill=black,inner sep=0.1em,minimum size=5pt,label={[label distance=0.1em]90:$n-1$}] (g) at (6,1) {};
  \node[circle,fill=black,inner sep=0.1em,minimum size=5pt,label={[label distance=0.1em]270:$n$}] (h) at (6,-1) {};
  \draw[->,thick] (a) -- (b);
  \draw[->,thick] (b) -- (c);
  \draw[->,thick] (c) -- (d);
  \draw[->,thick] (d) -- (e);
  \draw[->,thick] (e) -- (f);
  \draw[->,thick] (f) -- (g);
  \draw[->,thick] (f) -- (h);
 \end{tikzpicture}

 \begin{tikzpicture}
  \node[circle,fill=black,inner sep=0.1em,minimum size=5pt,label={[label distance=0.1em]90:$1$}] (a) at (-6,0) {};
  \node[circle,fill=black,inner sep=0.1em,minimum size=5pt,label={[label distance=0.1em]90:$2$}] (b) at (-4,0) {};
  \node[circle,fill=black,inner sep=0.1em,minimum size=5pt,label={[label distance=0.1em]270:$3$}] (c) at (-2,0) {};
  \node[circle,fill=black,inner sep=0.1em,minimum size=5pt,label={[label distance=0.1em]90:$4$}] (c') at (-2,2) {};
  \node (d) at (0,0) {$\cdots$};
  \node[circle,fill=black,inner sep=0.1em,minimum size=5pt,label={[label distance=0.1em]90:$n-2$}] (e) at (2,0) {};
  \node[circle,fill=black,inner sep=0.1em,minimum size=5pt,label={[label distance=0.1em]90:$n-1$}] (f) at (4,0) {};
  \node[circle,fill=black,inner sep=0.1em,minimum size=5pt,label={[label distance=0.1em]90:$n$}] (g) at (6,0) {};
  \draw[->,thick] (a) -- (b);
  \draw[->,thick] (b) -- (c);
  \draw[->,thick] (c) -- (d);
  \draw[->,thick] (c) -- (c');
  \draw[->,thick] (d) -- (e);
  \draw[->,thick] (e) -- (f);
  \draw[->,thick] (f) -- (g);
 \end{tikzpicture}
\caption{$A_n,D_n,E_n$ quivers, respectively}
\label{fig: ADE}
\end{figure}
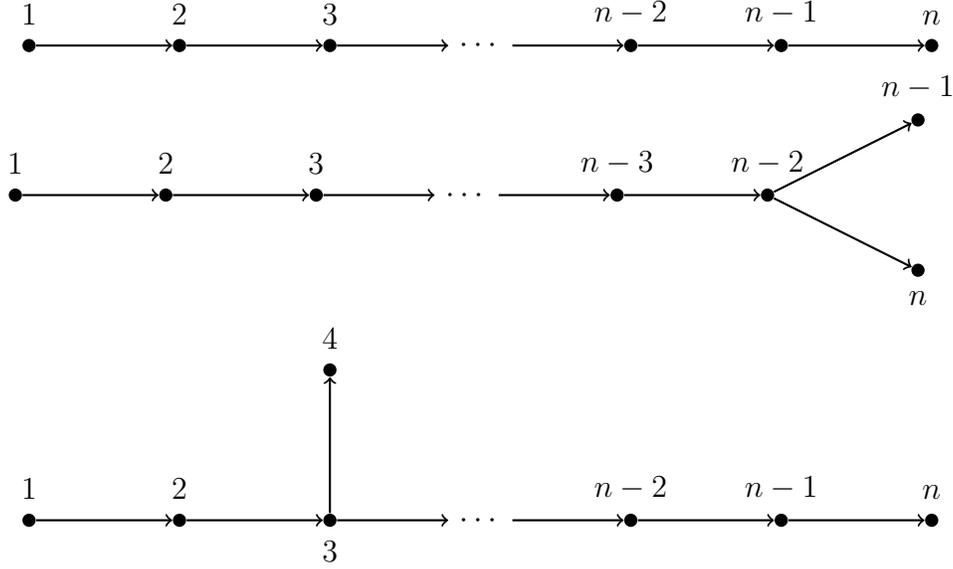

For a Dynkin diagram of type ADE, we can consider the quivers pictured in Figure~\ref{fig: ADE}. By abuse of notation, we denote the path algebra of these quivers by $A_n$, $D_n$, and $E_n$ as well.  For example, we write $\op{mod }A_{n}$ as the category of left modules over the path algebra of the quiver whose underlying graph is $A_{n}$.  Similarly, we write $\mathsf{proj}(A_{n})$ to denote the dg-category of bounded complexes of projective left modules over the path algebra of this particular quiver whose underlying graph is the Dynkin diagram, $A_{n}$.

\begin{remark} \label{rem: all ADE same}
 The choices of quivers above are a bit arbitrary. However, the derived category of modules over the path algebra of each of quivers does not depend on the choice of orientation of arrows. More precisely, the path algebra of any quiver whose underlying graph is an $A$, $D$, or $E$ Dynkin diagram, respectively, is derived equivalent, respectively, to our $A$, $D$, or $E$ choice by \cite[Theorem 5.12]{Hap87}.
\end{remark}

By \cite[Corollary 2.9]{Orl09} and its proof, we have a full strong exceptional collection
\begin{displaymath}
 [\proj{k[x]}{\Z}{x^d}] \cong \langle E_{d-2}, \ldots, E_{0} \rangle
\end{displaymath}
where $R = k[x]/(x^d)$ and $E_i$ is a projective resolution of the factorization
\begin{displaymath}
 (0,R/(R_{\geq i+1})(d-1),0,0)
\end{displaymath}
and
\begin{displaymath}
 \op{Hom}(\oplus_i E_i, \oplus_i E_i) \cong A_{d-1}.
\end{displaymath}
This induces a quasi-equivalence
\begin{displaymath}
 \mathsf{proj}(k[x],\Z,x^d) \simeq \mathsf{proj}(A_{d-1}).
\end{displaymath}
Note that under this quasi-equivalence the simple modules are exactly the projective resolutions of the factorizations corresponding to the $\Z$-graded $R$-modules
\begin{displaymath}
 k(1),\ldots,k(d-1).
\end{displaymath}

By Corollary~\ref{corollary: thickness for Orlov components}, Corollary~\ref{cor: Morita product of MF} and Lemma~\ref{lemma: Morita product of algebras}, we have quasi-equivalences,
\begin{align} \label{eq: type A}
 \mathsf{proj}(k[x_0,\ldots,x_n],B_{\mathbf{d}},\sum x^{d_i})& \simeq \barproj{k[x_0, \ldots , x_n]}{B_{\mathbf{d}}}{\sum x^{d_i}} \notag \\
 & \simeq \circledast_{i=0}^{n} \proj{k[x_i]}{\Z}{x^{d_i}} \notag \\
 & \simeq \circledast_{i=0}^{n} \mathsf{proj}(A_{d_i-1}) \notag \\
 & \simeq \mathsf{proj}(\bigotimes_{i=0}^{n} A_{d_i-1}).
\end{align}

\begin{definition} \label{definition: sign of weight sequence}
 Let $(d_0, \ldots , d_n)$ be a weight sequence.  We say that $(d_0, \ldots , d_n)$ is \newterm{negative, zero, positive, nonnegative,} or \newterm{nonpositive}  if
 \begin{displaymath}
  \bar{\mu}(d_0, ..., d_n):= -1 + \sum_{i=0}^n \frac{1}{d_i}
 \end{displaymath}
 is negative, zero, positive, nonnegative, or nonpositive, respectively.
\end{definition}

\begin{definition} \label{defn: certain modules}
 For the $A_n$ quiver, as above, let $S_i$ be the simple module corresponding to vertex $i$.
 Given a weight sequence, $(d_0, \ldots, d_n)$, and an element $b = \sum \phi_i(b_i) \in  B_{\mathbf{d}}$. Let
 \[
  S_b^{\prime} := S^{\prime}_{b_0} \ok \cdots \ok S^{\prime}_{b_n},
 \]
 where
 \begin{displaymath}
  S^{\prime}_{b_i} = \begin{cases}
                   S_{b_i} & b_i \in \{1,\ldots,d_i-1\} \subset \Z \\
                   (L_{S^{\prime}_{b_i+1}} \circ \cdots \circ L_{S^{\prime}_{b_i+d_i-2}}) S^{\prime}_{b_i+d_i-1} & b_i < 1 \\
                   (R_{S^{\prime}_{b_i-1}} \circ \cdots \circ R_{S^{\prime}_{b_i - d_i +2}} )S^{\prime}_{b_i - d_i+1} & b_i > d_i-1. \\
                  \end{cases}
 \end{displaymath}
 where $L$ denotes left mutation and $R$ denotes right mutation, \cite{Bon}.
\end{definition}

\begin{remark} \label{remark: k and simple modules}
 Under the quasi-equivalences of Equation~\eqref{eq: type A}, the factorization $k^{\otimes t}(b)$ and the module $S_b^{\prime}$ correspond. The exterior product of resolutions of $k(i)$'s corresponds to their tensor product over $k$ as objects of the singularity category by \cite[Lemma 3.66]{BFK11}. The equivalence of Lemma~\ref{lemma: Morita product of algebras} is induced by taking tensor products over $k$ of modules. Thus, the correspondence between $S_b^{\prime}$ and the factorization $k(b)$ in a single variable provides the correspondence in $n+1$-variables.
\end{remark}

We now have the following special case of Corollary~\ref{cor:MF products},
\begin{corollary}  \label{cor: fermat comparison}
Let $\mathbf{d} = (d_0, \ldots, d_n)$ be a weight sequence and set,
\[
\mu = \op{lcm}(d_0, \ldots, d_n)\left(-1 + \sum_{i=0}^n \frac{1}{d_i}\right).
\]

\renewcommand{\labelenumi}{\emph{\roman{enumi})}}
\begin{enumerate}
 \item If $\mu > 0$ (or equivalently $\mathbf{d}$ is positive), there is a semi-orthogonal decomposition,
  \begin{displaymath}
   \mathsf{Inj}_{\op{coh}}(Z_{\mathbf{d}}) \simeq \left \langle \bigoplus_{\op{deg }b= -\mu} \O_Z(b), \ldots , \bigoplus_{\op{deg }b = -1} \O_Z(b), \mathsf{proj}(\bigotimes_{i=0}^{n} A_{d_i-1}) \right \rangle.
  \end{displaymath}
 \item If $\mu = 0$ (or equivalently $\mathbf{d}$ is zero), there is a quasi equivalence of dg-categories,
  \begin{displaymath}\
   \mathsf{Inj}_{\op{coh}}(Z_{\mathbf{d}}) \simeq \mathsf{proj}(\bigotimes_{i=0}^{n} A_{d_i-1}).
  \end{displaymath}
 \item If $\mu < 0$ (or equivalently $\mathbf{d}$ is negative), there is a semi-orthogonal decomposition,
  \begin{displaymath}
   \mathsf{proj}(\bigotimes_{i=0}^{n} A_{d_i-1}) \simeq \left \langle \bigoplus_{\op{deg }b=-\mu-1}S_b,\ldots,\bigoplus_{\op{deg }b = 0}S_b, \mathsf{Inj}_{\op{coh}}(Z_{\mathbf{d}}) \right\rangle.
  \end{displaymath}
 \end{enumerate}
\end{corollary}

\begin{proof}
 We need to check that the statement of Corollary~\ref{cor:MF products} can be translated into the description above.  Recall that
 \[
  B_{\mathbf{d}}:= \Z \boxminus \cdots \boxminus \Z = \Z^{\oplus n+1} / ( d_i \mathbf{e}_i - d_j \mathbf{e}_j ),
 \]

 By Lemma~\ref{lem:parameter}, the Gorenstein parameter for $B_{\mathbf{d}}$ is
 \[
  \eta = -d_0 \mathbf{e}_0 + \sum_{i=0}^n \mathbf{e}_i.
 \]
 In this setting, the degree map, $\text{deg}: B_{\mathbf{d}} \to \Z$, takes the basis element for $\Z^{n+1}$, $\mathbf{e}_i$, to $\frac{\op{lcm}(d_0, ..., d_n)}{d_i}$.  Hence the degree of the Gorenstein parameter is,
 \begin{equation} \label{eq: hypersurface parameter}
  \mu = -d_0 \cdot \frac{\op{lcm}(d_0, ..., d_n)}{d_0} + \sum_{i=0}^n \frac{\op{lcm}(d_0, \ldots, d_n)}{d_i} = \op{lcm}(d_0, \ldots , d_n)\left(-1 + \sum_{i=0}^n \frac{1}{d_i}\right).
 \end{equation}

 We apply now Corollary~\ref{cor:MF products} and replace $ \circledast_{i=0}^{n} \proj{k[x_i]}{\Z}{x_i^{d_i}}$ by $\mathsf{proj}(\bigotimes_{i=0}^{n} A_{d_i-1})$ using the equivalences of Equation~\eqref{eq: type A}.

 It only remains to verify that the complementary pieces in the case $\mathbf{d}$ is negative are as claimed. But, this is Remark~\ref{remark: k and simple modules}.
\end{proof}

\begin{remark} \label{rem: number of objects}
A basic algebra computation reveals that,
\[
| (B_{\mathbf{d}})_{\op{tors}} | = \frac{d_0 \cdots d_n}{\op{lcm}(d_0, \ldots, d_n)}.
\]
As the degree map is, by definition, the map,
\[
B \to B/B_{\op{tors}},
\]
the number of elements of $A$ of a given degree is just $ \frac{d_0 \cdots d_n}{\op{lcm}(d_0, \ldots, d_n)}$.  Hence, the orthogonal to either $\mathsf{Inj}_{\op{coh}}(Z_{\mathbf{d}}) $ in case i) or $\mathsf{proj}(\bigotimes_{i=0}^{n} A_{d_i-1})$ in case iii) has exactly
\[
| \mu | \cdot  \frac{d_0 \cdots d_n}{\op{lcm}(d_0, \ldots, d_n)} =  \left(\prod_{i=0}^n d_i\right) \Biggl| -1 + \sum_{i=0}^n \frac{1}{d_i} \Biggr|
\]
exceptional objects.
\end{remark}

\begin{proposition} \label{prop: Fermat exceptional}
 For any nonnegative weight sequence, $\mathbf{d} = (d_0, \ldots , d_n)$, the category, $\mathsf{Inj}_{\op{coh}}(Z_{\mathbf{d}})$ has a full exceptional collection consisting of
 \[
  \prod_{i=0}^n (d_i - 1)  + \left(\prod_{i=0}^n d_i\right)\left(-1 + \sum_{i=0}^n \frac{1}{d_i}\right)
 \]
 exceptional objects.
\end{proposition}

\begin{proof}
 Each $\mathsf{proj}(A_{d_i-1})$ has a full (strong) exceptional collection of length $d_i-1$. By Lemma~\ref{lem:decomposition of products} applied iteratively, $\circledast_{i=0}^n \mathsf{proj}(A_{d_i-1})$ has a full (strong) exceptional collection of length
 \[
  \prod_{i=0}^n (d_i - 1)
 \]
 objects.
 By Corollary~\ref{cor: fermat comparison} and Remark~\ref{rem: number of objects} it follows that $\mathsf{Inj}_{\op{coh}}(Z_{\mathbf{d}})$ has a full exceptional collection consisting of
 \[
  \prod_{i=0}^n (d_i - 1) + \left(\prod_{i=0}^n d_i \right) \left( -1 + \sum_{i=0}^n \frac{1}{d_i}\right)
 \]
 objects.
\end{proof}

\begin{example}
 For the weight sequence, $\mathbf{d} = (3,3)$, the associated maximally weighted Fermat hypersurface is a point, $\op{Spec}k$ (the quotient of three points by a free and transitive $\Z/3\Z$-action).  Here one obtains a semi-orthogonal decomposition,
 \[
  \mathsf{proj}(A_2 \otimes_k A_2) \simeq \langle \bigoplus_{\op{deg }b =0} E(b),  \dbcoh{\op{Spec}k}  \rangle.
 \]
 There are $3$ elements of degree zero yielding four exceptional objects. Computing the endomorphism algebra of the sum of these four objects, one obtains the path algebra of a quiver whose underlying graph is $D_4$.  Therefore,
 \[
  \mathsf{proj}(A_2 \otimes_k A_2) \simeq \mathsf{proj}(D_4).
 \]
 See also Equation~\eqref{eq: D_4} below.
\end{example}

In \cite{MT}, M. Marcolli and G. Tabuada prove the following theorem.

\begin{theorem} \label{thm: motives}
 Let $X$ be a smooth projective Deligne-Mumford stack and suppose that $\dbcoh{X}$ has a full exceptional collection with $N$ objects.  Then, the rational Chow motive of $X$ is a direct sum of $N$-copies of various tensor powers of the Lefschetz motive.
\end{theorem}

\begin{proof}
 This is \cite[Theorem 1.3]{MT}.
\end{proof}

\begin{corollary}
 For a nonnegative weight sequence, $\mathbf{d}=(d_0, \ldots, d_n)$, the rational Chow motive of the associated maximally weighted Fermat hypersurface, $Z_{\mathbf{d}}$, is a direct sum of  $\prod_{i=0}^n (d_i - 1)  + (\prod_{i=0}^n d_i) (-1 + \sum_{i=0}^n \frac{1}{d_i})$ copies of various tensor powers of the Lefschetz motive.
\end{corollary}

\begin{proof}
 This follows directly from Proposition~\ref{prop: Fermat exceptional} and Theorem~\ref{thm: motives}.
\end{proof}

\begin{definition}
 A \newterm{hypersurface weighted projective line}, as introduced in \cite{GL}, is nothing more than a maximally weighted Fermat hypersurface with $n=3$. In other words, they are the maximally weighted Fermat hypersurfaces, $Z_{(d_0,d_1,d_2)}$, associated to triples, $(d_0,d_1,d_2)$.
\end{definition}

In \cite{GL}, W. Geigle and H. Lenzing show that $\dbcoh{Z_{(d_0,d_1,d_2)}}$ has a full strong exceptional collection consisting of the line bundles, $\O(\phi_i(b))$ with $ 1 \leq b \leq d_i$ for each $0 \leq i \leq 2$ together with $\O$ and the pullback of $\O(1)$ (which can be written $\O(d_0,0,0) = \O(0,d_1,0) = \O(0,0,d_2)$). Here $\phi_i$ is as in  \eqref{eq: phi}.  The endomorphism algebra of the sum of these line bundles is the path algebra of the quiver in Figure~\ref{fig: canonical quiver}.  Removing the left-most vertex, which corresponds to $\O$, and right-most vertex, which corresponds to $\O(1)$, the lengths of the chains of arrows in the three rows are $d_0-1, d_1-1$, and $d_2-1$.

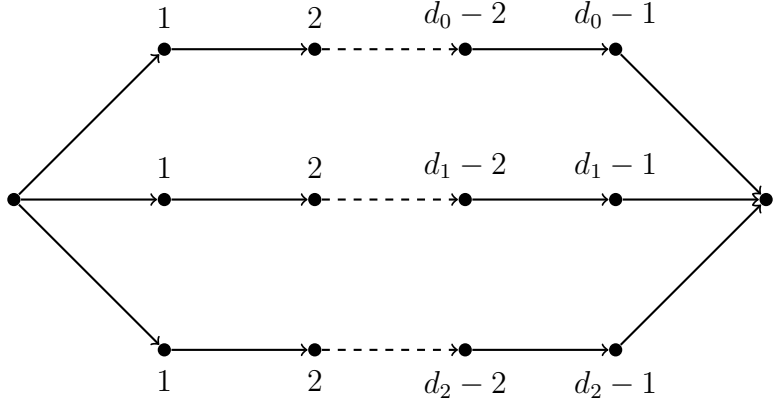
\begin{figure}
 \begin{tikzpicture}
  \node[circle,fill=black,inner sep=0pt,minimum size=5pt] (begin) at (-5,0) {};
  \node[circle,fill=black,inner sep=0pt,minimum size=5pt,label={[label distance=0.1em]90:$1$}] (a1) at (-3,2) {};
  \node[circle,fill=black,inner sep=0pt,minimum size=5pt,label={[label distance=0.1em]90:$2$}] (a2) at (-1,2) {};
  \node[circle,fill=black,inner sep=0pt,minimum size=5pt,label={[label distance=0.1em]90:$d_0-2$}] (a3) at (1,2) {};
  \node[circle,fill=black,inner sep=0pt,minimum size=5pt,label={[label distance=0.1em]90:$d_0-1$}] (a4) at (3,2) {};
  \node[circle,fill=black,inner sep=0pt,minimum size=5pt,label={[label distance=0.1em]90:$1$}] (b1) at (-3,0) {};
  \node[circle,fill=black,inner sep=0pt,minimum size=5pt,label={[label distance=0.1em]90:$2$}] (b2) at (-1,0) {};
  \node[circle,fill=black,inner sep=0pt,minimum size=5pt,label={[label distance=0.1em]90:$d_1-2$}] (b3) at (1,0) {};
  \node[circle,fill=black,inner sep=0pt,minimum size=5pt,label={[label distance=0.1em]90:$d_1-1$}] (b4) at (3,0) {};
  \node[circle,fill=black,inner sep=0pt,minimum size=5pt,label={[label distance=0.1em]270:$1$}] (c1) at (-3,-2) {};
  \node[circle,fill=black,inner sep=0pt,minimum size=5pt,label={[label distance=0.1em]270:$2$}] (c2) at (-1,-2) {};
  \node[circle,fill=black,inner sep=0pt,minimum size=5pt,label={[label distance=0.1em]270:$d_2-2$}] (c3) at (1,-2) {};
  \node[circle,fill=black,inner sep=0pt,minimum size=5pt,label={[label distance=0.1em]270:$d_2-1$}] (c4) at (3,-2) {};
  \node[circle,fill=black,inner sep=0pt,minimum size=5pt] (end) at (5,0) {};
  \draw[->,thick] (begin) -- (a1);
  \draw[->,thick] (begin) -- (b1);
  \draw[->,thick] (begin) -- (c1);
  \draw[->,thick] (a1) -- (a2);
  \draw[->,dashed,thick] (a2) -- (a3);
  \draw[->,thick] (a3) -- (a4);
  \draw[->,thick] (b1) -- (b2);
  \draw[->,dashed,thick] (b2) -- (b3);
  \draw[->,thick] (b3) -- (b4);
  \draw[->,thick] (c1) -- (c2);
  \draw[->,dashed,thick] (c2) -- (c3);
  \draw[->,thick] (c3) -- (c4);
  \draw[->,thick] (a4) -- (end);
  \draw[->,thick] (b4) -- (end);
  \draw[->,thick] (c4) -- (end);
 \end{tikzpicture}
\caption{Canonical quiver for $(d_0,d_1,d_2)$}
\label{fig: canonical quiver}
\end{figure}

\begin{remark}
The full strong exceptional collection mentioned above due to Geigle and Lenzing is not the same as the one chosen in the proof of Proposition~\ref{prop: Fermat exceptional}.
\end{remark}

For hypersurface weighted projective lines, we are able to compare the categories, $\mathsf{proj}(A_{d_i})$ and $\mathsf{Inj}_{\op{coh}}(Z_{(d_0,d_1,d_2)})$.  The degree of the Gorenstein parameter for the coordinate ring of $Z_{(d_0,d_1,d_2)}$ is
\begin{displaymath}
 \mu =\op{lcm}(d_0,d_1,d_2)\left(\frac{1}{d_0}+\frac{1}{d_1} + \frac{1}{d_2} -1\right).
\end{displaymath}
As we will need it for further use, we treat the case where $\mu >0$.  In this case, the weight sequence $(d_0,d_1,d_2)$ is said to be of \newterm{Dynkin type} and we have the following semi-orthogonal decomposition of $\dbcoh{Z_{(d_0,d_1,d_2)}}$,
\begin{equation} \label{eq: weighted projective line decomp}
 \mathsf{Inj}_{\op{coh}}(Z_{(d_0,d_1,d_2)}) \simeq \langle \bigoplus_{\op{deg }b = -\mu} \O_Z(b), \ldots , \bigoplus_{\op{deg }b=-1} \O_Z(b), \mathsf{proj}(A_{d_0-1} \ok A_{d_1-1} \ok A_{d_2-1}) \rangle.
\end{equation}

The possible weight sequences of Dynkin type, match up precisely with the simple singularities.  One has, $A_{p+q}$ for  $(1, p, q)$, $D_{l+2}$ for $(2, 2, l)$, $E_6$ for $(2, 3, 3)$, $E_7$ for $(2, 3, 4)$, and $E_8$ for $(2, 3, 5)$.

\begin{remark}
 In \cite{KST}, they show that for every weight sequence $(p,q,r)$ of Dynkin type, there exists a $\Z$-graded subring
 \[
  R \subseteq k[x,y,z] / (x^p+y^q+z^r)
 \]
 such that the $\Z$-graded category of singularities of $R$ is equivalent to the derived category of representations of the corresponding ADE quiver.
\end{remark}

In our setup, given a weight sequence of Dynkin type, $(d_0,d_1,d_2)$, the category,
\begin{displaymath}
\op{D}^{\op{b}}(\op{mod }A_{d_0-1} \otimes_k A_{d_1-1} \otimes_k A_{d_2-1} ),
\end{displaymath}
is an admissible subcategory of the category of representations of the quiver corresponding to the type of simple singularity.  In the case, $(1, p, q)$, the variety, $k[x,y,z]/ (x+ y^p + z^q)$, is smooth and we get the zero category as our category of matrix factorizations.  In the case, $(2, 2, l)$, the category of matrix factorizations is equivalent to:
\begin{displaymath}
\op{D}^{\op{b}}(\op{mod }A_{1} \otimes_k A_{1} \otimes_k A_{l-1} ) \cong \op{D}^{\op{b}}(\op{mod }A_{l-1}).
\end{displaymath}
  In the case, $(2, 3, 3)$, the category of matrix factorizations is equivalent to:
\begin{equation} \label{eq: D_4}
\op{D}^{\op{b}}(\op{mod } A_{1} \otimes_k A_{2} \otimes_k A_{2} )\cong \op{D}^{\op{b}}(\op{mod } A_{2} \otimes_k A_{2}) \cong \op{D}^{\op{b}}(\op{mod } D_4).
\end{equation}
  In the case, $(2, 3, 4)$, the category of matrix factorizations is equivalent to:
\begin{equation} \label{eq: E_6}
\op{D}^{\op{b}}(\op{mod }A_{1} \otimes_k A_{2} \otimes_k A_{3} ) \cong \op{D}^{\op{b}}(\op{mod }A_{2} \otimes_k A_{3} )\cong \op{D}^{\op{b}}(\op{mod }E_6).
\end{equation}
  In the case, $(2, 3, 5)$, the category of matrix factorizations is equivalent to:
\begin{equation} \label{eq: E_8}
\op{D}^{\op{b}}(\op{mod }A_{1} \otimes_k A_{2} \otimes_k A_{4}) \cong \op{D}^{\op{b}}(\op{mod }A_{2} \otimes_k A_{4}) \cong \op{D}^{\op{b}}(\op{mod }E_8).
\end{equation}
These exceptional equivalences seem to be well known, at least the $D_4$ and $E_6$ cases appear in  \cite{Hap}.  For completeness, let us quickly prove Equation~\eqref{eq: E_8}.  Consider the weight sequence $(2,3,5)$.  In this case, Equation~\eqref{eq: weighted projective line decomp} reads:
\[
 \dbcoh{Z_{2,3,5}} = \langle \O, \op{D}^{\op{b}}(\op{mod }A_{2} \otimes_k A_{4}) \rangle.
\]

\begin{figure}
 \begin{tikzpicture}
  \node[circle,fill=black,inner sep=0pt,minimum size=5pt] (begin) at (-5,0) {};
  \node[circle,fill=black,inner sep=0pt,minimum size=5pt] (a1) at (0,2) {};
  \node[circle,fill=black,inner sep=0pt,minimum size=5pt] (b1) at (-1,0) {};
  \node[circle,fill=black,inner sep=0pt,minimum size=5pt] (b2) at (1,0) {};
  \node[circle,fill=black,inner sep=0pt,minimum size=5pt] (c1) at (-3,-2) {};
  \node[circle,fill=black,inner sep=0pt,minimum size=5pt] (c2) at (-1,-2) {};
  \node[circle,fill=black,inner sep=0pt,minimum size=5pt] (c3) at (1,-2) {};
  \node[circle,fill=black,inner sep=0pt,minimum size=5pt] (c4) at (3,-2) {};
  \node[circle,fill=black,inner sep=0pt,minimum size=5pt] (end) at (5,0) {};
  \draw[->,thick] (begin) -- (a1);
  \draw[->,thick] (begin) -- (b1);
  \draw[->,thick] (begin) -- (c1);
  \draw[->,thick] (b1) -- (b2);
  \draw[->,thick] (c1) -- (c2);
  \draw[->,thick] (c2) -- (c3);
  \draw[->,thick] (c3) -- (c4);
  \draw[->,thick] (a1) -- (end);
  \draw[->,thick] (b2) -- (end);
  \draw[->,thick] (c4) -- (end);
 \end{tikzpicture}
\caption{Canonical quiver for $(2,3,5)$}
\label{fig: 235 quiver}
\end{figure}
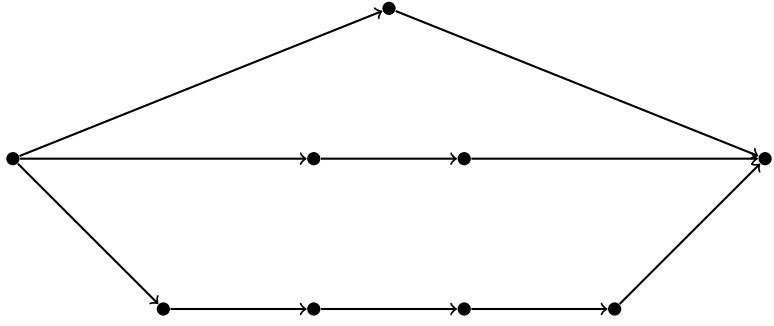

Removing the left-most vertex from Figure~\ref{fig: 235 quiver}, we see that we get a quiver whose underlying graph is of type $E_8$.  As noted in Remark~\ref{rem: all ADE same}, by \cite{Hap87} the path algebra of such a quiver has the same derived category as the path algebra for the quiver of our chosen $E_8$.  This proves Equation~\eqref{eq: E_8}.

\begin{definition} \label{definition: ADE for weight sequence}
We say that a weight sequence is \newterm{ADE} if it coincides with any of the following four sequences up to permutation
\[
 (2,\ldots,2,a), (2,\ldots,2,3,3), (2,\ldots,2,3,4), \tor (2,\ldots,2,3,5)
 \]
  where $2$ can occur with any multiplicity including zero.
\end{definition}

\begin{proposition} \label{prop: is ADE}
Consider a decomposition of a weight sequence, $(d_0, \ldots, d_n)$, into a disjoint union of ADE weight sequences,
\[
\mathbf{d} = \coprod_{i=1}^s \mathbf{d}_i.
\]
Then,
\[
\op{D}^{\op{b}}(\op{mod } \bigotimes_{j=0}^n A_{d_j-1}) \cong \op{D}^{\op{b}}(\op{mod } \bigotimes_{i=1}^{s} kQ_i),
\]
where $Q_i$ is a quiver whose underlying graph is of type ADE corresponding to the ADE type of $\mathbf{d}_i$ for $1 \leq l \leq s$.
\end{proposition}

\begin{proof}
As each weight sequence, $\mathbf{d}_i :=(d_0^i, \ldots , d_{n_i}^i)$ is ADE,  using Equations~\eqref{eq: D_4}, \eqref{eq: E_6}, and \eqref{eq: E_8}, for each $i$ we have
\begin{equation} \label{eq: ADE sequence}
\circledast_{j=0}^{n_i} \mathsf{proj}(A_{d_j^i-1}) \simeq \mathsf{proj}(kQ_i),
\end{equation}
for some $S_i$ which is the path algebra of a quiver whose underlying graph is of type ADE.
Therefore,
\begin{align*}
\circledast_{i=0}^{n} \mathsf{proj}(A_{d_i-1}) \simeq & \circledast_{i=1}^s \circledast_{j=0}^{n_i} \mathsf{proj}(A_{d_j^i-1}) \\
\simeq & \circledast_{i=1}^s  \mathsf{proj}(kQ_i) \\
\simeq & \mathsf{proj}(\bigotimes_{i=1}^s kQ_i),
\end{align*}
where the first equivalence comes from Lemma~\ref{lemma: Morita product of algebras}, the second equivalence comes from Equation~\eqref{eq: ADE sequence}, and the third equivalence again comes from Lemma~\ref{lemma: Morita product of algebras}.
\end{proof}

\begin{proposition} \label{prop: sitting inside ADE}
Let $\mathbf{d} = (d_0, \ldots , d_n)$ be a weight sequence which decomposes into a disjoint union of nonpositive weight sequences,
\[
\mathbf{d} = \coprod_{i=1}^s \mathbf{d}_i.
\]
 Then $\mathsf{Inj}_{\op{coh}}(\prod_{i=1}^s Z_{\mathbf{d}_i})$ is an admissible subcategory of $\mathsf{proj}(\bigotimes_{i=0}^n A_{d_i-1})$. In particular the category, $\dbcoh{\prod_{i=1}^s Z_{\mathbf{d}_i}}$, is an admissible subcategory of
$\op{D}^{\op{b}}(\op{mod }\bigotimes_{i=0}^n A_{d_i-1})$.
\end{proposition}

\begin{proof}
 By Corollary~\ref{cor: fermat comparison} and Equation~\eqref{eq: type A}, $\mathsf{Inj}_{\op{coh}}(Z_{\mathbf{d}_i})$ is an admissible subcategory of  $\mathsf{proj}(\bigotimes_{d_j \in \mathbf{d}_i} A_{d_j-1})$. By Lemma~\ref{lem:decomposition of products}, it follows that
\[
\circledast_{i=1}^s \mathsf{Inj}_{\op{coh}}(Z_{\mathbf{d}_i})
\]
is an admissible subcategory of
\[
\circledast_{i=1}^s \mathsf{proj}(\bigotimes_{d_j \in \mathbf{d}_i} A_{d_j-1})
\]
Now,
\[
\circledast_{i=1}^s \inj{Z_{\mathbf{d}_i}} \simeq \inj{\prod_{i=1}^s Z_{\mathbf{d}_i}}
\]
by Equation~\eqref{eq:geometric product} and
\begin{align*}
 \circledast_{i=1}^s \mathsf{proj}(\bigotimes_{d_j \in \mathbf{d}_i} A_{d_j-1}) \simeq \circledast_{i=0}^n  \mathsf{proj}(A_{d_i-1}) \simeq \mathsf{proj}(\bigotimes_{i=0}^n A_{d_i-1}).
\end{align*}
where both of the quasi-equivalences in the above display come from Lemma~\ref{lemma: Morita product of algebras}.
\end{proof}

\begin{remark}
 We will use the above propositions later to examine the Rouquier dimension of the derived category of modules over tensor products of path algebras of quivers whose underlying graph is of type ADE and the derived category of coherent sheaves on various weighted Fermat hypersurfaces.
\end{remark}

\subsection{Graded Kn\"orrer periodicity and mirrors to products of hypersurfaces} \label{sec: graded knorrer}

We now wish to use the observations of Section~\ref{sec:begin Orlov} to discuss a general picture for creating Landau-Ginzburg mirrors to products of hypersurfaces, $Z_i \subseteq \P^{n_i-1}$, defined by polynomials, $w_i$, of degree $d_i$ for $1 \leq i \leq t$.

For simplicity, we will describe the case where all of the $Z_i$ are either Calabi-Yau or general type.  However, we want to first ensure that when we product these varieties the resulting Gorenstein parameter is positive, as the mirror symmetry constructions in this case have been explored in much greater depth (see, for example, \cite{Bat, Bor, Giv, HV, Cla} and especially \cite{Aur}).  To do so, we simply employ a graded version of Kn\"orrer periodicity.  Namely, it was proven by Kn\"orrer in \cite{Kno} that for a regular ring, $R$, there is an equivalence of categories,
\[
 [\mathsf{proj}(R,w)] \cong [\mathsf{proj}(R[x,y],w+x^2+y^2)].
\]
In the graded case, this manifests similarly by first adding the potential $x^2$ alone.
\begin{proposition}
Let $S = k[x_1, \ldots, x_n]$ be a positively $A$-graded polynomial ring with $x_i \in S_{a_i}$ and $\op{deg }a_i > 0$.  Given $w \in S_d$, we have:
\begin{equation}
\proj{S}{A}{w}  \simeq \proj{S[x]}{A \boxminus \Z}{w+x^2} \label{+x^2}
\end{equation}
and
\begin{equation}\label{+x^2+y^2}
\proj{S}{A}{w} \simeq \proj{S[x,y]}{A \boxminus \Z \boxminus \Z}{w+x^2+y^2},
\end{equation}
\end{proposition}
\begin{proof}
 One can see by inspection that $\proj{k[x]}{\Z}{x^2}$ is quasi-equivalent to the dg-category of chain complexes of vector spaces (see also \cite{Orl09}),
\begin{equation} \label{eq: Knorrer}
\proj{k[x]}{\Z}{x^2} \simeq \mathsf{C}(k)_{\op{pe}}.
\end{equation}

Now, we have:
\begin{align*}
\proj{S}{A}{w} & \simeq \barproj{S}{A}{w}  \\
 & \simeq \proj{k[x]}{\Z}{x^2} \circledast \proj{S}{A}{w}  \\
 & \simeq \proj{S[x]}{A \boxminus \Z)}{w+x^2}.
\end{align*}
The first line is Corollary~\ref{corollary: thickness for Orlov components}.  The second line follows from Equation~\eqref{eq: Knorrer}, Lemma~\ref{lemma: Morita product and idempotent completion} and Corollary~\ref{corollary: thickness for Orlov components} again.  The third line is Corollary~\ref{cor: Morita product of MF} and  Corollary~\ref{corollary: thickness for Orlov components} one final time for good measure.

The aligned display gives Equation~\eqref{+x^2}.  Equation~\eqref{+x^2+y^2} follows from Equation~\eqref{+x^2}.
\end{proof}

\begin{lemma} \label{lem: always positive}
 Assume that $w$ is nonzero. The Gorenstein parameter for the $A \boxminus \Z \boxminus \Z$-graded algebra, $S[x,y]/(w+x^2+y^2)$, always has positive degree.
\end{lemma}

\begin{proof}
 Let $a_{i}$ denote the degree of $x_i$ in $S$. The Gorenstein parameter of $A \boxminus \Z \boxminus \Z$ was calculated in Corollary~\ref{cor:MF products}, with
\[
 \eta =  -d + \sum_{i=1}^{n} (a_i, 0, 0) + (0,1,0) + (0,0,1) = \sum_{i=1}^{n} (a_i, 0, 0) +(0,-1,1) .
\]
The first equality is just writing out the definitions. The second equality comes from the fact that $-d = (0,-2,0)$ in $A \boxminus \Z \boxminus \Z$.  Notice that $(0,-1,1)$ is torsion (of order 2) in $A \boxminus \Z \boxminus \Z$.  Therefore, $(0,-1,1)$ has degree zero.  Hence, the degree of the Gorenstein parameter is equal to
\begin{equation} \label{this parameter}
\sum_{i=1}^{n} \op{deg}((a_i, 0, 0)).
\end{equation}
Now, iterating Equation~\eqref{degree formula} we obtain,
\[
\op{deg}((a_i, 0, 0)) = \frac{4\op{deg}(d)\op{deg}(a_i)}{(\op{gcd}(2,\op{deg}(d))^2}.
\]
As $\op{deg}(d)$ and $\op{deg}(a_i)$ are positive for all $i$ by assumption,
$\op{deg}((a_i, 0, 0))$ are positive as well.  Hence, the sum in \eqref{this parameter} is positive.
\end{proof}

Due to Lemma~\ref{lem: always positive}, we see that by adding a quadratic term as in Equation~\eqref{+x^2+y^2}, we can always ensure that we are in case 1) of Corollary~\ref{cor:MF products}. Therefore, while we may start with some hypersurface of general type, we can nevertheless embed $\mathsf{proj}(S,A,w)$ as an admissible subcategory of the derived category of coherent sheaves on a certain stack.  This should be viewed as a generalization of the conic bundle construction for varieties of general type observed by the third named author and carried to fruition in \cite{Sei08, Efi09, KKOY}.

Specifically, denote by $x_{ij}$ for $1 \leq i \leq n_j$, coordinates on $\P^{n_j-1}$ for $1 \leq j \leq t$.  Suppose $Z_j \subseteq \P^{n_j-1}$ are smooth varieties which are either Calabi-Yau or general type defined respectively by polynomials $w_j$ of degree $d_j$ (the Fano case is not much different but a bit more involved to discuss) in the variables $x_{ij}$.

We can construct a group,
\[
A := \Z \boxminus \cdots \boxminus \Z,
\]
where the product occurs $t+2$ times and the chosen elements are $(d_1, \ldots, d_s, 2, 2)$ and a function,
\[
w := \sum_{j=1}^t w_j(x_{i1}, \ldots , x_{jn_j}) + u^2+v^2.
\]
From Equation~\eqref{eq: definition of Z}, we obtain a stack, $Z_{(w,A)}$.

\begin{proposition}
$\mathsf{Inj}_{\op{coh}}(Z_1 \times \cdots  \times Z_t)$ is an admissible subcategory of $\mathsf{Inj}_{\op{coh}}(Z_{(w,A)})$.
\end{proposition}
\begin{proof}
As each hypersurface, $Z_i$, is of general type, by Theorem~\ref{thm:Orlov generalized}, $\mathsf{Inj}_{\op{coh}}(Z_i)$ is an admissible subcategory of $\mathsf{proj}(S_i, \Z, w_i)$.
Therefore by Lemma~\ref{lem:decomposition of products} combined with Equation \eqref{eq:geometric product}, $\mathsf{Inj}_{\op{coh}}(Z_1 \times \cdots  \times Z_t)$ is a semi-orthogonal component of $\circledast_{i=1}^t \mathsf{proj}(S_i, \Z, w_i)$. By Lemma~\ref{lem: always positive}, Theorem~\ref{thm:Orlov generalized}, and Equation~\eqref{+x^2+y^2},
\[
\circledast_{j=1}^t \mathsf{proj}(S_j, \Z, w_j) \simeq \circledast_{j=1}^s \mathsf{proj}(S_j, \Z, w_j) \circledast \mathsf{proj}(k[u], \Z, u^2) \circledast \mathsf{proj}(k[v], \Z, v^2) \simeq \mathsf{proj}(S,A,w)
\]
is an admissible subcategory of $\mathsf{Inj}_{\op{coh}}(Z_{(w,A)})$.
\end{proof}

\begin{definition} \label{defn: LG}
 A \newterm{Landau-Ginzburg model}, $(X, f)$ is a K\"ahler manifold, $X$, together with a nonconstant holomorphic function,
\[
f: X \to \mathbb A^1_{\C}.
\]
\end{definition}

\begin{conjecture}
 A mirror, in the sense of \cite{Aur}, to the product of general-type hypersurfaces $Z_1 \times \cdots \times Z_t$ can be realized as the Landau-Ginzburg mirror to the Fano DM-stack $Z_{(w,A)}$ with all singular fibers away from the origin removed.
\end{conjecture}

\begin{remark}
The motivation for the above conjecture comes from Homological Mirror Symmetry.
Roughly, Kontsevich's Homological Mirror Symmetry Conjecture suggests that each singular fiber  of $(X, f)$ corresponds to a semi-orthogonal component of $\inj{Z_{(w,A)}}$.  We propose that $\inj{Z_1 \times \cdots  \times Z_t}$ corresponds to a single singular fiber, which can be translated to the origin.
\end{remark}
%
%
%

\begin{remark}
 For examples of constructions of Landau-Ginzburg mirrors to varieties of general type see \cite{KKOY} and \cite{AAK}.
\end{remark}

\begin{remark}
 In physics, the above process of adding a quadratic term is often referred to as adding bosonic mass terms.  A proposal of S. Sethi, \cite{Sethi}, suggests adding fermionic terms as well. This super commutative case can be handled with some slight adjustments and allows one to subtract from the Gorenstein parameter as well. We leave this to future work.
\end{remark}

\section{Different gradings and orbit categories} \label{sec:orbit}

We would now like to establish a geometric picture by varying the grading group for a fixed ring with potential.  We use the notion of an orbit category.

\begin{definition} \label{defn: orbit category}
 Let $\Gamma$ be a finitely generated Abelian subgroup of the autoequivalence group of a $k$-linear category $\mathcal T$, usually triangulated. The \newterm{orbit category} of $\mathcal T$ by $\Gamma$, denoted $\mathcal T/\Gamma$ has the same objects as $\mathcal T$ with morphisms from $t$ to $t^{\prime}$ given by
\begin{displaymath}
 \op{Hom}_{\mathcal T/\Gamma}(t, t^{\prime}) = \bigoplus_{g \in \Gamma}\op{Hom}_{\mathcal T}(t, g(t^{\prime})).
\end{displaymath}
 To compose $f \in \op{Hom}_{\mathcal T}(t,\gamma(t^{\prime}))$ and $f^{\prime} \in \op{Hom}_{\mathcal T}(t^{\prime},\gamma^{\prime}(t^{\prime \prime}))$ one sets
 \begin{displaymath}
  f^{\prime} \circ f := \gamma(f^{\prime}) \circ f.
 \end{displaymath}

 Note that we have a natural functor
 \begin{align*}
  \mathcal T & \to  \mathcal T / \Gamma \\
  t & \mapsto t \\
  (f: t \to t^{\prime}) & \mapsto f \in \op{Hom}_{\mathcal T}(t,t^{\prime}) \subset \op{Hom}_{\mathcal T/\Gamma}(t,t^{\prime}).
 \end{align*}
\end{definition}

\begin{remark}
 In general, the orbit category of a triangulated category $\mathcal T$ is not triangulated. However, in \cite{Kel2}, Keller provides a way of rectifying this when $\mathcal T$ is the homotopy category of a dg-category, $\mathsf{C}$, and $\Gamma$ lifts to an action on $\mathsf{C}$.  In this case, one can take the homotopy category of perfect modules over the dg orbit category to get a triangulated category.
\end{remark}

\begin{definition}
 Let $\mathcal T$ and $\mathcal S$ be triangulated categories and $\Gamma$ be a finitely-generated Abelian group of exact autoequivalences of $\mathcal T$. We say that $\mathcal T$ is a $\Gamma$\newterm{-cover} of $\mathcal S$ if $\mathcal T/\Gamma$ is a full subcategory of $\mathcal S$ such that the composition with the natural functor,
 \begin{displaymath}
  F: \mathcal T \to \mathcal T/\Gamma \to \mathcal S,
 \end{displaymath}
 is exact and every object in $\mathcal S$ is a summand of an object in the essential image of $F$.
\end{definition}

\begin{proposition} \label{prop: spaces as covers}
 Let $G$ be a finite Abelian group acting on a smooth Deligne-Mumford stack, $\mathcal X$, over $k$. Let $\widehat{G}$ act on $\dbcoh{\mathcal X/G}$ by twisting via characters. Then $\dbcoh{\mathcal X}$ is a $\widehat{G}$-cover of $\dbcoh{[\mathcal X/G]}$.
\end{proposition}

\begin{proof}
 From \cite[Section 2]{BFK11}, we have an adjoint pair of exact functors
 \begin{align*}
  \op{Res}^{G}_{1} & : \op{Qcoh}_{G} \mathcal X \to \op{Qcoh} \mathcal X \\
  \op{Ind}^{G}_{1} & : \op{Qcoh} \mathcal X \to \op{Qcoh}_{G} \mathcal X
 \end{align*}
 given by restriction and induction along the inclusion of the trivial subgroup $1 \to G$.
 Furthermore, by \cite[Lemma 2.18]{BFK11}, we have
 \begin{equation}  \label{eq: cover1}
  (\op{Ind}^{G}_{1} \circ \op{Res}^{G}_{1}) \mathcal E \cong \bigoplus_{\chi \in \widehat{G}} \mathcal E(\chi).
 \end{equation}

 Let $\sigma: G \times \mathcal X \to \mathcal X$ denote the action and set $\sigma(g,\bullet) =: \sigma_g$. It is straightforward to see that
 \begin{displaymath}
  \op{Ind}^{G}_{1} \mathcal E \cong \bigoplus_{g \in G} \sigma_g^*\mathcal E.
 \end{displaymath}
 Thus
 \begin{equation} \label{eq: cover2}
  (\op{Res}^{G}_{1} \circ \op{Ind}^{G}_{1}) \mathcal E \cong \bigoplus_{g \in G} \sigma_g^*\mathcal E.
 \end{equation}

 Consider the functor
 \begin{displaymath}
  \op{Res}^G_1 : \dbcoh{[\mathcal X/G]} \to \dbcoh{\mathcal X}.
 \end{displaymath}
 From Equation~\eqref{eq: cover2}, the essential image of $\dbcoh{[\mathcal X/G]}$ is dense.

 Now, for $\mathcal E, \mathcal F \in \dbcoh{\mathcal X}$,
 \begin{align*}
  \op{Hom}_{\dbcoh{[\mathcal X / G]}}(\op{Res}^{G}_{1} \mathcal E, \op{Res}^{G}_{1} \mathcal F))  = & \op{Hom}_{\dbcoh{\mathcal X}}(\mathcal E, (\op{Ind}^{G}_{1} \circ \op{Res}^{G}_{1})\mathcal F) \\
						      = & \op{Hom}_{\dbcoh{\mathcal X}}(\mathcal E, \bigoplus_{\gamma \in \widehat{G}} \mathcal F(\gamma)) \\
						      = & \bigoplus_{\gamma \in \widehat{G}} \op{Hom}_{\dbcoh{\mathcal X}}(E, \mathcal F(\gamma)),
 \end{align*}
 where the first equality comes from adjunction and the second comes from Equation~\eqref{eq: cover1}, and the third uses the finiteness of the coproduct. Thus, the image of $\op{Res}^{G}_{1}$ is the orbit category $\dbcoh{\mathcal X}/\widehat{G}$.
\end{proof}

%


\begin{proposition}\label{prop: MF grading change}
 Let $A$ be a finitely generated Abelian group and $\Gamma$ be a finite subgroup of $A$ and $\pi: A \to A/\Gamma$ be the quotient map. Let $R$ be an $A$-graded polynomial algebra.  Let $w \in R$ be an element which is $A$-homogeneous. The functor,
\[
 \op{R}^A_{A / \Gamma} : \dabs[\mathsf{fact}(R,A,w)] \to \dabs[\mathsf{fact}(R,A/\Gamma,w)],
\]
 is a $\Gamma$-cover.
\end{proposition}

\begin{proof}
 First, we must show that every object $G \in  \dabs[\mathsf{fact}(R,A,w)]$ is a summand of an object in the essential image of $\op{R}^A_{A / \Gamma}$.  Let $G = \op{Spec} k[\Gamma]$. Since $\Gamma$ is finite, it is straightforward to verify that $\op{I}^A_{A / \Gamma} F = \bigoplus_{g \in G} g^* F$. Thus,
 \begin{displaymath}
  (\op{R}^A_{A / \Gamma} \circ \op{I}^A_{A / \Gamma})F \cong \bigoplus_{g \in G} g^* F,
 \end{displaymath}
 and we see that $F$ is a summand of $(\op{R}^A_{A / \Gamma} \circ \op{I}^A_{A / \Gamma})F$.

Now, by Proposition~\ref{proposition: restriction and induction properties}, we have an adjoint pair of exact functors,
 \begin{align*}
  \op{R}^A_{A / \Gamma} & : \op{Mod}_A R \to \op{Mod}_{A/\Gamma} R \\
  \op{I}^A_{A / \Gamma} & : \op{Mod}_{A/\Gamma} R \to \op{Mod}_A R
 \end{align*}
 such that
 \begin{displaymath}
  (\op{I}^A_{A / \Gamma} \circ \op{R}^A_{A / \Gamma}) M \cong \bigoplus_{\gamma \in \Gamma} M(\gamma).
 \end{displaymath}
 The functor, $\op{R}^A_{A / \Gamma}$, preserves finite-generation of graded modules while $\op{I}^A_{A / \Gamma}$ preserves finite-generation since $\Gamma$ is finite. Thus, we also have an adjoint pair of functors on categories of finitely-generated graded modules. Applying the functors component-wise, gives an adjoint pair of dg-functors
 \begin{align*}
  \op{R}^A_{A / \Gamma} & : \mathsf{fact}(R,A,w) \to \mathsf{fact}(R,A/\Gamma,w) \\
  \op{I}^A_{A / \Gamma} & : \mathsf{fact}(R,A/\Gamma,w) \to \mathsf{fact}(R,A,w).
 \end{align*}
 Furthermore, since $\op{R}^A_{A / \Gamma}$ and $\op{I}^A_{A / \Gamma}$ are exact, they preserve acyclic factorizations and descend to an adjoint pair
 \begin{align*}
  \op{R}^A_{A / \Gamma} & : \dabs[\mathsf{fact}(R,A,w)] \to \dabs[\mathsf{fact}(R,A/\Gamma,w)] \\
  \op{I}^A_{A / \Gamma} & : \dabs[\mathsf{fact}(R,A/\Gamma,w)] \to \dabs[\mathsf{fact}(R,A,w)]
 \end{align*}
 such that
 \begin{equation} \label{eq: MF cover}
  (\op{I}^A_{A / \Gamma} \circ \op{R}^A_{A / \Gamma}) F \cong \bigoplus_{\gamma \in \Gamma} F(\gamma)
 \end{equation}
 for any $F \in \dabs[\mathsf{fact}(R,A,w)]$.

 Now, for $E, F \in \dabs[\mathsf{fact}(R,A,w)]$,
\begin{align*}
 \op{Hom}_{ \dabs[\mathsf{fact}(R,A/\Gamma,w)]}(\op{R}^A_{A / \Gamma} E,\op{R}^A_{A / \Gamma} F) = & \op{Hom}_{ \dabs[\mathsf{fact}(R,A,w)]}( E,(\op{I}^A_{A / \Gamma} \circ \op{R}^A_{A / \Gamma}) F) \\
 = & \op{Hom}_{ \dabs[\mathsf{fact}(R,A,w)]}(E, \bigoplus_{\gamma \in \Gamma} F(\gamma)) \\
 =  &\bigoplus_{\gamma \in \Gamma} \op{Hom}_{ \dabs[\mathsf{fact}(R,A,w)]}(E, F(\gamma)).
\end{align*}
 where the first equality comes from adjunction and the second comes from Equation~\eqref{eq: MF cover}, and the third is by finiteness of coproduct.
 Hence, the image of $\op{R}^A_{A / \Gamma}$ is the orbit category.

\end{proof}

\begin{corollary} \label{corollary: gcd cover}
 Let $R$ and $S$ be positively $\Z$-graded polynomial algebras with chosen nonzero elements, $w \in R_d$ and $v \in S_e$ with $\op{deg} d, \op{deg} e > 0$. Equip $R \otimes_k S$ with the diagonal $\Z$-grading. The category,
 \begin{displaymath}
  [\proj{R}{\Z}{w} \circledast \proj{S}{\Z}{v}],
 \end{displaymath}
 is a $\Z/\op{gcd}(d,e)\Z$-cover of
 \[
 \dabs[\mathsf{fact}(R \otimes_k S, \Z, w\boxplus v)].
 \]
\end{corollary}

\begin{proof}
 Let $r = \op{gcd}(d,e)$. Since $R$ and $S$ are positively $\Z$-graded, by Corollary~\ref{cor: Morita product of MF} and Corollary~\ref{corollary: thickness for Orlov components}, $\proj{R}{\Z}{w} \circledast \proj{S}{\Z}{v}$ is quasi-equivalent to $\proj{R \otimes_k S}{\Z \oplus \Z/(d,-e)}{w \boxplus v}$.  The element, $(\frac{d}{r},-\frac{e}{r})$, generates a cyclic subgroup of $\Z \oplus \Z/(d,-e)$ of order $r$.  When one quotients by this cyclic subgroup, it induces the diagonal $\Z$-grading.  Hence by Proposition~\ref{prop: MF grading change} we obtain the result.
\end{proof}

\begin{example} \label{ex: product of elliptic}
 Let $w(x,y,z) = x(x-z)(x - \lambda z)-zy^2$ and $v(a,b,c) =a(a-c)(a-\rho c)-cb^2$ define two smooth elliptic curves, $\Sigma$ and $\Sigma^{\prime}$ respectively.  Then $w \boxplus v$ defines a smooth cubic fourfold containing at least three planes, $P,Q,R$, with $P+Q+R = H^2$ by setting $z=c=0$.

 Hence combining Theorem~\ref{thm:Orlov generalized} with the results in \cite{Kuz09b} on cubics which contain planes, the category $\proj{k[x,y,z,a,b,c]}{\Z}{w \boxplus v}$ is equivalent to the derived category of a gerbe on certain $K3$ surface, $(Y,\beta)$, see \cite{Kuz09b} for details (the gerbe is trivial if and only if there exists a 2-dimensional cycle, $T$, such that $T \cdot H^2 - T \cdot P$ is odd).  On the other hand, letting $A = \Z \oplus \Z/(3,-3)$ with $x,y,z$ in degree $(1,0)$ and $a,b,c$ in degree $(0,1)$, by Corollary~\ref{cor: Morita product of MF} and Corollary~\ref{corollary: thickness for Orlov components}, we have
 \[
  \proj{k[x,y,z,a,b,c]}{A}{w \boxplus v} \simeq \proj{k[x,y,z]}{\Z}{w} \circledast \proj{k[a,b,c]}{\Z}{v}
 \]
 and by Equation~\eqref{eq: Toen product}
 \[
\inj{\Sigma} \circledast \inj{\Sigma'} \simeq \inj{\Sigma \times \Sigma^{\prime}}.
 \]
 From Theorem~\ref{thm:Orlov generalized} we have,
\begin{equation} \label{elliptic equivalence}
 \mathsf{proj}(k[x,y,z],\Z,w) \simeq \inj{\Sigma} \tand \mathsf{proj}(k[a,b,c],\Z,v) \simeq \inj{\Sigma^{\prime}}.
\end{equation}
Hence,
\begin{equation} \label{product of elliptic curves}
\proj{k[x,y,z,a,b,c]}{\Z \boxminus \Z}{w\boxplus v} \simeq \inj{\Sigma \times \Sigma^{\prime}}.
\end{equation}
From Corollary~\ref{corollary: gcd cover}, $\dbcoh{\Sigma \times \Sigma^{\prime}}$ is a $\Z/3\Z$-cover of $\dbcoh{(Y, \beta)}$.
\end{example}

\begin{definition}
 Let $\mathcal T$ be a $k$-linear triangulated category with finite dimensional morphism spaces which possesses a Serre functor, $S$, see \cite{BK}. An object $E \in \mathcal T$ is called \newterm{spherical} \cite{ST} if,
 \begin{itemize}
  \item $S(E) \cong E[n]$
  \item $\op{Hom}_{\mathcal T}(E,E[i]) \cong
   \begin{cases} k &  i=0,n \\
    0 & \text{otherwise.}
   \end{cases}$
 \end{itemize}
 \label{def:spherical object}
\end{definition}

\begin{definition}\label{def:left twist}
 Let $\mathcal T$ be a triangulated category with finite dimensional morphism spaces. Assume there is dg-category $\mathsf{C}$ enhancing $\mathcal T$, $\mathcal T \cong [\mathsf{C}]$. For any pair of objects, $E$ and $F$, of $\mathcal T$, we have a natural evaluation map,
 \begin{displaymath}
  \bigoplus \op{Hom}_{\mathcal T}(E[i],F) \otimes_k E[i] \overset{\op{ev}}{\to} F.
 \end{displaymath}
 Define $L_{E}: \mathcal T \to \mathcal T$ as the functor which takes $F$ to the cone over $\op{ev}$. The functor $L_{E}$ is called the \newterm{left twist} by $E$.  If $E$ is spherical, then $L_{E}$ is called the \newterm{spherical twist} by $E$.
\end{definition}

\begin{remark}
 The assumption that $\mathcal T$ is enhanceable ensures that the description above can be made functorial, see \cite{ST}.
\end{remark}

\begin{theorem}
 Let $\mathcal T$ be a triangulated category with finite dimensional morphism spaces. Assume there is dg-category $\mathsf{C}$ enhancing $\mathcal T$, $\mathcal T \cong [\mathsf{C}]$. For any spherical object, $E \in \mathcal T$, the spherical twist, $L_{E}$, is an exact autoequivalence of $\mathcal T$.
\end{theorem}

\begin{proof}
 This is due to P. Seidel and R. Thomas \cite{ST}.
\end{proof}

\begin{example} \label{example: autoequivalence identification}
Continuing notation from Example~\ref{ex: product of elliptic}, let $\Sigma$ be an elliptic curve viewed as a plane cubic defined by $w$. The objects, $\O$ and $\O_p$, for any $p \in \Sigma$, are spherical in $\dbcoh{\Sigma}$. In fact, there is an isomorphism of autoequivalences,
\begin{equation} \label{twisting by a point}
L_{\O_p} \cong (\O(p) \otimes - ).
\end{equation}
Equation~\eqref{elliptic equivalence} allows us to view the grading shift autoequivalence of $\dabs[\mathsf{fact}(k[x,y,z],\Z,w)] $, $(1)$, as an autoequivalence of $\dbcoh{\Sigma}$ by conjugation.  In fact, as we have seen in \cite[Theorem 6.13]{BFK11},
\begin{equation} \label{grading shift}
\Phi \circ (1) \circ \Phi^{-1} \cong L_{\mathcal O} \circ (\O(1) \otimes - ).
\end{equation}
Intersecting  a generic hyperplane of $\mathbb{P}^2$ with $\Sigma$ we obtain three points, $p,q,r$,and an isomorphism,
\begin{equation} \label{H isomorphism}
\O(1) \cong \O(p) \otimes \O(q) \otimes \O(r).
\end{equation}
Combining equations, \eqref{twisting by a point}, \eqref{grading shift}, and \eqref{H isomorphism}, we obtain,
\begin{align*}
\Phi \circ (1) \circ \Phi^{-1} & = L_{\mathcal O} \circ (\O(1) \otimes - ) \\
& = L_{\O} \circ (\O(p) \otimes \O(q) \otimes \O(r) \otimes - ) \\
& = L_{\O} \circ L_{\O_p} \circ L_{\O_q} \circ L_{\O_r}.
\end{align*}
In conclusion, the autoequivalence, $(1)$, viewed as an autoequivalence of $\dbcoh{\Sigma}$, is a composition of spherical twists.

Assume that $\Sigma$ and $\Sigma^{\prime}$ are two smooth planar cubics defined by $w$ and $v$. On the category $\dbcoh{\Sigma \times \Sigma^{\prime}}$, we have an autoequivalence corresponding to the difference of the two $(1)$. Let $Q^{-1} \in \dbcoh{\Sigma \times \Sigma}$ be the kernel of $\Phi \circ (-1) \circ \Phi^{-1}$ and let $Q^{\prime} \in \dbcoh{\Sigma^{\prime} \times \Sigma^{\prime}}$ be the kernel of $\Phi \circ (1) \circ \Phi^{-1}$ \cite{Orl97}. Set
\begin{displaymath}
 P := \pi_1^* Q^{-1} \otimes \pi_2^* Q
\end{displaymath}
where
\begin{align*}
 \pi_1 & : \Sigma \times \Sigma \times \Sigma^{\prime} \times \Sigma^{\prime} \to \Sigma \times \Sigma \\
 \pi_2 & : \Sigma \times \Sigma \times \Sigma^{\prime} \times \Sigma^{\prime} \to \Sigma^{\prime} \times \Sigma^{\prime}
\end{align*}
are the projections. Then, the integral transform associated to $P$
\begin{displaymath}
 \Phi_P : \dbcoh{\Sigma \times \Sigma^{\prime}} \to \dbcoh{\Sigma \times \Sigma^{\prime}}
\end{displaymath}
satisfies
\begin{displaymath}
 \Phi_P^3 \cong \op{Id}
\end{displaymath}
as $(-3) = [-2]$ and $(3) = [2]$. This gives a $\Z/3\Z$-action on $\dbcoh{\Sigma \times \Sigma^{\prime}}$ that corresponds to the action of $(-1,1)$ on $\dabs[\mathsf{fact}(k[x_1,y_1,z_1,x_2,y_2,z_2], \Z \boxminus \Z, w \boxplus w^{\prime})]$ under the equivalences
\begin{align*}
 \dabs[\mathsf{fact}(k[x_1,y_1,z_1,x_2,y_2,z_2], \Z \boxminus \Z, w \boxplus w^{\prime})] & \cong [\mathsf{proj}(k[x_1,y_1,z_1],\Z,w) \circledast \mathsf{proj}(k[x_2,y_2,z_2],\Z,w^{\prime})] \\
 & \cong [\inj{\Sigma} \circledast \inj{\Sigma^{\prime}}] \\
 & \cong \dbcoh{\Sigma \times \Sigma^{\prime}}
\end{align*}
coming from Proposition~\ref{prop: projective enhancement}, Corollary~\ref{cor: Morita product of MF}, Theorem~\ref{thm:Orlov generalized}, and Theorem~\ref{theorem: Morita product of varieties}.

Corollary~\ref{corollary: gcd cover} says that $\dabs[\mathsf{fact}(k[x_1,y_1,z_1,x_2,y_2,z_2], \Z \boxminus \Z, w \boxplus w^{\prime})]$ is a $\Z/3\Z$-cover of $\dabs[\mathsf{fact}(k[x_1,y_1,z_1,x_2,y_2,z_2], \Z, w \boxplus w^{\prime})]$. Again applying Theorem~\ref{thm:Orlov generalized}, under the equivalences above, $\dbcoh{\Sigma \times \Sigma^{\prime}}$ is a $\Z/3\Z$-cover of $\dbcoh{Y, \beta}$ where $(Y, \beta)$ is the gerby K3 surface from Example~\ref{ex: product of elliptic}.
\end{example}

\begin{remark}
The $\Z/3\Z$-covering relating $\Sigma \times \Sigma^{\prime}$ and $Y$ from Example~\ref{example: autoequivalence identification} can be viewed as coming from a symplectic automorphism. Homological Mirror Symmetry provides exact equivalences
\begin{align*}
 \varphi & : \dbcoh{\Sigma} \to \op{D}(\op{FS}(\widehat{\Sigma})) \\
 \varphi^{\prime} & : \dbcoh{\Sigma^{\prime}} \to \op{D}(\op{FS}(\widehat{\Sigma}^{\prime}))
\end{align*}
between the derived categories of coherent sheaves on $\Sigma$ and $\Sigma^{\prime}$ and the derived Fukaya category of the mirror symplectic tori, $\widehat{\Sigma}$ and $\widehat{\Sigma^{\prime}}$ \cite{PZ}. The spherical objects in $\dbcoh{\Sigma}$ and $\dbcoh{\Sigma^{\prime}}$, give honest Lagrangian spheres (circles in this case) in $\widehat{\Sigma}$  $\widehat{\Sigma^{\prime}}$ (motivating the terminology).

As such, these spherical twists can be viewed as autoequivalences induced by Dehn twists about these circles.  Hence, the composition of these spherical twists, is nothing more than the autoequivalence induced by the symplectic automorphism, $\psi$, consisting of the composition of these Dehn twists. Conjugation by the equivalence, $\varphi \circ \Phi$, allows one to view the autoequivalence, $(1)$, of $\proj{k[x_1,y_1,z_1]}{\Z}{w}$  as an autoequivalence induced by a symplectic automorphism, $\psi$, of $\widehat{\Sigma}$. Similarly, the autoequivalence, $(1)$, of $\proj{k[x_2,y_2,z_2]}{\Z}{w^{\prime}}$, conjugated to become an autoequivalence of $\op{FS}(\widehat{\Sigma^{\prime}})$ is isomorphic to an autoequivalence induced by some symplectic automorphism, $\psi^{\prime}$ of $\widehat{\Sigma^{\prime}}$.

The action in Example~\ref{ex: product of elliptic}, of $\Z/3\Z$ on $\proj{k[x_1,y_1,z_1,x_2,y_2,z_2]}{\Z \oplus \Z/(3,-3)}{w \boxplus v }$ is given by powers of $(-1,1)$.

Now the equivalence,
\[
 \beta: \op{Fuk}(\widehat{\Sigma}) \circledast \op{Fuk}(\widehat{\Sigma^{\prime}}) \to\op{Fuk}(\widehat{\Sigma} \times \widehat{\Sigma^{\prime}}),
\]
is induced by generating $\op{Fuk}(\widehat{\Sigma} \times \widehat{\Sigma^{\prime}})$ using Lagrangian tori \cite{AS}. Furthermore, we know that $(-1,1)$ corresponds to $\psi^{-1} \times \psi^{\prime}$ via conjugation by the composition of equivalences,
\begin{align*}
 \proj{k[x,y,z,u,v,w]}{M}{w \boxplus w^{\prime}}  & \simeq  \proj{k[x,y,z]}{\Z}{w} \circledast \proj{k[u,v,w]}{\Z}{w^{\prime}} \\
  & \overset{\varphi \otimes \varphi^{\prime}}{\simeq} \op{Fuk}(\widehat{\Sigma}) \circledast \op{Fuk}(\widehat{\Sigma^{\prime}}) \\
  & \overset{\beta}{\simeq} \op{Fuk}(\widehat{\Sigma} \times \widehat{\Sigma^{\prime}}).
\end{align*}
In conclusion, taking the orbit category by the $\Z/3\Z$-action of $\Phi_P$ on $\dbcoh{\Sigma \times \Sigma^{\prime}}$ producing the gerby K3 surface $(Y, \beta)$, amounts, under mirror symmetry, to taking a three to one, symplectic quotient in the mirror by the symplectic automorphism, $\psi^{-1} \times \psi^{\prime}$. A similar observation appears in \cite{KP}.
\end{remark}

\begin{example} \label{example: double solid}
 Consider a quartic $K3$ surface, $Y$, defined by $w(x_1,x_2,x_3,x_4)$ in $\P^3$.  Let $Q$ be the (stacky) quartic double solid defined by the equation, $t^2 -w$, in weighted projective space $\P(2 \! : \! 1 \! : \! 1 \! : \! 1 \! : \!1)$. As mentioned above, $\proj{k[t]}{\Z}{t^2}$ is quasi-equivalent to the dg category of complexes of vector spaces.  Hence,
 \begin{align*}
  \dabs[\mathsf{fact}(k[t,x_1,x_2,x_3,x_4],\Z,t^2-w)] &\cong [\proj{k[t]}{\Z}{t^2} \circledast \proj{k[x_1,x_2,x_3,x_4]}{\Z}{w}]\\
  & \cong [\barproj{k[x_1,x_2,x_3,x_4]}{\Z}{w}] \\
  &\cong [\mathsf{Inj}_{\op{coh}}(Y)] \cong \dbcoh{Y}.
 \end{align*}
 The first equivalence is Corollary~\ref{cor: Morita product of MF}, the second equivalence is Equation~\eqref{+x^2}, and the third equivalence is Theorem~\ref{thm:Orlov generalized}.  The final equivalence is standard. Therefore, $\dbcoh{Y}$ is a $\Z/2\Z$-cover of $\dabs[\mathsf{fact}(k[t,x_1,x_2,x_3,x_4],\Z,t^2-w)]$, an admissible subcategory of $\dbcoh{Q}$ by Theorem~\ref{thm:Orlov generalized}.
\end{example}

\begin{remark}
 In \cite{IK}, Ingalls and Kuznetsov prove that for any nodal Enriques surface there is an associated singular quartic surface, such that a certain admissible subcategory of the derived category of coherent sheaves on the Enriques surface is also an admissible subcategory of the derived category of coherent sheaves on a small resolution of a double cover of $\P^3$ ramified along the quartic surface. Ingalls and Kuznetsov's work achieves part of a mirror symmetry based conjecture found in \cite{IKP}.   It would be interesting to compare Example~\ref{example: double solid} to their result.
\end{remark}

\section{Applications to generation time} \label{sec: gt}

\subsection{Preliminaries} \label{subsec: prelim on Rouquier dimension}

We recall the following definitions.  For a more complete treatment see, \cite{Ro2,BFK}. Let $\mathcal T$ be a triangulated category. For a full subcategory, $\mathcal I$, of $\mathcal T$ we denote by $\langle \mathcal I \rangle$ the full subcategory of $\mathcal T$ whose objects are isomorphic to summands of finite coproducts of shifts of objects in $\mathcal I$. In other words, $\langle \mathcal I \rangle$ is the smallest full subcategory containing $\mathcal I$ which is closed under isomorphisms, shifting, and taking finite coproducts and summands. For two full subcategories, $\mathcal I_1$ and $\mathcal I_2$, we denote by $\mathcal I_1 \ast \mathcal I_2$ the full subcategory of objects, $T$, such that there is a distinguished triangle,
\begin{displaymath}
 I_1 \to T \to I_2 \to I_1[1],
\end{displaymath}
with $I_i \in \mathcal I_i$.  Set
\[
\mathcal I_1 \diamond \mathcal I_2 := \langle \mathcal I_1 \ast \mathcal I_2 \rangle,
\]
\[
\langle \mathcal I \rangle_0 :=\langle \mathcal I \rangle,
\]
 and, for $n \geq 1$, inductively define,
\begin{displaymath}
 \langle \mathcal I \rangle_n := \langle \mathcal I \rangle_{n-1} \diamond \langle \mathcal I \rangle.
\end{displaymath}
Similarly we define
\begin{displaymath}
 \langle \mathcal I \rangle_{\infty} := \bigcup_{n \geq 0} \langle \mathcal I \rangle_{n}.
\end{displaymath}
For an object, $E \in \mathcal T$, we notationally identify $E$ with the full subcategory consisting of $E$ in writing, $\langle E \rangle_n$.  

\begin{remark}
 The reader is warned that, in some of  the previous literature, $\langle \mathcal I \rangle_0 := 0$ and $\langle \mathcal I \rangle_1 := \langle \mathcal I \rangle$. We follow the notation in \cite{BF, BFK}.  With our convention, the index equals the number of cones allowed.
\end{remark}

\begin{definition} \label{defn: generation time}
Let $\mathcal I$ be a full subcategory of a triangulated category, $\mathcal{T}$.  If there is an $n$ with $\langle \mathcal I \rangle_{n} = \mathcal T$, we set
\begin{displaymath}
 \tritime_{\mathcal T}(\mathcal I):=  \text{min } \lbrace  n \geq 0  \mid \langle \mathcal I \rangle_{n} = \mathcal T \rbrace.
\end{displaymath}
Otherwise, we set $\tritime_{\mathcal T}(\mathcal I) := \infty$.   We call $\tritime_{\mathcal T}(\mathcal I)$ the \newterm{generation time} of $\mathcal I$. When, $\mathcal T$ is clear from context, we omit it and simply write $\tritime(\mathcal I)$.
\end{definition}

Most commonly we consider the case when $\mathcal I = E$ is a single object.
\begin{definition} \label{defn: OSpec rdim}
Let $E$ be an object of a triangulated category, $\mathcal T$.  If $\langle E \rangle_{\infty}$ equals $\mathcal{T}$, we say that $E$ is a \newterm{generator}. If $\tritime(E)$ is finite, we say that $E$ is a \newterm{strong generator}. The \newterm{Orlov spectrum} of $\mathcal T$, denoted $\op{OSpec}\mathcal T$, is the set,
\begin{displaymath}
 \op{OSpec}\mathcal T := \lbrace \tritime(G)  \mid G \in \mathcal T, \ \tritime(G) < \infty \rbrace \subseteq \mathbb{Z}_{\geq 0}.
\end{displaymath}
The \newterm{Rouquier dimension} of $\mathcal T$, denoted $\op{rdim }\mathcal T$, is the infimum of the set of all generation times,
\[
\op{rdim }\mathcal T := \op{min }\lbrace \tritime(G)   \mid  G \in \mathcal T, \ \tritime(G) < \infty \rbrace.
\]
\end{definition}

The results from the previous sections can be applied to the study of Orlov spectra and Rouquier dimension.  In particular, in some special cases we will get a positive answer to Conjecture~\ref{conj: dimension}.

\begin{definition}
 Let $\mathcal T$ be a triangulated category, $f: X \to Y$ be a morphism, and $\mathcal I$ be a full subcategory.  We say that $f$ is \newterm{$\mathcal I$ ghost} if, for all $I \in \mathcal I$, the induced map,
\begin{displaymath}
 f \circ -: \op{Hom}_{\mathcal T}(I,X) \to \op{Hom}_{\mathcal T}(I,Y),
\end{displaymath}
 is zero.  We say that $f$ is \newterm{$\mathcal I$ co-ghost} if, for all $I \in \mathcal I$, the induced map,
\begin{displaymath}
 - \circ f: \op{Hom}_{\mathcal T}(Y,I) \to \op{Hom}_{\mathcal T}(X,I)
\end{displaymath}
 is zero. Analogously, for an object $G$ of $\mathcal T$, $f$ is \newterm{$G$ ghost} if $f$ is $\langle G \rangle_0$ ghost and $f$ is \newterm{$G$ co-ghost} if $f$ is $\langle G \rangle_0$ co-ghost.
\end{definition}

One common source of ghosts is the following.

\begin{proposition} \label{prop: nat trans as ghosts}
 Assume we have a natural transformation, $\nu: \op{Id} \to F$, for some exact endofunctor of a triangulated category $\mathcal T$. Let $T$ be an object of $\mathcal T$ such that $\nu|_{\langle T \rangle_0} = 0$ for all $i \in \Z$. Then, for any $T^{\prime} \in \mathcal T$, $\nu_{T^{\prime}}: T^{\prime} \to F(T^{\prime})$ is $T$ ghost.
\end{proposition}

\begin{proof}
 Since $\nu$ is natural transformation, for any morphism, $p: P \to T^{\prime}$ with $P \in \langle T \rangle_0$, we have a commutative diagram,
\begin{center}
\begin{tikzpicture}[description/.style={fill=white,inner sep=2pt}]
\matrix (m) [matrix of math nodes, row sep=3em, column sep=3em, text height=1.5ex, text depth=0.25ex]
{ T^{\prime} & F(T^{\prime}) \\
  P & F(P)  \\ };
\path[->,font=\scriptsize]
(m-1-1) edge node[above] {$\nu_{T^{\prime}}$} (m-1-2)
(m-2-1) edge node[left] {$p$}(m-1-1)
(m-2-1) edge node[above] {$\nu_{P}$} (m-2-2)
(m-2-2) edge node[right] {$F(p)$} (m-1-2)
;
\end{tikzpicture}
\end{center}
 Since we have assumed that $\nu_{P} = 0$, we have that$\nu_{T^{\prime}} \circ p = 0$. Therefore, $\nu_{T^{\prime}}$ is $T$ ghost.
\end{proof}

The following lemma is a standard tool in the study of generation time.  It seems to appear independently multiple times in the literature.  The authors first learned it from \cite{Ro2}, however, it goes back, at least, to \cite{Kelly}.

\begin{lemma} [Ghost Lemma] \label{lem: category ghost lemma}
Let $\mathcal T$ be a triangulated category and $\mathcal I$ be a full subcategory.  If there exists a sequence,
\[
\begin{CD}
X_0      @>f_1>> X_1    @>f_2>> \cdots @>f_{n-1}>> X_{n-1} @>f_n>> X_n,   \\
\end{CD}
\]
of maps in $\mathcal T$ such that all the $f_i$ are ghosts for $\mathcal I$ and $f_n \circ \cdots \circ f_1 \neq 0$.  Then,
\[
X_0 \notin \langle \mathcal I \rangle_{n-1}.
\]
In particular,
\[
n-1 \leq \tritime(\mathcal I).
\]
\end{lemma}
\begin{proof}
See, for example, \cite[Lemma 4.11]{Ro2}.
\end{proof}

When $\mathcal I = G$ is a single object and $\mathcal T$ is a reasonably nice triangulated category there is also a converse.

\begin{lemma} [Ghost/Co-ghost Lemma and Converse]  \label{lem: ghost lemma}
Let $\mathcal S$ be a triangulated category for which small coproducts exist.  Let $\mathcal T$ be a full triangulated subcategory of the category of compact objects of $\mathcal S$.   Let $G$ and $X_0$ be objects in $\mathcal T$.  The following are equivalent:
\renewcommand{\labelenumi}{\emph{\roman{enumi})}}
\begin{enumerate}
\item one has $X_0 \in \langle G \rangle_n$ and $X_0 \notin \langle G \rangle_{n-1}$;

\item there exists a sequence, 
\begin{displaymath}
\begin{CD}
X_0      @>f_1>> X_1    @>f_2>> \cdots @>f_{n-1}>> X_{n-1} @>f_n>> X_n,   \\
\end{CD}
\end{displaymath}
of maps in $\mathcal T$ such that all the $f_i$ are ghost for $G$ and $f_n \circ \cdots \circ f_1 \neq 0$.  Furthermore there is no such sequence for $n+1$.

\item there exists a sequence, 
\begin{displaymath}
\begin{CD}
X_n      @>f_{n}>> X_{n-1}    @>f_{n-1}>> \cdots @>f_2>> X_{n-1} @>f_1>> X_0,   \\
\end{CD}
\end{displaymath}
of maps in $\mathcal T$ such that all the $f_i$ are co-ghost for $G$ and $f_1 \circ \cdots \circ f_n \neq 0$.  Furthermore there is no such sequence for $n+1$.

\item there exists a sequence, 
\begin{displaymath}
\begin{CD}
X_0      @>f_1>> X_1    @>f_2>> \cdots @>f_{n-1}>> X_{n-1} @>f_n>> X_n,   \\
\end{CD}
\end{displaymath}
of maps in $\mathcal S$ such that all the $f_i$ are ghost for $G$ and $f_n \circ \cdots \circ f_1 \neq 0$.  Furthermore there is no such sequence for $n+1$.

\item there exists a sequence, 
\begin{displaymath}
\begin{CD}
X_n      @>f_{n}>> X_{n-1}    @>f_{n-1}>> \cdots @>f_2>> X_{n-1} @>f_1>> X_0,   \\
\end{CD}
\end{displaymath}
of maps in $\mathcal S$ such that all the $f_i$ are co-ghost for $G$ and $f_1 \circ \cdots \circ f_n \neq 0$.  Furthermore there is no such sequence for $n+1$.
\end{enumerate}
\end{lemma}

\begin{proof}
See \cite{Opp2} or \cite{SeiBook} for the converse.  A simple proof in the case where $\mathcal T$ is $k$-linear and has finite dimensional morphism spaces  can be found in \cite[Lemma 2.17]{BFK}.
\end{proof}

In \cite{BFK}, a more in-depth analysis of the relationship between generation times and semi-orthogonal decompositions is provided. Here, however, we will only need some basic results.

\begin{lemma} \label{lem:decompositions and generation time}
 Let $\mathcal T = \mathcal \langle \mathcal S_1, \ldots , \mathcal S_s \rangle$ be a semi-orthogonal decomposition.  For each $i$, let $G_i$ be a strong generator of $\mathcal S_i$. Then $\bigoplus_{i=1}^{s} G_i$ is a strong generator of $\mathcal T$ and
\[
\tritime_{\mathcal T}(\bigoplus_{i=1}^{s} G_i) \leq \sum_{i=1}^s \tritime_{\mathcal S_i}(G_i) + s-1.
\]
  Furthermore if $\mathcal S$ is an admissible subcategory of $\mathcal T$ then,
\[
 \op{rdim}(\mathcal S) \leq \op{rdim}(\mathcal T).
\]
\end{lemma}

\begin{proof}
 This follows straightforwardly from the definitions. For a full set of details, see \cite{BFK}.
\end{proof}

\begin{proposition} \label{prop: orbit dimension}
 Let $\mathcal S$ be a triangulated category for which small coproducts exist.  Let $\mathcal T$ be a full triangulated subcategory of the category of compact objects of $\mathcal S$.  Let $\Gamma$ be a finite Abelian group of autoequivalences of $\mathcal T$.  Suppose $\pi: \mathcal T \to \mathcal T^{\prime}$ is a $\Gamma$-cover.

 If $\pi(G)$ is a generator of $\mathcal T^{\prime}$ then $\bigoplus_{\gamma \in \Gamma} \gamma(G)$ is a generator of $\mathcal T$.  Furthermore, if $\pi(G)$ is a strong generator then so is $\bigoplus_{\gamma \in \Gamma} \gamma(G)$ and one has,
\[
\tritime_{\mathcal T}(\bigoplus_{\gamma \in \Gamma} \gamma(G)) = \tritime_{\mathcal T^{\prime}}(\pi(G)).
\]
 In fact, for any object $X \in \mathcal T$,
\[
\pi(X) \in \langle \pi(G) \rangle_s \text{ if and only if } X \in \langle \bigoplus_{\gamma \in \Gamma} \gamma(G)  \rangle_s.
\]
\end{proposition}

\begin{proof}
 Suppose that $X \in \langle \bigoplus_{\gamma \in \Gamma} \gamma(G)\rangle_n$. We have
 \begin{displaymath}
  \langle \pi \left( \bigoplus_{\gamma \in \Gamma} \gamma(G) \right) \rangle_0 = \langle \pi(G) \rangle_0
 \end{displaymath}
 since $\pi(\gamma(G)) \cong \pi(G)$ for all $\gamma \in \Gamma$. Thus, $\pi(X) \in \langle \pi(G) \rangle_n$.

 Suppose $X \notin \langle \bigoplus_{\gamma \in \Gamma} \gamma(G)\rangle_n$. By Lemma~\ref{lem: ghost lemma}, there exists a nonzero co-ghost sequence for $\bigoplus_{\gamma \in \Gamma} \gamma(G)$,
\begin{displaymath}
 X_1 \to \cdots  \to X_{n+1}=X.
\end{displaymath}
 By applying $\pi$, we obtain a sequence of morphisms in the orbit category,
\begin{displaymath}
 \pi(X_1) \to \cdots  \to \pi(X_{n+1})=\pi(X).
\end{displaymath}
 The total composition is nonzero as $\pi$ is faithful. Furthermore, all maps in the sequence are co-ghost for $\pi(G)$. This follows from the functorial isomorphism,
 \begin{displaymath}
  \op{Hom}_{\mathcal T^{\prime}}(\pi(X_i), \pi(G)) = \op{Hom}_{\mathcal T}(X_i, \bigoplus_{\gamma \in \Gamma}\gamma(G)). 
 \end{displaymath}
 Hence by Lemma~\ref{lem: category ghost lemma}, $\pi(X) \notin \langle \pi(G) \rangle_n$.  This proves the final statement of the proposition.

 The final statement of the proposition implies that
 \begin{displaymath}
  \tritime_{\mathcal T^{\prime}}(\pi(G)) \geq \tritime_{\mathcal T}(\bigoplus_{\gamma \in \Gamma} \gamma(G)).
 \end{displaymath}
 Since we have assumed that every object is a summand of the essential image of $\mathcal T$, we have
 \begin{displaymath}
  \langle \pi(\mathcal T) \rangle_0 = \mathcal T^{\prime}.
 \end{displaymath}
 Thus, if $\langle \bigoplus_{\gamma \in \Gamma} \gamma(G) \rangle_n = \mathcal T$, we have
 \begin{displaymath}
  \mathcal T^{\prime} = \langle \pi(\mathcal T) \rangle_0 = \langle \oplus_{\gamma \in \Gamma} \pi(\gamma(G)) \rangle_n = \langle \pi(G) \rangle_n
 \end{displaymath}
 since $\pi(\gamma(G)) \cong \pi(G)$ in $\mathcal T^{\prime}$. Consequently,
 \begin{displaymath}
  \tritime_{\mathcal T}(\bigoplus_{\gamma \in \Gamma} \gamma(G)) \geq \tritime_{\mathcal T^{\prime}}(\pi(G)).
 \end{displaymath}
 This proves the second statement.

 The first statement is consequence of the previous arguments.
\end{proof}

\begin{remark}
 We will use Proposition~\ref{prop: orbit dimension} above to compare generation time in categories of matrix factorizations under grading change.  Rather than quotient by a finite torsion subgroup, one can also quotient by a divisor of the degree of the function, $w$.  In this case, while $\Gamma$ is not finite, the quotient of $\Gamma$ by its intersection with the powers of the shift functor, $\bar{\Gamma}$, is.  One can then extend Proposition~\ref{prop: orbit dimension} to this context by comparing $\tritime(\pi(G))$ with $\bigoplus_{\gamma \in \bar{\Gamma}}(\gamma(G))$.  Recall that, in general, the orbit category will not be triangulated.  There is a wall of rectifying this by discussing dg-orbit categories as in \cite{Kel2}.  However, one only obtains an inequality in this case and, in what follows, we will only need the case where $\Gamma$ is finite.
\end{remark}

\begin{remark}
 One could also assume $\Gamma$ is infinite but that for any object $X \in \mathcal T$, $\op{Hom}_{\mathcal T}(\gamma(G), X) = 0$ for all but finitely many $\gamma \in \Gamma$.  In this case the full subcategory consisting of all objects $\gamma(G)$, for all $\gamma \in \Gamma$ still generates $\mathcal T$ in equal amount of time to $G$ but can't be summed to form an object of the category.
\end{remark}

While Orlov spectra may change under $\Gamma$-coverings, the Rouquier dimension does not.

\begin{corollary} \label{cor: orbit dimension}
 Let $\mathcal S$ be a triangulated category for which small coproducts exist. Let $\mathcal T$ be a full triangulated subcategory of the category of compact objects of $\mathcal S$. Let $\Gamma$ be a finite group. Suppose that $\mathcal T^{\prime}$ is a $\Gamma$-cover of $\mathcal T$. Then $\mathcal T$ and $\mathcal T^{\prime}$ have the same Rouquier dimension,
 \[
  \op{rdim }\mathcal T = \op{rdim } \mathcal T^{\prime}.
 \]
\end{corollary}

\begin{proof}
Let $G$ be a generator of $\mathcal T$ such that,
\[
\tritime(G) = \op{rdim }\mathcal T.
\]
Therefore by Proposition~\ref{prop: orbit dimension} and the fact that $G \in \langle \bigoplus_{\gamma \in \Gamma} \gamma(G) \rangle_0$, we have
\[
 \op{rdim }\mathcal T^{\prime} \leq \tritime(\pi(G)) = \tritime (\bigoplus_{\gamma \in \Gamma} \gamma(G)) \leq \tritime(G) = \op{rdim }\mathcal T.
\]

Conversely, let $G^{\prime}$ be a generator of $\mathcal T^{\prime}$ such that,
\[
\tritime(G^{\prime}) = \op{rdim }\mathcal T^{\prime}.
\]
Now, $G^{\prime} \in \mathcal T^{\prime}$ is a summand of $\pi(G)$ for some generator $G \in \mathcal T$ (we may always make $G$ into a generator by adding additional summands).
Therefore by Proposition~\ref{prop: orbit dimension}
\[
\op{rdim }\mathcal T^{\prime} = \tritime(G^{\prime}) \geq \tritime(\pi(G)) = \tritime (\bigoplus_{\gamma \in \Gamma} \gamma(G)) \geq \op{rdim }\mathcal T.
\]
\end{proof}

\begin{corollary} \label{cor: equality of curves}
Let $G$ be a finite Abelian group acting on a a smooth Deligne-Mumford stack $\mathcal X$ over $k$. One has
\[
 \op{rdim }\dbcoh{\mathcal X} =  \op{rdim }\dbcoh{[\mathcal X/G]}.
\]

\end{corollary}

\begin{corollary} \label{cor: compare rdim}
 Let $\pi: A \to \Z$ and $\pi^{\prime}: A^{\prime} \to \Z$ be surjective maps of Abelian groups with finite kernel. Let $R$ be a regular finitely-generated $k$-algebra admitting $A$ and $A^{\prime}$-gradings each of which induce the same $\Z$-grading under $\pi$ and $\pi^{\prime}$.  Suppose $w \in R$ is homogeneous with respect to all gradings.  Then we have an equality of Rouquier dimensions
\[
 \op{rdim }\dabs[\mathsf{fact}(R,A,w)] = \op{rdim }\dabs[\mathsf{fact}(R,A^{\prime},w)].
\]
\end{corollary}

\begin{proof}
 This follows directly from Corollary~\ref{cor: orbit dimension} and Proposition~\ref{prop: MF grading change}. 
\end{proof}

\subsection{Rouquier dimension of weighted Fermat hypersurfaces} \label{subsec: Fermat}

In this section, we would like to apply our results to estimate the Rouquier dimension of products of quivers of type ADE and their corresponding weighted Fermat hypersurfaces.  An upper bound will come from the interpretation of this category as an algebra.

\begin{definition} \label{defn: LL}
 Let $A$ be a finite dimensional (possibly noncommutative) algebra over $k$. Let $N$ be the nilradical of $A$. The \newterm{Loewy length} of $A$ is
\begin{displaymath}
 \op{LL}(A) := \op{min} \{ t \mid N^t = 0 \}.
\end{displaymath}
\end{definition}

We will need the following lemma due to Rouquier.

\begin{lemma}\label{lem: ADE-1}
 Let $A$ be a finite dimensional algebra of Loewy length $r$ and $Q$ be a quiver whose underlying graph is Dynkin of type ADE. The Rouquier dimension of $\op{D}^{\op{b}}(\op{mod }A \otimes_k kQ)$ is at most $r-1$.
\end{lemma}

\begin{proof}
 Let $N$ be the nilradical of $A$ and set $I = N \otimes_k kQ$ in \cite[Lemma 7.36]{Ro2}.  Furthermore, notice that $A \otimes_k kQ/I$ is isomorphic to a finite direct sum of algebras $kQ$, hence $\text{D}^\text{b}(\op{mod }A \otimes_k kQ/I)$ has Rouquier dimension zero. 
\end{proof}

For a lower bound we will need the following statement.

\begin{lemma} \label{lem:lowerstacks}
 Let $\mathcal X$ be a tame Deligne-Mumford stack with a reduced and separated coarse moduli space. The Rouquier dimension of $\dbcoh{\mathcal X}$ is at least the dimension of $\mathcal X$,
\[
 \op{dim }\mathcal X \leq \op{rdim }\dbcoh{\mathcal X}.
\]
\end{lemma}

\begin{proof}
 This is a slight generalization of a result of Rouquier. The proof can be found in \cite[Lemma 2.13]{BF}.
\end{proof}

Recall, from Definition~\ref{definition: sign of weight sequence}, that a weight sequence, $(d_0, \ldots , d_n)$, is negative, zero, positive, nonnegative, or nonpositive if
\begin{displaymath}
 -1 + \sum_{i=0}^n \frac{1}{d_i}
\end{displaymath}
is negative, zero, positive, nonnegative, or nonpositive, respectively. And, from Definition~\ref{definition: ADE for weight sequence}, a weight sequence is ADE if it coincides with any of the following weight sequences
\[
 (2,\ldots,2,a), (2,\ldots,2,3,3), (2,\ldots,2,3,4), \tor (2,\ldots,2,3,5)
\]
where $2$ can occur with any multiplicity including zero.

Given a weight sequence, $\mathbf{d} =(d_0, \ldots, d_n)$, set
\begin{displaymath}
 q(\mathbf{d}) := \op{min} \{s \mid \mathbf{ d} = \coprod_{i=1}^s \mathbf{d}_i \text{ where } \mathbf{d}_i \text{ is nonpositive for all } i\},
\end{displaymath}
and
\begin{displaymath}
 h(\mathbf{d}) := \op{min} \{s \mid \mathbf{ d} = \coprod_{i=1}^s \mathbf{d}_i \text{ where } \mathbf{d}_i \text{ is ADE for all } i\}.
\end{displaymath}
Note that $h(\mathbf{d}),q(\mathbf{d}) \leq n+1$.

\begin{theorem} \label{thm: Fermat bound}
 Let $\mathbf{d}=(d_0, \ldots d_n)$ be a weight sequence. Suppose $\mathbf{d}$ decomposes into a disjoint union of nonpositive weight sequences,
\[
 \mathbf{d} = \coprod_{i=1}^s \mathbf{d}_i.
\]
 One has an inequality,
\[
 n+1-2s \leq \op{rdim}(\dbcoh{\prod_{i=1}^s Z_{\mathbf{d}_i}}) \leq \op{rdim}(\op{D}^{\op{b}}(\op{mod }\bigotimes_{i=0}^n A_{d_i-1})) \leq h(\mathbf{d})-1.
\]
 Hence,
\[
 n+1-2q(\mathbf{d}) \leq \op{rdim}(\op{D}^{\op{b}}(\op{mod }\bigotimes_{i=0}^n A_{d_i-1})) \leq h(\mathbf{d})-1
\]
\end{theorem}

\begin{proof}
 We start with the upper bound. By Proposition~\ref{prop: is ADE}, $\op{D}^{\op{b}}(\op{mod }\bigotimes_{i=0}^n A_{d_i-1})$ is equivalent to $\op{D}^{\op{b}}(\op{mod }\bigotimes_{l=1}^{h(\mathbf{d})} kQ_l)$ where each $Q_l$ is a quiver whose underlying graph is of type ADE. Now, we may assume that each of the quivers, $Q_i$, are as displayed in Figure~\ref{fig: ADE}. Let $S^i_v$ denote the simple left $kQ_i$-module corresponding to the vertex, $v$. The $S^i_v$ for all $v$ in the quiver give representatives of all the isomorphism classes of simple objects. A straightforward computation shows that,
\[
\op{Ext}^t(S^i_v, S^i_{v^{\prime}}) \cong \begin{cases} k & \text{ if } t=0, v=v^{\prime} \\ k & \text{ if } t=1 \text{ and there exists an arrow starting at }v^{\prime}\text{ and ending at }v \\
					   0 & \text{otherwise}.
                             \end{cases}
\]
 Hence, for any $i$, the simple modules in $\op{mod }(kQ_i)$, after appropriate shifting, form a full strong exceptional collection whose endomorphism algebra has Loewy length $2$, as
\begin{displaymath}
 \op{Ext}^2_{kQ_i}(S^i_v, S^i_{v^{\prime}}) = 0
\end{displaymath}
 for all $v,v^{\prime}$.  Tensoring together, after shifting, the simple modules of $\op{mod }(\bigotimes_{l=1}^{h(\mathbf{d})-1} kQ_l)$ form a full strong exceptional collection whose Loewy length is $h(\mathbf{d})$.  Let $R$ be the endomorphism algebra of this full strong exceptional collection in $\op{mod }\bigotimes_{l=1}^{h(\mathbf{d})-1} kQ_l$. Let $P^{h(\mathbf{d})}_v$ be the projective module corresponding the vertex $v$ in $kQ_{h(\mathbf{d})}$. Obviously, the $P^{h(\mathbf{d})}_v$ for all $v$ in the quiver for $kQ_{h(\mathbf{d})}$ form a full strong exceptional collection for $\op{D}^{\op{b}}(\op{mod } kQ_{h(\mathbf{d})})$ (whose endomorphism algebra returns the path algebra of $kQ_{h(\mathbf{d})}$).  Therefore, appropriate shifts of the objects, $(\bigotimes_{i=0}^{h(\mathbf{d})-1} S^i_{v_i}) \otimes_k P^{h(\mathbf{d})}_{v^{\prime}}$ for any vertices, $v_i \in Q_i$ $v^{\prime} \in Q_{h(\mathbf{d})}$
form a full strong exceptional collection in $\op{D}^{\op{b}}(\op{mod }\bigotimes_{l=1}^{h(\mathbf{d})}
kQ_l)$ with endomorphism algebra, $R \otimes_k kQ_{h(\mathbf{d})}$.  This induces an equivalence,
\begin{displaymath}
 \op{D}^\text{b}(\op{mod }\bigotimes_{l=1}^{h(\mathbf{d})} kQ_l) \cong \text{D}^\text{b}(\op{mod }R \otimes_k kQ_{h(\mathbf{d})}),
\end{displaymath}
 by \cite{Bae, Bon}.  Applying Lemma~\ref{lem: ADE-1} we arrive at the upper bound.

 For the lower bound, suppose the weight sequence, $\mathbf{d} = (d_0, \ldots , d_n)$, decomposes into a disjoint union of nonpositive weight sequences,
\[
 \mathbf{d} = \coprod_{i=1}^s \mathbf{d}_i.
\]
 Then, by Proposition~\ref{prop: sitting inside ADE}, the category, $\dbcoh{\prod_{i=1}^s Z_{\mathbf{d}_i}}$, is an admissible subcategory of $\op{D}^{\op{b}}(\op{mod }\bigotimes_{i=0}^n A_{d_i-1})$. Hence,
\begin{equation} \label{eq: inequality for ADE}
 n+1-2s = \op{dim}(\prod_{i=1}^s Z_{\mathbf{d}_i}) \leq \op{rdim}(\dbcoh{\prod_{i=1}^s Z_{\mathbf{d}_i}}) \leq \op{rdim}(\op{D}^{\op{b}}(\op{mod }\bigotimes_{i=0}^n A_{d_i-1})),
\end{equation}
 where the first inequality comes from Lemma~\ref{lem:lowerstacks} and the second inequality comes from Lemma~\ref{lem:decompositions and generation time}.
\end{proof}

Recall, from Definitions~\ref{defn: weight sequence} and \ref{def: Fermat}, that to each weight sequence, $(d_0, \ldots ,d_n)$ we attach a finitely-generated Abelian group, $B_{\mathbf{d}}$, and two Abelian algebraic groups,  $\smallfermatgroup$ and $G(B_{\mathbf{d}})$.  To these algebraic groups we associate two stacks, $  Z_{({\mathbf{d}}, \op{min})}$ and $Z_{({\mathbf{d}})}$ respectively.
Furthermore, to any intermediate group,
$\smallfermatgroup \subseteq \intermediatefermatgroup \subseteq G(B_{\mathbf{d}}),$
we attach a weighted Fermat hypersurface $Z_{({\mathbf{d}}, \intermediatefermatgroup)}$.

\begin{corollary} \label{cor: Orlov conjecture fermat}
 Let $(d_0, \ldots , d_n)$ be a weight sequence such that
\[
 n + 1 = h(\mathbf{d}) + 2q(\mathbf{d}) - 1.
\]
 Let,
\[
\mathbf{d} = \coprod_{i=1}^{q(\mathbf{d})} \mathbf{d}_i,
\]
 be a decomposition of $\mathbf{d}$ into a disjoint union of nonpositive weight sequences.  For each weight sequence, $\mathbf{d}_i$, choose a group,
\[
 \bar{G}(B_{\mathbf{d}_i}) \subseteq H_i \subseteq G(B_{\mathbf{d}_i})
\]
 Then,
\[
 n+1-2q(\mathbf{d})  = \op{rdim}(\dbcoh{\prod_{i=1}^{q(\mathbf{d})} Z_{\mathbf{d}_i,H_i}}) = \op{rdim}(\op{D}^{\op{b}}(\op{mod }\bigotimes_{i=0}^n A_{d_i-1})).
\]
 In particular, Conjecture~\ref{conj: dimension} holds for the product of weighted Fermat hypersurfaces, $Z_{(\mathbf{d}_i, H_i)}$.
\end{corollary}

\begin{proof}
 Notice that $G(B_{\mathbf{d}_i})/\intermediatefermatgroup_i$ is a finite group for all $i$ which is a summand of $\intermediatefermatgroup_i$. Applying Proposition~\ref{prop: spaces as covers}, one deduces that $\dbcoh{Z_{(\mathbf{d}_i, \intermediatefermatgroup_i)}}$ is a $\widehat{G(B_{\mathbf{d}_i})/\intermediatefermatgroup_i}$-cover of $\dbcoh{Z_{\mathbf{d}_i}}$ for any $\intermediatefermatgroup_i$ as above. Similarly, $\dbcoh{\prod_{i=1}^{q(\mathbf{d})}  Z_{\mathbf{d}_i,H_i}}$ is a $\Gamma$-cover of $\dbcoh{\prod_{i=1}^{q(\mathbf{d})} Z_{\mathbf{d}_i}}$ for finite $\Gamma$.

 Therefore, by Corollary~\ref{cor: orbit dimension},
\[
 \op{rdim} \dbcoh{\prod_{i=1}^{q(\mathbf{d})} Z_{\mathbf{d}_i,H_i}} = \op{rdim} \dbcoh{\prod_{i=1}^{q(\mathbf{d})}  Z_{\mathbf{d}_i}}.
\]
 Hence, it suffices to prove the proposition for $H_i = G(B_{\mathbf{d}_i})$, i.e.~just for the maximally weighted Fermat hypersurfaces, $Z_{\mathbf{d}_i}$.

 Now when $n+1 = h(\mathbf{d}) + 2q(\mathbf{d}) -1$, the lower bound and upper bound in Theorem~\ref{thm: Fermat bound} agree. Hence,
\[
 \op{dim}(\prod_{i=1}^{q(\mathbf{d})}  Z_{\mathbf{d}_i}) = n+1-2q(\mathbf{d}) = \op{rdim}(\dbcoh{\prod_{i=1}^s  Z_{\mathbf{d}_i}}).
\]
\end{proof}

\begin{example}
 Consider the weight sequence $\mathbf{d} =(3,3,3)$.  We have $h(\mathbf{d}) = 2$ from the ADE decomposition,
 \begin{displaymath}
  \mathbf{d} = (3) \coprod (3,3).
 \end{displaymath}
 The trivial decomposition is nonpositive, so $q(\mathbf{d}) = 1$. Hence the conditions of Corollary~\ref{cor: Orlov conjecture fermat} and Conjecture~\ref{conj: dimension} hold for the Fermat elliptic curve and similarly for the hypersurface weighted projective line corresponding to this weight sequence.
 \end{example}

 \begin{remark}
Conjecture~\ref{conj: dimension}  holds for smooth curves in general, see \cite{O4}, and for tubular weighted projective lines, i.e.~quotients of elliptic curves by finite groups of automorphisms which preserve the identity, \cite{Opp}.  Note that the latter also follows from the former using Corollary~\ref{cor: equality of curves}.
 \end{remark}


\begin{example} \label{eg: ExExK3}
 Extending the previous example, we have the following. Consider the weight sequence $\mathbf{d} = (3,3,3,3,3,3,4,4,4,4)$.  We have $h(\mathbf{d}) = 5$ from the ADE decomposition,
 \begin{displaymath}
  \mathbf{d} = (3,4) \coprod (3,4) \coprod (3,4) \coprod (3,4) \coprod (3,3),
 \end{displaymath}
 and $q(\mathbf{d}) = 3$ from
 \begin{displaymath}
  \mathbf{d} = (3,3,3) \coprod (3,3,3) \coprod (4,4,4,4).
 \end{displaymath}
 Hence, the conditions of Corollary~\ref{cor: Orlov conjecture fermat} hold. Consider the (usual) Fermat elliptic curve, $Z_{((3,3,3), \op{min})}$, and the (usual) Fermat K3 surface, $Z_{((4,4,4,4), \op{min})}$. Hence, Conjecture~\ref{conj: dimension} holds for $Z_{((3,3,3), \op{min})} \times Z_{((3,3,3), \op{min})} \times Z_{((4,4,4,4), \op{min})}$.
\end{example}

\begin{example}
 Let $\mathbf{d} = (2,d_1, \ldots , d_n)$ be a nonpositive weight sequence so that $q(\mathbf{d})=1$.  Notice that $h(\mathbf{d}) = n$ as we can decompose $\mathbf{d}$ into ADE sequences,
 \[
 \mathbf{d} = (2,d_1) \coprod (d_2) \cdots \coprod (d_n).
 \]
 and if there were any more ADE sequences then $\mathbf{d}$ would be negative as the value of
\[
 -1 + \sum_{i=0}^n \frac{1}{d_i^{\prime}} > \frac{1}{2}
 \]
 for any ADE sequence, $(d_0^\prime, ..., d_n^\prime)$ with $n > 0$.
 Hence,
\[
 n-1  = \op{rdim}(Z_{(\mathbf{d}, H)}) = \op{rdim}(\op{D}^{\op{b}}(\op{mod }\bigotimes_{i=1}^n A_{d_i-1})).
\]
 In particular, Conjecture~\ref{conj: dimension} holds for the corresponding weighted Fermat hypersurface, $Z_{(\mathbf{d}, \intermediatefermatgroup)}$ for any $\smallfermatgroup \subseteq \intermediatefermatgroup \subseteq G(B_{\mathbf{d}})$.
\end{example}

\begin{remark}
 Truncate the weight sequence in the previous example to $\mathbf{d}^{\prime}=(d_1, \ldots , d_n)$  with $\sum_{i=1}^n \frac{1}{d_i} \leq \frac{1}{2}$. Then the comparison between the quantities, $\op{rdim }\dbcoh{Z_{(\mathbf{d}^{\prime}, \intermediatefermatgroup)}}$ and $\op{rdim }\op{D}^{\op{b}}(\op{mod }\bigotimes_{i=1}^n A_{d_i-1})$, is strict if Conjecture~\ref{conj: dimension} holds.  As in this case,
\[
 n-2 = \op{dim}(Z_{(\mathbf{d}^{\prime}, \intermediatefermatgroup)}) =  \op{rdim } \dbcoh{Z_{(\mathbf{d}^{\prime}, L)}} \lneq \op{rdim }\op{D}^{\op{b}}(\op{mod }\bigotimes_{i=1}^n A_{d_i-1}) = n-1
\]
 where the second equality uses the example above.

 For example, as mentioned above, we know that the Rouquier dimension of the bounded derived category of coherent sheaves on any smooth curve is one, \cite{O4}, however, the above argument shows that the category of matrix factorizations of a Fermat curve of sufficiently high genus has Rouquier dimension two.
\end{remark}

This is not an exhaustive list.  By applying Corollary~\ref{cor: Orlov conjecture fermat}, one can obtain many other examples which provide new valid cases of Conjecture~\ref{conj: dimension}.

\end{document}